\documentclass[11pt]{article}
\usepackage{LFC}

\title{Lexicographic functional calculus and its \\
application to functional calculus calculus}
\author{Evangelos A.\ Nikitopoulos}
\affil{Department of Mathematics, University of Michigan\protect\\
\noindent 530 Church Street, Ann Arbor, MI 48109-1043 (USA)\protect\\
{Email: \texttt{\href{mailto:enikitop@umich.edu}{enikitop@umich.edu}}}}
\date{\vspace{-5ex}}

\begin{document}

\maketitle

\begin{abstract}
Let $\mathcal{A}$ be a unital $C^*$-algebra and $\mathcal{I}$ be a symmetrically normed ideal of $\mathcal{A}$.
I introduce and study lexicographic functional calculus (LFC), a kind of multivariate functional calculus for tuples $(a_1,\ldots,a_m)$ of noncommuting self-adjoint elements of $\mathcal{A}$ ``acting in lexicographic order,'' i.e., from left to right, with an element $b_i \in \mathcal{I}$ ``inserted'' between the action of $a_i$ and $a_{i+1}$ for each $i=1,\ldots,m-1$.
The reason for its introduction is an application to ``functional calculus calculus,'' the differential calculus of maps induced by single-variate (continuous) functional calculus.
Specifically, I prove that if $f \colon \mathbb{R} \to \mathbb{C}$ is sufficiently regular and $a \in \mathcal{A}_{\mathrm{sa}} \coloneqq \{c \in \mathcal{A} : c^*=c\}$, then $f_{a,\mathsmaller{\mathcal{I}}}(b) \coloneqq f(a+b)-f(a) \in \mathcal{I}$ for all $b \in \mathcal{I}_{\mathrm{sa}} \coloneqq \mathcal{I} \cap \mathcal{A}_{\mathrm{sa}}$, the map $f_{a,\mathsmaller{\mathcal{I}}} \colon \mathcal{I}_{\mathrm{sa}} \to \mathcal{I}$ is Fr\'echet $C^k$, and the $k^{\text{th}}$ Fr\'echet derivative of $f_{a,\mathsmaller{\mathcal{I}}}$ may be written in terms of LFC applied to the $k^{\text{th}}$ divided difference of $f$, a function of $k+1$ variables.
This result recovers or vastly generalizes nearly all comparable results in the literature.
For example, it simultaneously recovers the following three highly related results on the regularity of the function $f_{\mathsmaller{\mathcal{A}}} \colon \mathcal{A}_{\mathrm{sa}} \to \mathcal{A}$ defined by $a \mapsto f(a)$:
(1) If $\mathcal{A}$ is commutative and $f \in C^k(\mathbb{R})$, then $f_{\mathsmaller{\mathcal{A}}}$ is Fr\'echet $C^k$;
(2) if $\mathcal{A}$ is finite dimensional and $f \in C^k(\mathbb{R})$, then $f_{\mathsmaller{\mathcal{A}}}$ is Fr\'echet $C^k$; and
(3) if $f \colon \mathbb{R} \to \mathbb{C}$ is ``slightly better than $C^k$,'' e.g., belongs to the homogeneous Besov space $\dot{B}_1^{k,\infty}(\mathbb{R})$, then $f_{\mathsmaller{\mathcal{A}}}$ is Fr\'echet $C^k$ no matter the choice of $\mathcal{A}$.
Before LFC, there was no unifying framework for results (1)--(3);
in particular, there was no single result of which they were all corollaries.

$\,$

\noindent \textbf{Keywords:}
functional calculus, lexicographic functional calculus, multivariate functional calculus, $C^*$-algebra, symmetrically normed ideal, higher Fr\'echet derivatives

$\,$

\noindent \textbf{MSC (2020):} 47A60, 47A13, 26E15, 47L20 (Primary); 46E10, 46E35 (Secondary)
\end{abstract}

\tableofcontents

\section{Introduction}\label{sec.intro}

Let $\cA$ be a unital $C^*$-algebra and $\cA_{\sa} \coloneqq \{a \in \cA : a^*=a\}$ be the set of self-adjoint elements of $\cA$.
Recall that if $a \in \cA_{\sa}$ and $\sigma(a)$ is the spectrum of $a$, then the \textbf{continuous functional calculus for $\boldsymbol{a}$} is the unique unital $\ast$-homomorphism $\Phi_a \colon C(\sigma(a)) \to \cA$ sending the inclusion $\iota_{\sigma(a)} \colon \sigma(a) \to \C$ to $a$.
Write $f(a) \coloneqq \Phi_a(f)$ for all $f \in C(\sigma(a))$.

Now, for each continuous function $f \colon \R \to \C$, write $f_{\scA} \colon \cA_{\sa} \to \cA$ for the map defined via the continuous functional calculus by
\[
f_{\scA}(a) \coloneqq f(a) = (f|_{\sigma(a)})(a) \in \cA \qquad (a \in \cA_{\sa}).
\]
It is elementary to show that the continuity of $f$ implies the continuity of $f_{\scA}$.
It is therefore natural to wonder, e.g., whether $f \in C^k(\R)$ implies $f_{\scA} \in C^k(\cA_{\sa};\cA)$ when $k \in \N$.
Interestingly, this is not generally true.
Take $\cA = B(H)$, where $H$ is an infinite-dimensional complex Hilbert space, in which case $f_{\scA} = f_{\mathsmaller{B(H)}}$ is called the \textbf{operator function induced by $\boldsymbol{f}$}.
By \cite[Thm.\ 1.2.9]{AP2016}, if $f \in C(\R)$ and $f_{\mathsmaller{B(H)}} \in C^1(B(H)_{\sa};B(H))$, then $f$ is locally operator Lipschitz, i.e., for each $r > 0$, $f_{\mathsmaller{B(H)}}|_{\{a \in B(H)_{\sa} : \norm{a} \leq r\}}$ is Lipschitz with respect to the operator norm $\norm{\cdot}$.
Yu.\ B.\ Farforovskaya showed in \cite{Farforovskaya1972,Farforovskaya1976} that there exist $C^1$ functions that are \textit{not} locally operator Lipschitz.
In particular, there exist functions $f \in C^1(\R)$ such that $f_{\mathsmaller{B(H)}} \not\in C^1(B(H)_{\sa};B(H))$.
Please see \cite{Peller1985} and \cite[\S1.2 \& \S1.5]{AP2016} for more information.

Despite the negative result described in the previous paragraph, there is a vast array of positive results on the regularity of operator functions, a few of which I quote below.
The key tool used most frequently to establish and apply such results is called a multiple operator integral (MOI).
Please see A.\ Skripka and A.\ Tomskova's book \cite{ST2019}, as well as the more recent papers \cite{LMS2020,LMM2021,Coine2022,Nikitopoulos2022,vNS2022,Nikitopoulos2023h,Nikitopoulos2023m,Nikitopoulos2023n,CCGP2024,CvNP2024,vNS2025,PS2025,FS2025,Furst2025,CCGP2026,JKN2026}, for information about MOIs and examples of their many applications.
Though MOIs form part of the inspiration for the present work, they do not actually play an important role in this paper's framework and main results.

The state of the art on the question of the smoothness of $f_{\scA}$ is that if $f \colon \R \to \C$ is ``slightly better than $C^k$,'' then $f_{\scA} \in C^k(\cA_{\sa};\cA)$ for all unital $C^*$-algebras $\cA$.
However, for particular choices of $\cA$, one can do better.

\begin{theorem}\label{thm.motivatingthm}
Let $\cA$ be a unital $C^*$-algebra and $f \colon \R \to \C$ be a continuous function.
\begin{enumerate}[font=\normalfont,label=(\roman*)]
    \item If $\cA$ is commutative and $f \in C^k(\R)$, then $f_{\scA} \in C^k(\cA_{\sa};\cA)$.\label{item.commutative}
    \item {\rm (Daletskii--Krein \cite{DK1956})} If $\cA$ is finite dimensional and $f \in C^k(\R)$, then $f_{\scA} \in C^k(\cA_{\sa};\cA)$.\label{item.fd}
    \item {\rm (Peller \cite{Peller2006}, N.\ \cite{Nikitopoulos2023n})} If $f$ belongs to the homogeneous Besov space $\dot{B}_1^{k,\infty}(\R)$ (Definition \ref{def.Besov} below) or the space $C_{\loc}^{k,\e}(\R)$ of $C^k$ functions with locally $\e$-H\"older-continuous $k^{\text{th}}$ derivatives ($\e > 0$), then $f_{\scA} \in C^k(\cA_{\sa};\cA)$.\label{item.PellerN}
\end{enumerate}
\end{theorem}

As cited above, item \ref{item.fd} is due to Yu.\ L.\ Daletskii and S.\ G.\ Krein in 1956;
however, please see \cite{Hiai2010,Nikitopoulos2023n} for modern formulations and self-contained proofs.
No reference is given for item \ref{item.commutative} because it is on the level of an exercise to establish.

In each of the situations in Theorem \ref{thm.motivatingthm}, it is possible to write down an explicit formula for the $k$-fold directional derivative
\[
\partial_{b_k}\cdots\partial_{b_1}f_{\scA}(a) = \frac{\d}{\d s_k}\Big|_{s_k = 0}\cdots\frac{\d}{\d s_1}\Big|_{s_1 =0} f_{\scA}(a+s_1b_1+\cdots+s_kb_k).
\]
Interestingly, though the proofs of \ref{item.commutative}--\ref{item.PellerN} in Theorem \ref{thm.motivatingthm} and the aforementioned derivative formulas are strongly related to each other, there is no general result of which they are all corollaries, whence comes one of the main goals of this paper.

\begin{metatheorem}\label{metathm.meta}
There exists a kind of functional calculus for tuples of self-adjoint elements of unital $C^*$-algebras, called lexicographic functional calculus (LFC), that enables the formulation and proof of a general result of which Theorem \ref{thm.motivatingthm} is a corollary.
Moreover, the derivatives $\partial_{b_k}\cdots\partial_{b_1} f_{\scA}(a)$ may be written in terms of LFC.
\end{metatheorem}

The ``general result'' referenced above is Theorem \ref{thm.derform}, which actually answers the question of when the map
\[
\cI_{\sa} \coloneqq \cI \cap \cA_{\sa} \ni b \mapsto f_{a,\scI}(b) \coloneqq f(a+b) - f(a) \in \cI
\]
is well defined and $C^k$, where $\cI$ is a symmetrically normed ideal of $\cA$ (Definition \ref{def.sni} below) and $a \in \cA_{\sa}$ is fixed.
For the sake of exposition, I postpone the discussion of this more general setup to the body of the paper;
please see subsection \ref{subsec.derivatives}, which also discusses known results in this setting.

Subsection \ref{subsec.HFC} motivates LFC and its appearance in ``functional calculus calculus,'' the differential calculus of maps induced by functional calculus, by studying the holomorphic case, which does not suffer any of the pathologies of the real-$C^k$ case.
Then, in subsection~\ref{subsec.results}, I lay out the framework that converts Metatheorem \ref{metathm.meta} into a collection of rigorous theorems and, along the way, describe the structure of the remainder of the paper.

\subsection{Motivation from holomorphic functional calculus calculus}\label{subsec.HFC}

Let $\cB$ be a unital Banach algebra, $U \subseteq \C$ be an open set, and $\cB_U \subseteq \cB$ be the open set of elements $a \in \cB$ such that the spectrum $\sigma(a)$ of $a$ is contained in $U$.
Also, if $X$ and $Y$ are complex Banach spaces and $V \subseteq X$ is an open set, let $\Hol\left(V;Y\right)$ be the set of holomorphic functions from $V$ to $Y$, and write $\Hol\left(V\right) \coloneqq \Hol\left(V;\C\right)$.

Recall that if $a \in \cB_U$, then the (\textbf{Dunford--Riesz}) \textbf{holomorphic functional calculus for $\boldsymbol{a}$} is the unique continuous, unital algebra homomorphism $H_a^U \colon \Hol\left(U\right) \to \cB$ sending the inclusion $\iota_{\mathsmaller{U}} \colon U \to \C$ to $a$;
here and throughout, $\Hol\left(U\right)$ is given the topology of locally uniform convergence.
As the reader may review by consulting \cite[Ch.\ 10]{Rudin1991}, $H_a^U$ is constructed using a Cauchy integral--type formula:
\begin{equation}
    f(a) \coloneqq H_a^U(f) = \frac{1}{2\pi i}\int_{\Gamma} f(z)\,(z-a)^{-1}\,\d z \in \cB \qquad \left(f \in \Hol\left(U\right)\right)\label{eq.HFC}
\end{equation}
for any cycle $\Gamma$ surrounding $\sigma(a)$ in $U$.\footnote{A \textbf{cycle} in $U$ is a finite collection of closed, piecewise $C^1$ curves in $U$.
A cycle $\Gamma$ in $U$ \textbf{surrounds $\boldsymbol{S}$ in $\boldsymbol{U}$} if $W_{\Gamma}(z) = 0$ for all $z \in \C \setminus U$ and $W_{\Gamma}(z) = 1$ for all $z \in S \subseteq U$, where $W_{\Gamma}(z) \in \Z$ is the winding number of $\Gamma$ around $z$.}

Now, for $f \in \Hol\left(U\right)$, let $f_{\scB} \colon \cB_U \to \cB$ be the map $a \mapsto f(a)$ induced by this holomorphic functional calculus.
(I entreat the reader to ignore the fact that this conflicts with the notation in the rest of the paper, aside from the appendix.)
It is not so difficult to see that $f_{\scB} \colon \cB_U \to \cB$ is holomorphic by differentiating under the integral in \eqref{eq.HFC}.
Indeed, if $a \in \cB_U$ and $b \in \cB$, then
\begin{align*}
    \partial_bf_{\scB}(a) & = \frac{1}{2\pi i}\partial_b\int_{\Gamma} f(z)\,(z-a)^{-1}\,\d z \\
    & = \frac{1}{2\pi i}\int_{\Gamma} f(z)\,\partial_b(z-a)^{-1}\,\d z \\
    & = \frac{1}{2\pi i}\int_{\Gamma} f(z)\,(z-a)^{-1}b\,(z-a)^{-1}\,\d z.
\end{align*}
(The technical details are unimportant at this time.)
Repeatedly differentiating under the integral in this way yields the following result.

\begin{theorem}\label{thm.HFCintro}
If $f \in \Hol\left(U\right)$, then $f_{\scB} \in \Hol\left(\cB_U;\cB\right)$.
Furthermore, if $a \in \cB_U$ and $\Gamma$ is a cycle surrounding $\sigma(a)$ in $U$, then
\[
\partial_{b_k}\cdots\partial_{b_1}f_{\scB}(a) = \frac{1}{2\pi i}\sum_{\pi \in S_k}\int_{\Gamma} f(z) \, (z-a)^{-1}b_{\pi(1)}\cdots(z-a)^{-1}b_{\pi(k)}\,(z-a)^{-1} \,\d z
\]
for all $b_1,\ldots,b_k \in \cB$, where $S_k$ is the symmetric group on $k$ letters.
\end{theorem}

The $k^{\text{th}}$ derivative formula above is worth pondering.
Some additional notation, used throughout the paper, is helpful for this purpose.
First, write $\potimes$ for the Banach-space projective tensor product (vid.\ \cite[Ch.\ 2]{Ryan2002}).
Second, if $X_1,\ldots,X_k,X$ are normed vector spaces, write $B_k(X_1 \times \cdots \times X_k ; X)$ for the space of bounded $k$-linear maps from $X_1 \times \cdots \times X_k$ to $X$.
As usual, $B(X_1;X) \coloneqq B(X_1;X)$, $\norm{\cdot}_{X_1 \to X} \coloneqq \norm{\cdot}_{B(X_1;X)}$, and $B(X) \coloneqq B(X;X)$.
Finally, define
\[
\sh_k = \sh \colon \cB^{\potimes(k+1)} \to B_k\big(\cB^k;\cB\big)
\]
to be the bounded (complex-)linear map, written $u\sh_k b = u\sh b \coloneqq \sh(u)[b]$, determined by
\[
(a_1\otimes\cdots\otimes a_{k+1})\sh[b_1,\ldots,b_k] = a_1b_1\cdots a_kb_ka_{k+1}
\]
for all $a_1,b_1,\ldots,a_k,b_k,a_{k+1} \in \cB$.
\pagebreak

Now, let $a \in \cB_U$, and define
\begin{equation}
    \tilde{a}_i \coloneqq 1^{\otimes(i-1)} \otimes a \otimes 1^{\otimes(k+1-i)} \in \cB^{\potimes(k+1)} \qquad (i = 1,\ldots,k+1).\label{eq.atilde}
\end{equation}
In this notation, the formula in Theorem \ref{thm.HFCintro} says
\begin{align*}
    \partial_{b_k}\cdots\partial_{b_1}f_{\scB}(a) & = \sum_{\pi \in S_k}\left(\frac{1}{2\pi i}\int_{\Gamma} f(z) \, (z-a)^{-1}\otimes \cdots \otimes (z-a)^{-1} \,\d z\right)\sh\left[b_{\pi(1)},\ldots,b_{\pi(k)}\right] \\
    & = \sum_{\pi \in S_k}\left(\frac{1}{2\pi i}\int_{\Gamma} f(z) \, (z-\tilde{a}_1)^{-1} \cdots (z-\tilde{a}_{k+1})^{-1} \,\d z\right)\sh\left[b_{\pi(1)},\ldots,b_{\pi(k)}\right]
\end{align*}
for all $b_1,\ldots,b_k \in \cB$.
The element $(2\pi i)^{-1}\int_{\Gamma} f(z)\,(z-\tilde{a}_1)^{-1} \cdots (z-\tilde{a}_{k+1})^{-1} \,\d z$ may be understood in terms of a multivariate version of the holomorphic functional calculus in the unital Banach algebra $\cB^{\potimes(k+1)}$;
please see subsection \ref{sec.multivarHFC}.

Next, let $S \subseteq \C$ and
\[
S_{\neq}^m \coloneqq \{(s_1,\ldots,s_m) \in S^m : s_1,\ldots,s_m \text{ are distinct}\} \qquad (m \in \N).
\]
For a function $g \colon S \to \C$, define $g^{[k]} \colon S_{\neq}^{k+1} \to \C$ ($k \in \N_0$) recursively by $g^{[0]} \coloneqq g$ and
\[
g^{[k]}(\mathbf{s}) = \frac{g^{[k-1]}(s_1,\ldots,s_k) - g^{[k-1]}(s_1,\ldots,s_{k-1},s_{k+1})}{s_k - s_{k+1}} \qquad \big(\mathbf{s} = (s_1,\ldots,s_{k+1}) \in S_{\neq}^{k+1}\big).
\]
The function $g^{[k]}$ is called the \textbf{$\boldsymbol{k^{\text{\textbf{th}}}}$ divided difference of $\boldsymbol{g}$}.
If $S = \R$ and $g \in C^k(\R)$, then $g^{[k]}$ extends uniquely to a continuous function on $\R^{k+1}$ (vid.\ \cite[Prop.\ 2.1.3(ii)]{Nikitopoulos2023n}), notated the same way.
If $f \in \Hol\left(U\right)$, then $f^{[k]}$ extends uniquely to a holomorphic function on $U^{k+1}$, notated the same way.
Indeed, suppose $\Gamma$ is any cycle in $U$ such that $W_{\Gamma}(z) = 0$ for all $z \in\C \setminus U$.
By Cauchy's integral formula and induction on $k$, if $\Gamma^*$ is the trace of $\Gamma$ (the union of the images of the curves comprising $\Gamma$) and $V \coloneqq \{z \in U \setminus \Gamma^* : W_{\Gamma}(z) = 1\}$, then
\begin{equation}
    f^{[k]}(\blambda) = \frac{1}{2\pi i}\int_{\Gamma} \frac{f(z)}{(z-\lambda_1)\cdots(z-\lambda_{k+1})} \,\d z \qquad \big( \blambda = (\lambda_1,\ldots,\lambda_{k+1}) \in V^{k+1}\big).\label{eq.divdiffCIF}
\end{equation}
This formula also exposes the connection to Theorem \ref{thm.HFCintro}.

\begin{theorem}\label{thm.HFCderformintro}
Let $f \in \Hol\left(U\right)$, $\mathbf{a} = (a_1,\ldots,a_{k+1}) \in \cB_U^{k+1}$, and $\tilde{a}_i$ be as in \eqref{eq.atilde} with $a_i$ in place of $a$.
If
\[
f_{\sotimes}^{[k]}(\mathbf{a}) \coloneqq f^{[k]}(\tilde{a}_1,\ldots,\tilde{a}_{k+1}) \in \cB^{\potimes(k+1)}
\]
is defined via the multivariate holomorphic functional calculus, which makes sense because $\cB^{\potimes(k+1)}$ is a unital Banach algebra and $\tilde{a}_i\tilde{a}_j=\tilde{a}_j\tilde{a}_i$ for all $i,j \in \{1,\ldots,k+1\}$, then
\[
f_{\sotimes}^{[k]}(\mathbf{a}) = \frac{1}{2\pi i}\int_{\Gamma} f(z) \, (z-\tilde{a}_1)^{-1}\cdots(z-\tilde{a}_{k+1})^{-1}\,\d z
\]
for any cycle $\Gamma$ surrounding $\sigma(a_1) \cup \cdots \cup \sigma(a_{k+1}) = \sigma(\tilde{a}_1) \cup \cdots \cup \sigma(\tilde{a}_{k+1})$ in $U$.
In particular, by Theorem \ref{thm.HFCintro},
\begin{equation}
    \partial_{b_k}\cdots\partial_{b_1} f_{\scB}(a) = \sum_{\pi \in S_k} f_{\sotimes}^{[k]}(\underbrace{a,\ldots,a}_{\mathsmaller{k+1\,\mathrm{times}}})\sh\big[b_{\pi(1)},\ldots,b_{\pi(k)}\big]\label{eq.derformintro}
\end{equation}
for all $a \in \cB_U$ and $b_1,\ldots,b_k \in \cB$.
\end{theorem}

I prove generalizations of Theorems \ref{thm.HFCintro} and \ref{thm.HFCderformintro} in appendix \ref{app.HFC} for the sake of completeness and because I am unaware of a reference in which the proofs of these independently interesting results appear.
The key takeaway from these results is the fact that the $k^{\text{th}}$ derivative of $f_{\scB}$ can be computed in terms of the function $f^{[k]}$ via a peculiar form of multivariate functional calculus.
Heuristically, to compute the directional derivative $\partial_{b_k}\cdots\partial_{b_1}f_{\scB}(a)$,
\begin{enumerate}
    \item take $\mathbf{a} = (a_1,\ldots,a_{k+1}) \coloneqq (a,\ldots,a)$;
    \item plug $\mathbf{a}$ into $f^{[k]}$ in such a way that the variables act ``in lexicographic order,'' i.e., the first variable acts all the way on the left, the second variable acts immediately to the right of the first, and so on;\label{item.step2}
    \item ``insert'' $b_i$ between the action of $a_i$ and $a_{i+1}$ for all $i=1,\ldots,k$; and\label{item.step3}
    \item symmetrize in $(b_1,\ldots,b_k)$.
\end{enumerate}
Steps \ref{item.step2} and \ref{item.step3} motivate the definition of lexicographic functional calculus, which I introduce, along with several results on its appearance in the derivatives of maps induced by the continuous functional calculus, in the next subsection.

Before moving on, it is worthwhile to address the reasonable question,
Why does the ``real-$C^k$ case'' not work the same as the holomorphic case?
After all, the multivariate continuous functional calculus for commuting self-adjoint---or, more generally, normal---elements of a unital $C^*$-algebra is just as basic as the univariate version, so why can we not formulate and prove \eqref{eq.derformintro} in the ``real-$C^k$'' case as well?
For starters, if $\cA$ is a unital $C^*$-algebra, the Banach $\ast$-algebra $\cA^{\potimes(k+1)}$ is not generally a $C^*$-algebra.
Furthermore, if $f \in C^k(\R)$ and $\mathbf{a} = (a_1,\ldots,a_{k+1}) \in \cA_{\sa}^{k+1}$, then
\[
f_{\sotimes}^{[k]}(\mathbf{a}) = f^{[k]}(a_1 \otimes 1^{\otimes k},\ldots,1^{\otimes k} \otimes a_{k+1}),
\]
defined via the multivariate continuous functional calculus, naturally lives in $\cA^{\otimes_{\nu}(k+1)}$, where $\nu$ is a $C^*$-tensor norm on $\cA^{\otimes(k+1)}$ (vid.\ \cite[Ch.\ 3]{BO2008}).
The issue is that the operation $\sh_k|_{\cA^{\otimes (k+1)}}$ is rarely bounded with respect to such $\nu$;
please see \cite[Thm.\ 4.6]{Quigg1985} and \cite[Prop.\ 3.6]{DS2013} as well as \cite{Eckhardt2014}.
Consequently, the object appearing on the right-hand side in \eqref{eq.derformintro} simply does not make sense in the real-$C^k$ case.

\subsection{Framework and results}\label{subsec.results}

Let $\cA$ be a unital $C^*$-algebra, and recall that $\cA_{\sa} = \{a \in \cA : a^*=a\}$.
The goal of lexicographic functional calculus (LFC) is to give meaning to an object resembling $\varphi_{\sotimes}(\mathbf{a}) \sh b$ for all $\mathbf{a} \in \cA_{\sa}^m$, $b \in \cA^{m-1}$, and many, but usually not all, continuous functions $\varphi  \colon \R^m \to \C$.
As explained in the previous subsection, the reason is that such objects should appear in the formula for the $k^{\text{th}}$ derivative of the map $f_{\scA} \colon \cA_{\sa} \to \cA$ when $f \colon \R \to \C$ is sufficiently~regular.

The key is to limit the collection of functions $\varphi$ for which one attempts to make sense of $\varphi_{\sotimes}(\mathbf{a})\sh b$.
The object I use to impose such limitations is a family
\begin{equation}
    \alpha = (\alpha_{r,m} \colon C([-r,r]^m) \to [0,\infty])_{r > 0, \, m \in \N}\label{eq.alphaintro}
\end{equation}
of possibly infinite norms;
to be clear, a \textbf{possibly infinite norm} $\eta$ on a vector space $V$ over $\F \in \{\R,\C\}$ is a function $\eta \colon V \to [0,\infty]$ such that for all $v,w \in V$ and $c \in \F$, $\eta(cv) = |c|\,\eta(v)$ with the convention $0\cdot \infty = 0$, $\eta(v+w) \leq \eta(v)+\eta(w)$, and $\eta(v) = 0$ if and only if $v=0$.
In this case, write
\[
V_{\eta} \coloneqq \{v \in V : \eta(v) < \infty\} \; \text{ and } \; \eta|_{<\infty} \coloneqq \eta|_{V_{\eta}}.
\]
Clearly, $V_{\eta} \subseteq V$ is a linear subspace, and $\eta|_{<\infty}$ is a norm on $V_{\eta}$.

Let $\alpha$ be a family of possibly infinite norms as in \eqref{eq.alphaintro}.
For $m \in \N$, write
\begin{align*}
    \cC_{\alpha}(\R^m) & \coloneqq \big\{\varphi \in C(\R^m) : \varphi|_{[-r,r]^m} \in C([-r,r]^m)_{\alpha_{r,m}} \text{ for all } r > 0\big\} \\
    & = \big\{\varphi \in C(\R^m) : \alpha_{r,m}\big(\varphi|_{[-r,r]^m}\big) < \infty \text{ for all } r > 0\big\},
\end{align*}
and endow $\cC_{\alpha}(\R^m)$ with the Hausdorff, locally convex topology generated by the seminorms
\[
\cC_{\alpha}(\R^m) \ni \varphi \mapsto \alpha_{r,m}(\varphi)  \in [0,\infty) \qquad (r > 0).
\]
Observe that I abused notation above by writing $\alpha_{r,m}(\varphi)$ instead of $\alpha_{r,m}(\varphi|_{[-r,r]^m})$.
I shall continue to do so in the sequel.

\begin{definition}\label{def.alphacont}
$\alpha$ is \textbf{finite on polynomials} if for all $m \in \N$, the set $\cP_m \coloneqq \C[\lambda_1,\ldots,\lambda_m]$ of $m$-variate complex polynomials, viewed as functions from $\R^m$ to $\C$, is contained in $\cC_{\alpha}(\R^m)$, i.e., $\alpha_{r,m}(P) < \infty$ for all $r > 0$ and $P \in \cP_m$.
In this case, define
\[
C_{\alpha}(\R^m) \coloneqq \overline{\cP_m} \subseteq \cC_{\alpha}(\R^m)
\]
to be the space of \textbf{$\boldsymbol{\alpha}$-continuous functions from $\boldsymbol{\R^m}$ to $\boldsymbol{\C}$}, where the closure above takes place in the afore-described topology on $\cC_{\alpha}(\R^m)$.
\end{definition}

For the purposes of this paper, the most important families of possibly infinite norms are the uniform family and the Varopoulos family.

\begin{example}[Uniform family]\label{ex.uniformfamily}
The family of norms
\[
u = \big(u_{r,m} \coloneqq \norm{\cdot}_{\ell^{\infty}([-r,r]^m)} \colon C([-r,r]^m) \to [0,\infty)\big)_{r > 0, \, m \in \N}
\]
is called the \textbf{uniform family}.
Of course, $u$ is a family of possibly infinite norms that is finite on polynomials.
The space $\cC_u(\R^m)$ is simply the space $C(\R^m)$ with the topology of locally uniform convergence.
By the Stone--Weierstrass theorem, $\cP_m$ is dense in $C(\R^m)$, so $C_u(\R^m) = \cC_u(\R^m) = C(\R^m)$.
\end{example}

Some additional notation is necessary to introduce the Varopoulos family.
If $\Om_1,\ldots,\Om_m$ are compact Hausdorff spaces, then the Banach-space projective tensor product
\[
\big(C(\Om_1) \potimes \cdots \potimes C(\Om_m),\norm{\cdot}_{C(\Om_1) \potimes \cdots \potimes C(\Om_m)}\big)
\]
is naturally isomorphic to a normed subalgebra $(V(\Om_1,\ldots,\Om_m),\norm{\cdot}_{V(\Om_1,\ldots,\Om_m)})$ of the space $C(\Om_1 \times \cdots \times \Om_m)$ called the \textbf{Varopoulos algebra};
please see subsection \ref{subsec.Varopoulos} for details.
If $\Om_1=\cdots=\Om_m \coloneqq \Xi$, then $V(\Xi_{(m)}) \coloneqq V(\Om_1,\ldots,\Om_m)$.
In general, if $S$ is a set and $s \in S$, then $s_{(m)}$ is the element of $S^m$ with all components equal to $s$.

\begin{example}[Varopoulos family]\label{ex.Varopoulosfamily}
The family of possibly infinite norms
\[
\beta = \big(\beta_{r,m} \coloneqq \norm{\cdot}_{V([-r,r]_{(m)})} \colon C([-r,r]^m) \to [0,\infty]\big)_{r > 0, \, m \in \N}\pagebreak
\]
is called the \textbf{Varopoulos family};
to be clear, if $\varphi \in C([-r,r]^m) \setminus V([-r,r]_{(m)})$, then $\beta_{r,m}(\varphi) \coloneqq \infty$.
The Varopoulos family is finite on polynomials, and $\cP_m$ is dense in $\cC_{\beta}(\R^m)$ (vid.\ Corollary \ref{cor.polydenseinVC} below), so $C_{\beta}(\R^m) = \cC_{\beta}(\R^m)$.
Write
\[
VC(\R^m) \coloneqq C_{\beta}(\R^m) = \cC_{\beta}(\R^m).
\]
The elements of $VC(\R^m)$ are the \textbf{Varopoulos-continuous functions from $\boldsymbol{\R^m}$ to $\boldsymbol{\C}$}.
\end{example}

Now comes the construction of ($\alpha$-)lexicographic functional calculus.

\begin{notation}\label{nota.algstuff}
Let $A$ be a unital algebra over $\C$ and $m \geq 2$.
\begin{enumerate}[font=\normalfont,label=(\roman*)]
    \item Extend the $\sh = \sh_{m-1}$ notation from subsection \ref{subsec.HFC} to the algebraic case:
    Write
    \[
    \sh = \sh_{m-1} \colon A^{\otimes m} \to L_{m-1}(A^{m-1};A) = \{(m-1)\text{-linear maps from $A^{m-1}$ to }A\}
    \]
    for the (complex-)linear map, written $u \sh_{m-1} b = u \sh b \coloneqq \sh_{m-1}(u)[b]$, determined by
    \[
    (a_1 \otimes \cdots \otimes a_m)\sh[b_1,\ldots,b_{m-1}] = a_1b_1\cdots a_{m-1}b_{m-1}a_m
    \]
    for all $a_1,b_1,\ldots, a_{m-1},b_{m-1}, a_m \in A$.
    Observe that if $u \in A^{\otimes m}$, $I$ is an ideal of $A$,\footnote{$I \subseteq A$ is an \textbf{ideal} of $A$ if $I$ is a complex-linear subspace of $A$ with the property that $arb \in I$ whenever $a,b \in A$ and $r \in I$.} $i=1,\ldots,m-1$, and $b \in A^{i-1} \times I \times A^{m-1-i}$, then $u \sh b \in I$.\label{item.algsharp}
    \item If $P \in \cP_m$ and $\mathbf{a} = (a_1,\ldots,a_m) \in A^m$, define
    \[
    P_{\sotimes}(\mathbf{a}) \coloneqq P(a_1 \otimes 1^{\otimes(m-1)},1 \otimes a_2 \otimes 1^{\otimes (m-2)},\ldots,1^{\otimes(m-1)}\otimes a_m) \in A^{\otimes m}
    \]
    via the multivariate polynomial functional calculus in the algebra $A^{\otimes m}$.\label{item.Ptensor}
\end{enumerate}
In general, $m$-tuples will be bold and denoted by $\mathbf{a} = (a_1,\ldots,a_m)$, while $(m-1)$-tuples will not be bold and will be denoted by $b = (b_1,\ldots,b_{m-1})$.
When a $k$ is involved, it will usually be the case that $m=k+1$.
\end{notation}

As I encourage the reader to verify, if $\mathbf{a} \in A^m$, $b \in A^{m-1}$, and
\[
P(\blambda) = \sum_{|\delta| \leq d} c_{\delta} \blambda^{\delta} = \sum_{|\delta| \leq d} c_{\delta} \lambda_1^{\delta_1}\cdots\lambda_m^{\delta_m} \in \cP_m
\]
in standard multi-index notation, then
\[
P_{\sotimes}(\mathbf{a})\sh b = \sum_{|\delta| \leq d} c_{\delta} \,a_1^{\delta_1}b_1\cdots a_{m-1}^{\delta_{m-1}} b_{m-1} a_m^{\delta_m} \in A.
\]
The map
\[
\cP_m \ni P \mapsto (\mathbf{a} \mapsto (b \mapsto P_{\sotimes}(\mathbf{a})\sh b)) \in \operatorname{Fun}\left(A^m;L_{m-1}(A^{m-1};A)\right)
\]
is called the (\textbf{$\boldsymbol{m}$-variate}) \textbf{polynomial lexicographic functional calculus} (\textbf{LFC}) \textbf{in $\boldsymbol{A}$}.
Under natural assumptions on $\alpha$, the polynomial LFC in a unital $C^*$-algebra extends to the $\alpha$-continuous functions.

\begin{theorem}[Existence of LFC]\label{thm.LFCintro}
Recall that $\alpha$ is a family of possibly infinite norms as in \eqref{eq.alphaintro} and $\cA$ is a unital $C^*$-algebra.
Also, write
\[
\cA_{\sa,r} \coloneqq \{a \in \cA_{\sa} : \norm{a} \leq r\} \qquad (r \geq 0).
\]
Suppose $\alpha$ is finite on polynomials and that for every $m \geq 2$, $P \in \cP_m$, and $b \in \cA^{m-1}$,
\[
\sup_{\mathbf{a} \in \cA_{\sa,r}^m}\norm{P_{\sotimes}(\mathbf{a})\sh b} \leq \alpha_{r,m}(P)\prod_{i=1}^{m-1} \norm{b_i} \qquad (r > 0).
\]
If $\mathbf{a} \in \cA_{\sa}^m$ and $b \in \cA^{m-1}$, then the map $\cP_m \ni P \mapsto P_{\sotimes}(\mathbf{a})\sh b \in \cA$ extends uniquely to a continuous linear map $\Phi \colon C_{\alpha}(\R^m) \to \cA$, denoted by
\[
\varphi_{\scA,\alpha}(\mathbf{a}) \sh b \coloneqq \Phi(\varphi) \qquad (\varphi \in C_{\alpha}(\R^m)).
\]
The map sending $\varphi \in C_{\alpha}(\R^m)$ to the function $\mathbf{a} \mapsto (b \mapsto \varphi_{\scA,\alpha}(\mathbf{a})\sh b)$ is called the (\textbf{$\boldsymbol{m}$-variate}) \textbf{$\boldsymbol{\alpha}$-lexicographic functional calculus}---($m$-variate) $\alpha$-LFC for short.
\end{theorem}

The result above is a special case of (a weakening of) Theorem \ref{thm.LFC}\ref{item.LFCinI} from subsection~\ref{subsec.constructLFC}.
The next result synthesizes special cases of some of the development in subsections \ref{subsec.commLFC}--\ref{subsec.MOIfam}, which study examples of LFC in detail.

\begin{theorem}\label{thm.LFCexamplesintro}
Recall that $\cA$ is a unital $C^*$-algebra, $u$ is the uniform family, and $\beta$ is the Varopoulos family.
\begin{enumerate}[label=(\roman*),font=\normalfont]
    \item If $\cA$ is commutative, then $u$ satisfies the hypotheses of Theorem \ref{thm.LFCintro}.\label{item.commLFCintro}
    \item If $\cA$ is finite dimensional, then the family $u^{\scA} \coloneqq ((\dim \cA)^{m-1}u_{r,m})_{r > 0, \, m \in \N}$ satisfies the hypotheses of Theorem \ref{thm.LFCintro}.\label{item.fdLFCintro}
    \item $\beta$ satisfies the hypotheses of Theorem \ref{thm.LFCintro}.\label{item.VaropLFCintro}
\end{enumerate}
In particular, the $u$-LFC in $\cA$ is defined if $\cA$ is commutative, the $u^{\scA}$-LFC in $\cA$ is defined if $\cA$ is finite dimensional, and the $\beta$-LFC---\textbf{Varopoulos LFC}---in $\cA$ is defined no matter the choice of $\cA$.
\end{theorem}

Item \ref{item.commLFCintro} is a special case of Theorem \ref{thm.commLFC}, item \ref{item.fdLFCintro} is a special case of Corollary \ref{cor.findimLFC},\footnote{The constant $\dim \cA$ in item \ref{item.fdLFCintro} can be improved substantially;
please see Proposition \ref{prop.optimestim}.} and item \ref{item.VaropLFCintro} is a special case of Corollary \ref{cor.VaropLFC}\ref{item.VaropLFC}.
Each of these results also provides a recipe to evaluate the LFC in each respective context.
Notably, one method I present for evaluating the Varopoulos LFC makes use of a new, independently interesting characterization of the Varopoulos algebra, inspired by MOI theory, developed in subsection \ref{subsec.Varopoulos}.

Finally, let us see how LFC appears in the formulas for the higher derivatives of $f_{\scA}$ when $f \colon \R \to \C$ is sufficiently regular.
Perhaps unsurprisingly, ``sufficiently regular'' will be measured using $\alpha$.
For $k \in \N_0 \cup \{\infty\}$, define
\begin{align*}
    \cC_{\alpha}^k(\R) & \coloneqq \big\{f \in C^k(\R) : f^{[i]} \in \cC_{\alpha}\big(\R^{i+1}\big) \text{ whenever } 0 \leq i < k +1\big\}\\
    & = \left\{f \in C^k(\R) : \alpha_{r,i+1}\big(f^{[i]}\big)  < \infty \text{ whenever } r > 0 \text{ and } 0 \leq i < k+1\right\},
\end{align*}
and endow $\cC_{\alpha}^k(\R)$ with the Hausdorff, locally convex topology induced by the seminorms
\[
\cC_{\alpha}^k(\R) \ni f \mapsto \alpha_{r,i+1}\big(f^{[i]}\big)  \in [0,\infty) \qquad (r > 0, \; 0 \leq i < k+1).
\]
If $\alpha$ is finite on polynomials, then $\cP_1 = \C[\lambda] \subseteq \cC_{\alpha}^k(\R)$ because the $k^{\text{th}}$ divided difference of a polynomial is a polynomial in $k+1$ variables (vid.\ \cite[Ex.\ 2.1.5]{Nikitopoulos2023n}).
This ensures the following definition makes sense.

\begin{definition}\label{def.Ckalphaintro}
Suppose $\alpha$ is finite on polynomials, and write $\cP \coloneqq \cP_1 = \C[\lambda]$.
Define
\[
C_{\alpha}^k(\R) \coloneqq \overline{\cP} \subseteq \cC_{\alpha}^k(\R)
\]
to be the space of \textbf{$\boldsymbol{\alpha}$-$\boldsymbol{C^k}$ functions} (\textbf{from $\boldsymbol{\R}$ to $\boldsymbol{\C}$}), where the closure above takes place in the afore-described topology on $\cC_{\alpha}^k(\R)$.
Also, define
\[
VC^k(\R) \coloneqq C_{\beta}^k(\R)
\]
to be the space of \textbf{Varopoulos-$\boldsymbol{C^k}$ functions} (\textbf{from $\boldsymbol{\R}$ to $\boldsymbol{\C}$}).
\end{definition}

Subsection \ref{subsec.Ckalpha} develops various properties of $\cC_{\alpha}^k(\R)$ and $C_{\alpha}^k(\R)$.
Particularly noteworthy is Theorem \ref{thm.mathbfCkalpha=Ckalpha} therein, which establishes some conditions, satisfied by both the Varopoulos and uniform families, under which $C_{\alpha}^k(\R) = \cC_{\alpha}^k(\R)$.
In particular, $C_u^k(\R) = C^k(\R)$, and $C_{\beta}^k(\R) = \cC_{\beta}^k(\R)$, i.e., $VC^k(\R) = \cC_{\beta}^k(\R)$.
The space $VC^k(\R) = C_{\beta}^k(\R)$ is a concrete realization of the space $C_{\mathrm{nc}}^k(\R)$ of ``noncommutative $C^k$ functions'' introduced and briefly studied by D.\ A.\ Jekel in \cite[Ch.\ 18]{Jekel2020}.
Therein, Jekel defines $C_{\mathrm{nc}}^k(\R)$ as the abstract Fr\'echet-space completion of $\cP$ with respect to seminorms similar to $p \mapsto \beta_{r,i+1}(p^{[i]})$ ($r > 0$, $0 \leq i < k+1$) but formulated more algebraically in terms of Voiculescu's free difference quotients.
The fact that $C_{\beta}^k(\R) = \cC_{\beta}^k(\R)$ gives the new characterization of Jekel's space of noncommutative $C^k$ functions as the space of $f \in C^k(\R)$ such that $f^{[i]} \in VC(\R^{i+1})$ whenever $0 \leq i < k+1$.

Though $VC^k(\R) \subsetneq C^k(\R)$ for all $k \in \N$ (vid.\ Corollary \ref{cor.VCksubsetneqCk} below), it is perhaps unsurprising, considering the exposition, that if $f \colon \R \to \C$ is ``slightly better than $C^k$,'' then $f$ is Varopoulos $C^k$.

\begin{theorem}\label{thm.BesovHolderVaropoulos}
$\dot{B}_1^{k,\infty}(\R) \subseteq VC^k(\R)$, and $C_{\loc}^{k,\e}(\R) \subseteq VC^k(\R)$ for all $\e > 0$.
\end{theorem}

This result is proven at the end of subsection \ref{subsec.Ckalpha}.
Jekel showed in \cite[Prop.\ 18.1.5]{Jekel2020} that if a function $f \colon \R \to \C$ has a continuous Fourier transform $\widehat{f}$ satisfying $\int_{\R} (1+|t|^k)\widehat{f}(t)\,\d t < \infty$, then $f \in C_{\mathrm{nc}}^k(\R)$, i.e., $f$ is Varopoulos $C^k$.
Such functions $f$ are very special examples of functions in the Besov space $\dot{B}_1^{k,\infty}(\R)$, so Theorem \ref{thm.BesovHolderVaropoulos} improves Jekel's result substantially.

Next, observe that if $f \in C_{\alpha}^k(\R)$ and $i=1,\ldots,k$, then $f^{[i]} \in C_{\alpha}(\R^{i+1})$.
Consequently, if $\alpha$ satisfies the hypotheses of Theorem \ref{thm.LFCintro}, then $f_{\scA,\alpha}^{[i]}(\mathbf{a})\sh b \in \cA$ is defined, via the $(i+1)$-variate $\alpha$-LFC, for any $\mathbf{a} \in \cA_{\sa}^{i+1}$ and $b \in \cA^i$.
In particular, the right-hand side of the formula below is defined.

\begin{theorem}[Derivative formula involving LFC]\label{thm.derformintro}
Suppose $\alpha_{r,1} = u_{r,1}$ for all $r > 0$ and $\alpha$ satisfies the hypotheses of Theorem \ref{thm.LFCintro}.
If $f \in C_{\alpha}^k(\R)$, then $f_{\scA} \in C^k(\cA_{\sa};\cA)$, and
\[
\partial_{b_k}\cdots\partial_{b_1} f_{\scA}(a) = \sum_{\pi \in S_k} f_{\scA,\alpha}^{[k]}(\underbrace{a,\ldots,a}_{\mathsmaller{k+1\,\mathrm{times}}})\sh\big[b_{\pi(1)},\ldots,b_{\pi(k)}\big]
\]
for all $a,b_1,\ldots,b_k \in \cA_{\sa}$.
\end{theorem}

This result is a special case of Theorem \ref{thm.derform}.
\pagebreak

Finally, let us complete the picture by piecing together why Theorem \ref{thm.motivatingthm} is a corollary of Theorem \ref{thm.derformintro}.
Since, as mentioned after Definition \ref{def.Ckalphaintro}, $C_u^k(\R) = C^k(\R)$ and it is obvious that $C_{u^{\scA}}^k(\R) = C_u^k(\R)$ if $\cA$ is finite dimensional, the first and second items of Theorem~\ref{thm.motivatingthm} are proven by combining Theorem \ref{thm.derformintro} with the first and second items, respectively, of Theorem \ref{thm.LFCexamplesintro}.
Also, since $\dot{B}_1^{k,\infty}(\R) \cup C_{\loc}^{k,\e}(\R) \subseteq VC^k(\R)$ by Theorem \ref{thm.BesovHolderVaropoulos}, the third item of Theorem \ref{thm.motivatingthm} is proven by combining Theorem \ref{thm.derformintro} with the third item of Theorem \ref{thm.LFCexamplesintro}.
As promised, this confirms Metatheorem \ref{metathm.meta}.

\section{Preliminaries}\label{sec.prelims}

\subsection{The Varopoulos algebra}\label{subsec.Varopoulos}

Let $m \in \N$ and $\Om_1,\ldots,\Om_m$ be compact Hausdorff spaces, and write $\Om \coloneqq \Om_1 \times \cdots \times \Om_m$.
In this subsection, I discuss the Varopoulos algebra, a concrete representation of the projective tensor product $C(\Om_1) \potimes \cdots \potimes C(\Om_m)$ named after N.\ Th.\ Varopoulos (vid.\ \cite{Varopoulos1967}).
Please see \cite[Ch.\ 2]{Ryan2002} or \cite[\S1.5]{Nikitopoulos2024} for information about projective tensor products.

\begin{notation}\label{nota.tensfunc}
Let $X,Y,S_1,\ldots,S_m$ be sets, and write $S \coloneqq S_1 \times \cdots \times S_m$.
\begin{enumerate}[font=\normalfont,label=(\roman*)]
    \item $Y^X$ is the set of functions from $X$ to $Y$.\label{item.functionXtoY}
    \item If $f_i \in \C^{S_i}$ for all $i =1,\ldots,m$, define the function $f_1 \otimes \cdots \otimes f_m \colon S \to \C$ by
    \[
    (f_1 \otimes \cdots \otimes f_m)(\mathbf{s}) \coloneqq f_1(s_1)\cdots f_m(s_m) \qquad ( \mathbf{s} = (s_1,\ldots,s_m) \in S).
    \]
    Since the linear map from $\C^{S_1} \otimes \cdots \otimes \C^{S_m}$ to $\C^S$ determined by
    \[
    f_1\otimes \cdots \otimes f_m \mapsto ((s_1,\ldots,s_m) \mapsto f_1(s_1)\cdots f_m(s_m))
    \]
    is injective, the notation above is justified.\label{item.tensorfunc}
    \item If $f \in \C^X$, then $\norm{f}_{\ell^{\infty}(X)} \coloneqq \sup\left\{|f(x)| : x \in X\right\} \in [0,\infty]$.\label{item.ellinfty}
\end{enumerate}
\end{notation}

\begin{definition}\label{def.Varopoulos}
Let $\varphi \in C(\Om)$, and suppose that for each $i =1,\ldots,m$, there exists a sequence $(\varphi_{i,n})_{n \in \N}$ in $C(\Om_i)$ such that
\begin{equation}
    \sum_{n=1}^{\infty} \prod_{i=1}^m\norm{\varphi_{i,n}}_{\ell^{\infty}(\Om_i)} < \infty \; \text{ and } \; \varphi(\boldsymbol{\om}) = \sum_{n=1}^{\infty} (\varphi_{1,n} \otimes \cdots \otimes \varphi_{m,n})(\boldsymbol{\om}) \qquad (\boldsymbol{\om} \in \Om). \label{eq.Vardecomp}
\end{equation}
Define
\[
\norm{\varphi}_{V(\Om_1,\ldots,\Om_m)}\hspace{-0.2mm} \coloneqq\hspace{-0.2mm} \inf\hspace{-0.2mm}\left\{\sum_{n=1}^{\infty} \prod_{i=1}^m\norm{\varphi_{i,n}}_{\ell^{\infty}(\Om_i)}\hspace{-0.2mm} :\hspace{-0.2mm} \text{the\hspace{-0.2mm} sequences } (\varphi_{i,n})_{n \in \N}\hspace{-0.2mm} \in\hspace{-0.2mm} C(\Om_i)^{\N}\hspace{-0.2mm} \text{ satisfy\hspace{-0.2mm} \eqref{eq.Vardecomp}}\hspace{-0.2mm}\right\}\hspace{-0.2mm}.
\]
If no such sequences exist, then $\norm{\varphi}_{V(\Om_1,\ldots,\Om_m)} \coloneqq \infty$.
The \textbf{Varopoulos algebra} is defined to be the set $V(\Om_1,\ldots,\Om_m) \coloneqq \big\{\varphi \in C(\Om) : \norm{\varphi}_{V(\Om_1,\ldots,\Om_m)} < \infty\big\}$.
\end{definition}

By standard arguments, the Varopoulos algebra $V(\Om_1,\ldots,\Om_m)$ is a unital $\ast$-subalgebra of $C(\Om)$, and $\big(V(\Om_1,\ldots,\Om_m),\norm{\cdot}_{V(\Om_1,\ldots,\Om_m)}\big)$ is a unital Banach $\ast$-algebra.
Furthermore, it is clear that $\norm{\varphi}_{\ell^{\infty}(\Om)} \leq \norm{\varphi}_{V(\Om_1,\ldots,\Om_m)}$ for all $\varphi \in C(\Om)$.
In particular, the inclusion of $V(\Om_1,\ldots,\Om_m)$ into $C(\Om)$ is continuous.

\begin{example}[Multivariate polynomials]\label{ex.mvarpoly}
Recall that $\cP_m = \C[\lambda_1,\ldots,\lambda_m]$.
Suppose
\[
P(\blambda) = \sum_{|\delta| \leq d}c_{\delta} \,\blambda^{\delta} \in \cP_m.
\]
If $r_i > 0$ and $\Om_i \coloneqq \{z \in \C : |z| \leq r_i\}$ for all $i =1,\ldots,m$, then
\[
\norm{P|_{\Om_1 \times \cdots \times \Om_m}}_{V(\Om_1,\ldots,\Om_m)} \leq \sum_{|\delta| \leq d} |c_{\delta}| \prod_{i=1}^m\sup_{|\lambda_i| \leq r_i}\big|\lambda_i^{\delta_i}\big| \leq \sum_{|\delta| \leq d} |c_{\delta}| \, r^{|\delta|} < \infty,
\]
where $r \coloneqq \max\left\{r_1,\ldots,r_m\right\}$.
Since $V(\Om_1,\ldots,\Om_m)$ is closed under multiplication and complex conjugation, (the restrictions of) multivariate polynomials in $\blambda$ and $\bar{\blambda}$ belong to $V(\Om_1,\ldots,\Om_m)$.
Actually, such polynomial functions are dense.
\end{example}

\begin{proposition}[Density of $\ast$-polynomials]\label{prop.polydenseVC}
Suppose $\Om_i \subseteq \C$ is a compact set for all $i = 1,\ldots,m$.
The set $\mathrm{P}^*(\Om_1,\ldots,\Om_m) \subseteq V(\Om_1,\ldots,\Om_m)$ of functions of the form
\[
\Om \ni \blambda \mapsto P\big(\blambda,\bar{\blambda}\big) \in \C,
\]
where $P(\lambda_1,\ldots,\lambda_m,\mu_1,\ldots,\mu_m) \in \cP_{2m}$, is dense in $V(\Om_1,\ldots,\Om_m)$. 
\end{proposition}

\begin{proof}[Sketch of proof]
By definition of $V(\Om_1,\ldots,\Om_m)$, the set
\[
T(\Om_1,\ldots,\Om_m) \coloneqq \Bigg\{\sum_{n=1}^N \varphi_{1,n} \otimes \cdots \otimes \varphi_{m,n} : N \in \N \text{ and } (\varphi_{i,n})_{n=1}^N \in C(\Om_i)^N, \; i = 1,\ldots,m\Bigg\}
\]
is dense in $V(\Om_1,\ldots,\Om_m)$.
By the Stone--Weierstrass theorem, $\mathrm{P}^*(\Om_i)$ is dense in $C(\Om_i)$ for all $i = 1,\ldots,m$.
By approximating the $\varphi_{i,n}$s by elements of $\mathrm{P}^*(\Om_i)$, it follows that $\mathrm{P}^*(\Om_1,\ldots,\Om_m)$ is dense in $T(\Om_1,\ldots,\Om_m)$.
The result follows.
\end{proof}

\begin{corollary}\label{cor.polydenseinVC}
$\cP_m$ is dense in $\cC_{\beta}(\R^m)$, i.e., $\cC_{\beta}(\R^m) = C_{\beta}(\R^m)$, as stated in Example \ref{ex.Varopoulosfamily}.
\end{corollary}

\begin{proof}
Let $\varphi \in \cC_{\beta}(\R^m)$.
By Proposition \ref{prop.polydenseVC}, if $n \in \N$, then there exists a $P_n \in \cP_m$ such that $\beta_{n,m}(\varphi - P_n) < 1/n$.
The sequence $(P_n)_{n \in \N}$ converges in $\cC_{\beta}(\R^m)$ to $\varphi$.
\end{proof}

I now provide a description of $V(\Om_1,\ldots,\Om_m)$ (with $\Om_1,\ldots,\Om_m$ metrizable) inspired by the integral projective tensor products (vid.\ \cite[\S4.1]{Nikitopoulos2023m} or \cite[\S2.2]{Nikitopoulos2023n}) from the theory of multiple operator integrals (MOIs).
The characterization and its proof, as well as several other results and proofs in this paper, require the theories of strong/Bochner measurability and integrability in Fr\'echet spaces;
please see \cite[App.\ E]{Cohn2013} for the Banach-space case, which is all that is needed in this subsection, and \cite[\S1.1]{Nikitopoulos2024} for the general case.

\begin{lemma}\label{lem.projmeas}
Suppose $\Om_1,\ldots,\Om_m$ are metrizable, and let $(\Sigma,\sH)$ be a measurable space.
If for all $i =1,\ldots,m$, $\varphi_i \colon \Om_i \times \Sigma \to \C$ is product measurable, i.e., $\cB_{\Om_i} \otimes \sH/\cB_{\C}$-measurable, and $\varphi_i(\cdot,\sigma) \in C(\Om_i)$ whenever $\sigma \in \Sigma$, then the map
\[
\Sigma \ni \sigma \mapsto \varphi_1(\cdot,\sigma) \otimes \cdots \otimes \varphi_m(\cdot,\sigma) \in V(\Om_1,\ldots,\Om_m)
\]
is strongly measurable.
\end{lemma}

\begin{proof}
I first prove the lemma assuming $m=1$, in which case $\Om_1 = \Om$ and $\varphi \coloneqq \varphi_1$.
By the Riesz--Markov theorem, $C(\Om)^*$ is isometrically isomorphic to the space $M(\Om)$ of Radon complex measures on $\Om$.
Now, if $\mu \in M(\Om)$, then the function $\Sigma \ni \sigma \mapsto \int_{\Om} \varphi(\cdot,\sigma)\,\d\mu \in \C$ is measurable by a standard measure-theory argument (vid.\ \cite[Lem.\ 4.2.2]{Nikitopoulos2023m}).
Therefore, the map $\Sigma \ni \sigma \mapsto \varphi(\cdot,\sigma) \in C(\Om)$ is weakly measurable.
Since $\Om$ is compact and metrizable, $C(\Om)$ is a separable Banach space.
The strong measurability of $\Sigma \ni \sigma \mapsto \varphi(\cdot,\sigma) \in C(\Om)$ then follows from Pettis's measurability theorem.

Next, let $m \in \N$ be general, and fix $i = 1,\ldots,m$.
By the previous paragraph, the map $\Sigma \ni \sigma \mapsto F_i(\sigma) \coloneqq \varphi_i(\cdot,\sigma) \in C(\Om_i)$ is strongly measurable.
Let $(s_{i,n})_{n \in \N}$ be a sequence of simple maps from $\Sigma$ to $C(\Om_i)$ converging pointwise to $F_i$.
Then $(s_{1,n}(\cdot) \otimes \cdots \otimes s_{m,n}(\cdot))_{n \in \N}$ is a sequence of simple maps from $\Sigma$ to $V(\Om_1,\ldots,\Om_m)$ converging pointwise to the function $F(\cdot) \coloneqq F_1(\cdot) \otimes \cdots \otimes F_m(\cdot)$, which shows that $F$ is strongly measurable.
\end{proof}

\begin{theorem}[Integral description of the Varopoulos algebra]\label{thm.IPTPV}
Suppose $\Om_1,\ldots,\Om_m$ are metrizable.
Let $(\Sigma,\sH,\rho)$ be a measure space, and for all $i = 1,\ldots,m$, let $\varphi_i \colon \Om_i \times \Sigma \to \C$ be a product-measurable function such that $\varphi_i(\cdot,\sigma) \in C(\Om_i)$ whenever $\sigma \in \Sigma$.
If
\begin{equation}
    \int_{\Sigma} \prod_{i=1}^m\norm{\varphi_i(\cdot,\sigma)}_{\ell^{\infty}(\Om_i)} \, \rho(\d\sigma) < \infty \; \text{ and } \; \varphi(\boldsymbol{\om}) \coloneqq \int_{\Sigma} \prod_{i=1}^m\varphi_i(\om_i,\sigma)\,\rho(\d\sigma) \quad (\boldsymbol{\om} \in \Om), \label{eq.integphi}
\end{equation}
then
\[
\varphi = \int_{\Sigma} \varphi_1(\cdot,\sigma) \otimes \cdots \otimes \varphi_m(\cdot,\sigma)\,\rho(\d\sigma) \in V(\Om_1,\ldots,\Om_m)
\]
as a $V(\Om_1,\ldots,\Om_m)$-valued Bochner integral, and
\begin{equation}
    \norm{\varphi}_{V(\Om_1,\ldots,\Om_m)} \leq \int_{\Sigma} \prod_{i=1}^m\norm{\varphi_i(\cdot,\sigma)}_{\ell^{\infty}(\Om_i)}  \, \rho(\d\sigma).\label{eq.integestim}
\end{equation}
\end{theorem}

\begin{proof}
By Lemma \ref{lem.projmeas}, the map
\[
\Sigma \ni \sigma \mapsto F(\sigma) \coloneqq \varphi_1(\cdot,\sigma) \otimes \cdots \otimes \varphi_m(\cdot,\sigma) \in V(\Om_1,\ldots,\Om_m)
\]
is strongly measurable.
Since
\[
\int_{\Sigma}\norm{F}_{V(\Om_1,\ldots,\Om_m)}\,\d\rho = \int_{\Sigma} \prod_{i=1}^m\norm{\varphi_i(\cdot,\sigma)}_{\ell^{\infty}(\Om_i)} \, \rho(\d\sigma) < \infty,
\]
it follows that $F$ is Bochner $\rho$-integrable.
The identity $\varphi = \int_{\Sigma} F \,\d\rho$ then follows by applying the evaluation functionals $\{V(\Om_1,\ldots,\Om_m) \ni \psi \mapsto \psi(\boldsymbol{\om}) \in \C : \boldsymbol{\om} \in \Om\}$ to $\int_{\Sigma} F \,\d\rho$.
Finally, \eqref{eq.integestim} follows from the triangle inequality for Bochner integrals.
\end{proof}

The reason for the name of Theorem \ref{thm.IPTPV} is the following immediate consequence:
If $\Om_1,\ldots,\Om_m$ are metrizable, then the Varopoulos algebra $V(\Om_1,\ldots,\Om_m)$ is precisely the space of functions $\varphi \in C(\Om)$ such that there exist a measure space $(\Sigma,\sH,\rho)$ and product-measurable functions $\varphi_1 \colon \Sigma \times \Om_1 \to \C,\ldots,\varphi_m \colon \Sigma \times \Om_m \to \C$ as in Theorem \ref{thm.IPTPV} satisfying
\[
\varphi(\boldsymbol{\om}) = \int_{\Sigma} \varphi_1(\om_1,\sigma)\cdots \varphi_m(\om_m,\sigma)\,\rho(\d\sigma) \qquad (\boldsymbol{\om} \in \Om).
\]
Furthermore,
\[
\norm{\varphi}_{V(\Om_1,\ldots,\Om_m)} = \inf\Bigg\{\int_{\Sigma} \prod_{i=1}^m\norm{\varphi_i(\cdot,\sigma)}_{\ell^{\infty}(\Om_i)}\,\rho(\d\sigma) : \underset{\text{are as in the previous sentence}}{{}^{\mbox{\smaller$(\Sigma,\sH,\rho)\text{ and } \varphi_1,\ldots,\varphi_m$}}}\Bigg\}.
\]
In the terminology and notation of MOIs, one might say that $V(\Om_1,\ldots,\Om_m)$ is the ``integral projective tensor product $C(\Om_1) \iotimes \cdots \iotimes C(\Om_m)$.''

\begin{example}\label{ex.multivarWiener}
Let $(\Sigma,\sH)$ be a measurable space, $\mu$ be a complex measure on $(\Sigma,\sH)$, $|\mu|$ be the total variation of $\mu$, and $\boldsymbol{\xi} \colon \Sigma \to \R^m$ be a Borel-measurable function.
Define
\[
\varphi(\blambda) \coloneqq \int_{\Sigma}e^{i\,\boldsymbol{\xi}(\sigma)\cdot \blambda} \, \mu(\d\sigma) \qquad (\blambda \in \R^m).
\]
I claim that $\varphi \in VC(\R^m) = C_{\beta}(\R^m) = \cC_{\beta}(\R^m)$ and
\begin{equation}
    \sup_{r > 0}\beta_{r,m}(\varphi) = \sup_{r > 0} \norm{\varphi|_{[-r,r]^m}}_{V([-r,r]_{(m)})} \leq |\mu|(\Sigma). \label{eq.multivarWiener}
\end{equation}
Indeed, let $r > 0$, and apply Theorem \ref{thm.IPTPV} with $\Om_1 = \cdots = \Om_m = [-r,r]$, $\rho \coloneqq |\mu|$,
\[
\varphi_1(\om_1,\sigma) \coloneqq e^{i\xi_1(\sigma)\om_1}\frac{\d\mu}{\d|\mu|}(\sigma), \; \text{ and } \; \varphi_j(\om_j,\sigma) \coloneqq e^{i\xi_j(\sigma)\om_j} \; \text{ for } j = 2,\ldots,m
\]
to deduce the claim.
\end{example}

For the sake of completeness and since it will be important later (in subsection \ref{subsec.VaropLFC}), I end this subsection by proving the well-known result that $V(\Om_1,\ldots,\Om_m)$ is naturally isometrically isomorphic to $C(\Om_1) \potimes \cdots \potimes C(\Om_m)$.

\begin{theorem}\label{thm.inj}
If $\iota_{\mathsmaller{\Om_1,\ldots,\Om_m}} \colon C(\Om_1) \potimes \cdots \potimes C(\Om_m) \to C(\Om)$ is the bounded linear map determined via the universal property of $\potimes$ by
\[
\varphi_1 \otimes \cdots \otimes \varphi_m \mapsto ((\om_1,\ldots,\om_m) \mapsto \varphi_1(\om_1)\cdots \varphi_m(\om_m)),
\]
then $\iota_{\mathsmaller{\Om_1,\ldots,\Om_m}}$ is an injective, unital $\ast$-homomorphism.
\end{theorem}

\begin{proof}
The only nontrivial claim is that $\iota_{\mathsmaller{\Om_1,\ldots,\Om_m}}$ is injective.
I prove this by induction on $m \geq 2$.
By \cite[Ex.\ 4.2]{Ryan2002}, $C(\Om_1)$ has the approximation property.
Consequently, the injectivity of $\iota_{\mathsmaller{\Om_1,\Om_2}}$ follows from \cite[Prop.\ 4.6]{Ryan2002}.

Now, assume the result is true for $m \geq 2$ spaces, and write $\Xi \coloneqq \Om_2 \times \cdots \times \Om_m$.
By the $m=2$ case, the map $\iota_{\mathsmaller{\Om_1,\Xi}} \colon C(\Om_1) \potimes C(\Xi) \to C(\Om_1 \times \Xi) = C(\Om)$ is injective.
By the induction hypothesis, the map $\iota_{\mathsmaller{\Om_2,\ldots,\Om_m}} \colon C(\Om_2) \potimes \cdots \potimes C(\Om_m) \to C(\Xi)$ is injective.
Since $C(\Om_1)$ has the approximation property, it follows from \cite[Exer.\ 4.1]{Ryan2002} that the map
\[
\id_{C(\Om_1)} \potimes \iota_{\mathsmaller{\Om_2,\ldots,\Om_m}} \colon \underbrace{C(\Om_1) \potimes (C(\Om_2) \potimes \cdots \potimes C(\Om_m) )}_{= C(\Om_1) \potimes \cdots \potimes C(\Om_m)} \to C(\Om_1) \potimes C(\Xi)
\]
is injective as well.
Since $\iota_{\mathsmaller{\Om_1,\ldots,\Om_m}} = \iota_{\mathsmaller{\Om_1,\Xi}} \circ \big(\id_{C(\Om_1)} \potimes \iota_{\mathsmaller{\Om_2,\ldots,\Om_m}}\big)$, the proof is complete.
\end{proof}

\begin{corollary}\label{cor.seriesVarop}
If $\iota_{\mathsmaller{\Om_1,\ldots,\Om_m}}$ is as in Theorem \ref{thm.inj}, then $\im \iota_{\mathsmaller{\Om_1,\ldots,\Om_m}} = V(\Om_1,\ldots,\Om_m)$, and
\[
\norm{\iota_{\mathsmaller{\Om_1,\ldots,\Om_m}}(a)}_{V(\Om_1,\ldots,\Om_m)} = \norm{a}_{C(\Om_1) \potimes \cdots \potimes C(\Om_m)} \qquad (a \in C(\Om_1) \potimes \cdots \potimes C(\Om_m)).
\]
In other words, $\iota_{\mathsmaller{\Om_1,\ldots,\Om_m}}$ is an isometric $\ast$-isomorphism between $C(\Om_1) \potimes \cdots \potimes C(\Om_m)$ and the Varopoulos algebra $V(\Om_1,\ldots,\Om_m)$.
\end{corollary}

\begin{proof}
By \cite[Prop.\ 2.8]{Ryan2002} and an inductive argument, if $V_1,\ldots,V_m$ are Banach spaces, then for every $u \in V_1 \potimes \cdots \potimes V_m$, there exist sequences $(v_{1,n})_{n \in \N} \in V_1^{\N},\ldots,(v_{m,n})_{n \in \N} \in V_m^{\N}$ such that
\begin{equation}
    \sum_{n=1}^{\infty}\prod_{i=1}^m\norm{v_{i,n}}_{V_i} < \infty \; \text{ and } \; u = \sum_{n=1}^{\infty} v_{1,n} \otimes \cdots \otimes v_{m,n} \; \text{ in } \; V_1 \potimes \cdots \potimes V_m;\label{eq.projdecomp}
\end{equation}
furthermore,
\[
\norm{u}_{V_1 \potimes \cdots \potimes V_m} = \inf \Bigg\{\sum_{n=1}^{\infty}\prod_{i=1}^m\norm{v_{i,n}}_{V_i} : \text{the sequences } (v_{i,n})_{n \in \N} \in V_i^{\N} \text{ satisfy \eqref{eq.projdecomp}}\Bigg\}.
\]
Combining this fact with Theorem \ref{thm.inj} yields the result.
\end{proof}

\subsection{Symmetrically normed ideals}\label{subsec.SNI}

In this subsection, I introduce the normed ideals of interest:
symmetrically normed ideals.
Throughout, $\cB$ is a unital Banach algebra with norm $\norm{\cdot}$.

\begin{definition}\label{def.sni}
Let $\cI \subseteq \cB$ be an ideal, and suppose $\norm{\cdot}_{\cI}$ is a norm on $\cI$.
The pair $(\cI,\norm{\cdot}_{\cI})$ is a \textbf{Banach ideal} of $\cB$ if $(\cI,\norm{\cdot}_{\cI})$ is a Banach space and the inclusion
\[
\iota_{\scI} \colon (\cI,\norm{\cdot}_{\cI}) \to (\cB,\norm{\cdot})
\]
is bounded;
in this case, write $(\cI,\norm{\cdot}_{\cI}) \unlhd \cB$ and $C_{\cI} \coloneqq \|\iota_{\scI}\|_{\cI \to \cB} \in [0,\infty)$.
If, in addition,
\[
\norm{arb}_{\cI} \leq \norm{a}\norm{r}_{\cI}\norm{b} \qquad  (a,b \in \cB, \; r \in \cI),
\]
then $(\cI,\norm{\cdot}_{\cI})$ is a \textbf{symmetrically normed ideal} of $\cB$, written $(\cI,\norm{\cdot}_{\cI}) \sni \cB$ or, when confusion is unlikely, $\cI \sni \cB$.
\end{definition}

\begin{example}[Closed ideals]\label{ex.closedideals}
If $\cI \subseteq \cB$ is a closed ideal, then $(\cI,\norm{\cdot}) \sni \cB$.
In particular, the \textbf{trivial ideals}, $\cI = \{0\}$ and $\cI = \cB$, are symmetrically normed ideals.
\end{example}

\begin{example}[Schatten $p$-ideals]\label{ex.Schatten}
Let $H$ be a complex Hilbert space, and suppose that $1 \leq p < \infty$.
If $\cS_p(H)$ is the set of Schatten $p$-class operators on $H$ and $\norm{\cdot}_p$ is the Schatten $p$-norm, then $(\cS_p(H),\norm{\cdot}_p) \sni B(H)$.
\end{example}

\begin{example}[Ideals induced by symmetric spaces of measurable operators]\label{ex.NCLp}
Those familiar with noncommutative integration theory (vid.\ \cite{DdPS2023}) will recognize that the previous example generalizes massively.
Please see \cite{Nikitopoulos2023h} and the references therein for the terminology and results necessary to understand this example fully.
\pagebreak

Let $H$ be a complex Hilbert space, $\cM \subseteq B(H)$ be a von Neumann algebra, and $\tau$ be a faithful, normal, semifinite trace on $\cM$.
If $1 \leq p < \infty$ and $\cL^p(\tau) \coloneqq L^p(\tau) \cap \cM$ is the set of all $a \in \cM$ such that $\tau(|a|^p) < \infty$ and $\norm{\cdot}_p$ is the noncommutative $L^p$ norm associated with $\tau$, i.e., $\norm{a}_p = \tau(|a|^p)^{1/p}$, then
\[
\left(\cL^p(\tau),\max\big\{\norm{\cdot},\norm{\cdot}_p\big\}\right) \sni \cM.
\]
The previous example is the special case in which $\cM = B(H)$ and $\tau = \Tr$.

Even the previous paragraph generalizes.
In the same setting, if $(E,\norm{\cdot}_E)$ is a symmetric space of $\tau$-measurable operators, then \cite[Prop.\ 17]{DdP2014} implies that
\[
\big(\cE,\norm{\cdot}_{\cE}\big) \coloneqq \left(E \cap \cM, \max\big\{ \norm{\cdot}_E,\norm{\cdot}\big\}\right) \sni \cM.
\]
The previous paragraph is the special case in which $(E,\norm{\cdot}_E) = (L^p(\tau),\norm{\cdot}_p)$.
\end{example}

Beware:
There are many definitions of a symmetrically normed ideal (or related objects) in the literature.
Sometimes, it is required that $C_{\cI} = 1$.
Sometimes, $\cB$ is required to be a von Neumann or $C^*$-algebra, and $\cI$ is required to be a $\ast$-ideal with $\norm{r^*}_{\cI} = \norm{r}_{\cI}$ for all $r \in \cI$.
Sometimes, even more requirements are imposed.
I adopt the minimal definition above because it is the least restrictive.
It is worth noting, at least parenthetically, that due to the following fact, some of the aforementioned definitions are equivalent to mine.

\begin{proposition}[Ideals of von Neumann algebras]
Let $\cM$ be a von Neumann algebra, $\cI \subseteq \cM$ be an ideal of $\cM$, and $r,s \in \cM$.
\begin{enumerate}[label=(\roman*),font=\normalfont]
    \item $r \in \cI \iff r^* \in \cI \iff |r| \in \cI$.
    In particular, $\cI$ is a $\ast$-ideal of $\cM$.\label{item.star1}
    \item If $s \in \cI$ and $|r| \leq |s|$, then $r \in \cI$.\label{item.down1}
\end{enumerate}
Suppose, in addition, that $\norm{\cdot}_{\cI}$ is a norm on $\cI$ such that $\norm{atb}_{\cI} \leq \norm{a}\,\norm{t}_{\cI} \norm{b}$ whenever $t \in \cI$ and $a,b \in \cM$.
\begin{enumerate}[label=(\roman*),font=\normalfont]
\setcounter{enumi}{2}
    \item If $r \in \cI$, then $\norm{r}_{\cI} = \norm{r^*}_{\cI} = \norm{|r|}_{\cI}$.\label{item.star2}
    \item If $s \in \cI$ and $|r| \leq |s|$, then $\norm{r}_{\cI} \leq \norm{s}_{\cI}$.\label{item.down2}
\end{enumerate}
\end{proposition}

\begin{proof}
For the first and third items, let $r=u|r|$ be the polar decomposition of $r$, and recall that $|r| = u^*r$ as well.
The fact that $r \in \cM$ implies that $u,|r| \in \cM$ as well.
Consequently, if $r \in \cI$, then $r^* = |r|u^* = u^*ru^* \in \cI$ since $\cI$ is an ideal.
Now, if $r^* \in \cI$, then $|r| = |r|^* = (u^*r)^* = r^*u \in \cI$ since $\cI$ is an ideal.
Finally, if $|r| \in \cI$, then $r = u|r| \in \cI$ since $\cI$ is an ideal.
This takes care of the first item.
For the third, note that if $r \in \cI$, then
\begin{align*}
    \norm{r^*}_{\cI} & = \norm{u^*ru^*}_{\cI} \leq \norm{u^*}\,\norm{r}_{\cI}\norm{u^*} \leq \norm{r}_{\cI} = \norm{u|r|}_{\cI} \\
    & \leq \norm{u}\,\norm{|r|}_{\cI} \leq \norm{|r|}_{\cI} = \norm{r^*u}_{\cI} \leq \norm{r^*}_{\cI} \norm{u} \leq \norm{r^*}_{\cI},
\end{align*}
which yields the desired result.

For the second and fourth items, note that it suffices (by the other items) to assume $r,s \geq 0$ so that $r = |r|$ and $s = |s|$.
By (the proof of) \cite[Pt.\ I, Lem.\ 1.2]{Dixmier1981}, if $0 \leq r \leq s$, then there exists a $c \in \cM$ such that $\norm{c} \leq 1$ and $\sqrt{r}=c\sqrt{s}$.
In particular, if $s \in \cI$, then
\[
r = \sqrt{r}\big(\sqrt{r}\big)^* = c\sqrt{s}\big(c\sqrt{s}\big)^* = csc^* \in \cI
\]
because $\cI$ is an ideal.
This takes care of the second item.
Continuing for the fourth item, it follows that
\[
\norm{r}_{\cI} = \norm{csc^*}_{\cI} \leq \norm{c}\,\norm{s}_{\cI}\norm{c^*} \leq \norm{s}_{\cI},
\]
as desired.
\end{proof}

Consequently, if $\cB = \cM$ is a von Neumann algebra, the definitions of an invariant operator ideal of $\cM$ in \cite{ACDS2009} and a symmetrically normed ideal of $\cM$ in \cite{ST2019} are equivalent, up to a constant multiple of the ideal's norm, to the definition of a symmetrically normed ideal of $\cB = \cM$ in Definition \ref{def.sni}.

To wrap up this brief subsection, I present a basic property of symmetrically normed ideals that will be repeatedly useful later, namely, that symmetrically normed ideals play well with the $\sh$ operation defined in the paragraph after Theorem \ref{thm.HFCintro}.

\begin{proposition}\label{prop.hashonI}
Let $(\cI,\norm{\cdot}_{\cI}) \sni \cB$, and fix $k \in \N$.
If $u \in \cB^{\potimes(k+1)}$, $i =1,\ldots,k$, and $b = (b_1,\ldots,b_k) \in \cB^{i-1} \times \cI \times \cB^{k-i}$, then $u\sh b \in \cI$, and
\[
\norm{u\sh b}_{\cI} \leq \norm{u}_{\cB^{\potimes(k+1)}}\norm{b_i}_{\cI}\prod_{j \neq i} \norm{b_j}.
\]
In particular, if $b = (b_1,\ldots,b_k) \in \cI^k$, then
\[
\norm{u\sh b}_{\cI} \leq C_{\cI}^{k-1}\norm{u}_{\cB^{\potimes(k+1)}}\prod_{j = 1}^k \norm{b_j}_{\cI}.
\]
\end{proposition}

\begin{proof}
Let $b = (b_1,\ldots,b_k) \in \cB^{i-1} \times \cI \times \cB^{k-i}$.
By definition of a symmetrically normed ideal, if $a_1,\ldots,a_{k+1} \in \cB$ and $u \coloneqq a_1\otimes \cdots \otimes a_{k+1}$, then
\begin{align*}
    \norm{u\sh b}_{\cI} & = \norm{a_1b_1\cdots a_kb_ka_{k+1}}_{\cI} \\
    & \leq \norm{a_1b_1\cdots a_{i-1}b_{i-1}a_i}\,\norm{b_i}_{\cI} \norm{a_{i+1}b_{i+1}\cdots a_k b_ka_{k+1}} \\
    & \leq \norm{a_1}\cdots\norm{a_{k+1}}\,\norm{b_i}_{\cI}\prod_{j \neq i} \norm{b_j}.
\end{align*}
The result then follows via standard arguments from the universal property of the projective tensor product and the continuity of $\iota_{\scI} \colon (\cI, \norm{\cdot}_{\cI}) \to (\cB,\norm{\cdot})$.
\end{proof}

\section{Lexicographic functional calculus (LFC)}\label{sec.LFC}

For the duration of this section, let $\alpha$ be a family of possibly infinite norms as in \eqref{eq.alphaintro}.

\subsection{The spaces \texorpdfstring{$\cC_{\alpha}(\R^m)$}{} and \texorpdfstring{$C_{\alpha}(\R^m)$}{}}\label{subsec.Calpha}

To study and construct LFC, it is helpful to develop some terminology for properties the family $\alpha$ can have and see how those properties affect the spaces $\cC_{\alpha}(\R^m)$ and $C_{\alpha}(\R^m)$.
Before proceeding, the reader should review the material from subsection \ref{subsec.results} through Example \ref{ex.Varopoulosfamily}.

\begin{notation}\label{nota.norms}
Let $m \in \N$.
Write $\alpha_{\cdot,m} \coloneqq (\alpha_{r,m} \colon C([-r,r]^m) \to [0,\infty])_{r > 0}$.
Also, if $r > 0$ and $\varphi$ is a complex-valued continuous function whose domain contains $[-r,r]^m$, write $\alpha_{r,m}(\varphi) \coloneqq \alpha_{r,m}(\varphi|_{[-r,r]^m})$.
Finally, write $|\blambda|_{\infty} \coloneqq \max\left\{ |\lambda_i| : i=1,\ldots,m\right\}$ for $\blambda \in \R^m$.
\end{notation}
\pagebreak

\begin{definition}\label{def.norms1}
The family $\alpha$ is
\begin{enumerate}[label=(\roman*),font=\normalfont]
    \item \textbf{complete} if $(C([-r,r]^m)_{\alpha_{r,m}},\alpha_{r,m}|_{<\infty})$ is a Banach space for all $r > 0$ and $m \in \N$;\label{item.complete}
    \item \textbf{increasing} if
    \[
    \alpha_{s,m}(\varphi|_{[-s,s]^m}) \leq \alpha_{r,m}(\varphi)
    \]
    whenever $0 < s \leq r$, $m \in \N$, and $\varphi \in C([-r,r]^m)$;\label{item.increasing}
    \item \textbf{submultiplicative} if
    \[
    \alpha_{r,m}(\varphi\,\psi) \leq \alpha_{r,m}(\varphi)\,\alpha_{r,m}(\psi)
    \]
    for all $r > 0$, $m \in \N$, and $\varphi, \psi \in C([-r,r]^m)$;\label{item.submultiplicative}
    \item \textbf{$\boldsymbol{\ast}$-isometric} if
    \[
    \alpha_{r,m}\big(\overline{\varphi}\big) = \alpha_{r,m}(\varphi)
    \]
    for all $r > 0$, $m \in \N$, and $\varphi \in C([-r,r]^m)$;\label{item.starisometric}
    \item \textbf{translation invariant} if
    \[
    \alpha_{r,m}\big(\varphi(\cdot+\bmu)|_{[-r,r]^m}\big) \leq \alpha_{r+|\bmu|_{\infty},m}(\varphi)
    \]
    for all $m \in \N$, $\bmu \in \R^m$, and $\varphi \in C([-|\bmu|_{\infty}-r,r+|\bmu|_{\infty}]^m)$;\label{item.translationinvariant}
    \item \textbf{compatible} if
    \[
    \alpha_{r,m+n}(\varphi \otimes \psi) \leq \alpha_{r,m}(\varphi)\,\alpha_{r,n}(\psi)
    \]
    for all $r > 0$, $m,n \in \N$, $\varphi \in C([-r,r]^m)$, and $\psi \in C([-r,r]^n)$; and\label{item.compatible}
    \item \textbf{consistent} if for all $m \in \N$, $r > 0$, $i=1,\ldots,m$, and $\varphi \in C([-r,r]^m)$,
    \[
    \alpha_{r,m+1}(\tilde{\varphi}_i) = \alpha_{r,m}(\varphi),
    \]
    where $\tilde{\varphi}_i(\blambda) \coloneqq \varphi(\lambda_1,\ldots,\lambda_{i-1},\lambda_{i+1},\ldots,\lambda_{m+1})$ for all $\blambda \in [-r,r]^{m+1}$.\label{item.consistent}
\end{enumerate}
\end{definition}

\begin{example}[Uniform and Varopoulos families]\label{ex.uniformandVaropoulos}
The uniform family $\alpha = u$ (Example~\ref{ex.uniformfamily}) and the Varopoulos family $\alpha = \beta$ (Example \ref{ex.Varopoulosfamily}) both satisfy properties \ref{item.complete}--\ref{item.consistent} from Definition \ref{def.norms1}, as I encourage the reader to verify.
\end{example}

\begin{example}[$\cA$-Varopoulos family]\label{ex.betaA}
Let $\cA$ be a unital $C^*$-algebra.
Later, an $\cA$-specific adjustment of the Varopoulos family will come in handy.
For all $r > 0$, $m \in \N$, and $\varphi \in C([-r,r]^m)$, define
\[
\beta_{r,m}^{\scA}(\varphi) \coloneqq \sup_{\mathbf{a} \in \cA_{\sa,r}^m} \norm{\varphi|_{\sigma(a_1) \times \cdots \times \sigma(a_m)}}_{V(\sigma(a_1),\ldots,\sigma(a_m))} \in [0,\infty].
\]
(Recall that $\cA_{\sa,r} = \{a \in \cA_{\sa} : \norm{a} \leq r\}$.)
The family
\[
\beta^{\scA} \coloneqq \big(\beta_{r,m}^{\scA} \colon C([-r,r]^m) \to [0,\infty]\big)_{r > 0, \, m \in \N}
\]
of possibly infinite norms is the \textbf{$\boldsymbol{\cA}$-Varopoulos family}.
I encourage the reader to verify that $\alpha = \beta^{\scA}$ satisfies properties \ref{item.complete}--\ref{item.consistent} from Definition \ref{def.norms1}.
Only translation invariance is slightly tricky to verify;
the key fact is that $\sigma(a+\mu1) = \sigma(a)+\mu$ for all $a \in \cA$ and $\mu \in \C$.
Also, observe that $\beta^{\mathsmaller{C([0,1])}} = \beta$ because for each $r > 0$, there exists an $a \in C([0,1])$ such that $\sigma(a) = [-r,r]$.
(Take, e.g., $a(x) \coloneqq r(2x-1)$ for all $x \in [0,1]$.)
\end{example}

Recall that $\cC_{\alpha}(\R^m)$ is a Hausdorff locally convex topological vector space (LCTVS).
Next, we go through some additional properties of the spaces $\cC_{\alpha}(\R^m)$ and $C_{\alpha}(\R^m)$.
Below, $\hookrightarrow$ indicates continuous inclusion.
Also, recall that $C(\R^m) = \cC_u(\R^m) = C_u(\R^m)$ has the topology of locally uniform convergence.

\begin{definition}\label{def.stronger}
For $i=1,2$, let $\alpha^i = (\alpha_{r,m}^i \colon C([-r,r]^m) \to [0,\infty])_{r > 0, \, m \in \N}$ be a family of possibly infinite norms.
Write $\alpha^1 \leq \alpha^2$ whenever $\alpha_{r,m}^1 \leq \alpha_{r,m}^2$ for all $r > 0$ and $m \in \N$;
in this case, $\alpha^2$ is \textbf{stronger} than $\alpha^1$.
\end{definition}

\begin{example}\label{ex.Varopstrongerthanuniform}
Since $\norm{\cdot}_{\ell^{\infty}(\Om_1 \times \cdots \times \Om_m)} \leq \norm{\cdot}_{V(\Om_1,\ldots,\Om_m)}$ for all locally compact Hausdorff spaces $\Om_1,\ldots,\Om_m$, the Varopoulos family is stronger than the uniform family, i.e., $u \leq \beta$.
More generally, if $\cA$ is a unital $C^*$-algebra, then $u \leq \beta^{\scA} \leq \beta$.
Indeed, that $\beta^{\scA} \leq \beta$ is obvious.
To see $u \leq \beta^{\scA}$, observe that $\sigma(\lambda 1) = \{\lambda\}$ for all $\lambda \in \C$.
Consequently, if $r > 0$, $m \in \N$, $\blambda \in [-r,r]^m$, and $a_i \coloneqq \lambda_i1 \in \cA_{\sa,r}$ for all $i=1,\ldots,m$, then
\[
\norm{\varphi|_{\sigma(a_1) \times \cdots \times \sigma(a_m)}}_{V(\sigma(a_1),\ldots,\sigma(a_m))} = |\varphi(\blambda)|.
\]
It follows that $u \leq \beta^{\scA}$.
\end{example}

\begin{proposition}[Properties of $\cC_{\alpha}(\R^m)$ and $C_{\alpha}(\R^m)$]\label{prop.Calpha}
Let
\[
\tilde{\alpha} = (\tilde{\alpha}_{r,m} \colon C([-r,r]^m) \to [0,\infty])_{r > 0, \, m \in \N}
\]
be another family of possibly infinite norms and $m \in \N$.
\begin{enumerate}[label=(\roman*),font=\normalfont]
    \item If $\tilde{\alpha}$ is stronger than $\alpha$, then $\cC_{\tilde{\alpha}}(\R^m) \hookrightarrow \cC_{\alpha}(\R^m)$.
    If, in addition, $\tilde{\alpha}$ is finite on polynomials, then so is $\alpha$, and $C_{\tilde{\alpha}}(\R^m) \hookrightarrow C_{\alpha}(\R^m)$.\label{item.stronger}
    \item If $\alpha$ is submultiplicative (and finite on polynomials), then $\cC_{\alpha}(\R^m)$ is a complex algebra with a jointly continuous product operation (and $C_{\alpha}(\R^m)$ is a closed subalgebra of $\cC_{\alpha}(\R^m)$).
    If $\alpha$ is also $\ast$-isometric, then $\cC_{\alpha}(\R^m)$ is a $\ast$-algebra with a continuous $\ast$-operation (and $C_{\alpha}(\R^m)$ is a closed $\ast$-subalgebra of $\cC_{\alpha}(\R^m)$).\label{item.submultstarisom}
    \item Suppose $\tilde{\alpha}$ is increasing and $\tilde{\alpha} \leq \alpha$.
    If $\alpha$ is complete, then $\cC_{\alpha}(\R^m)$ is complete.
    In particular, if $u \leq \alpha$ and $\alpha$ is complete, then $\cC_{\alpha}(\R^m)$ is complete.\label{item.compinc}
    \item If $\alpha$ is increasing, then $\cC_{\alpha}(\R^m)$ is metrizable.
    Consequently, by the previous item with $\tilde{\alpha} = \alpha$, if $\alpha$ is increasing and complete, then $\cC_{\alpha}(\R^m)$ is a Fr\'{e}chet space.\label{item.inc}
    \item If $\alpha$ is translation invariant and $\bmu \in \R^m$, then the translation operator
    \[
    \cC_{\alpha}(\R^m) \ni \varphi \mapsto \tau_{\bmu}(\varphi) \coloneqq \varphi(\cdot+\bmu) \in \cC_{\alpha}(\R^m)
    \]
    is well defined, linear, and continuous.
    If, in addition, $\alpha$ is increasing and finite on polynomials, then
    \[
    \R^m \times C_{\alpha}(\R^m) \ni (\bmu,\varphi) \mapsto \tau(\mu,\varphi) \coloneqq \tau_{\bmu}(\varphi) \in C_{\alpha}(\R^m)
    \]
    is well defined and (jointly) continuous.\label{item.transinv}
    \item If $u \leq \alpha$, $\alpha_{\cdot,1} = u_{\cdot,1}$, and $\alpha$ is complete and compatible, then $\alpha \leq \beta$.
    In particular, $\alpha$ is finite on polynomials, and $VC(\R^m) \hookrightarrow C_{\alpha}(\R^m) \hookrightarrow \cC_{\alpha}(\R^m) \hookrightarrow C(\R^m)$.\label{item.compcomp}\pagebreak
    \item For $\cS \subseteq C(\R^m)$, define $\cS_{\loc}$ to be the set of all $\varphi \in C(\R^m)$ such that for all $r > 0$, there exists a $\psi_r \in \cS$ such that $\varphi|_{[-r,r]^m} = \psi_r|_{[-r,r]^m}$.
    If $\cS \subseteq \cC_{\alpha}(\R^m)$, then
    \[
    \cS_{\loc} \subseteq \overline{\cS} \subseteq \cC_{\alpha}(\R^m),
    \]
    where the closure above takes place in the space $\cC_{\alpha}(\R^m)$.\label{item.SlocCalpha}
\end{enumerate}
\end{proposition}

\begin{proof}
The first two items follow readily from the relevant definitions.
I take each remaining item in turn.

\ref{item.compinc} Let $(\varphi_j)_{j \in J}$ be a Cauchy net in $\cC_{\alpha}(\R^m)$.
If $r > 0$, then $(\varphi_j|_{[-r,r]^m})_{j \in J}$ is a Cauchy net in the Banach space $C([-r,r]^m)_{\alpha_{r,m}}$.
Thus, there is some $\psi_r \in C([-r,r]^m)_{\alpha_{r,m}}$ such that $(\varphi_j|_{[-r,r]^m})_{j \in J}$ converges to $\psi_r$ in $C([-r,r]^m)_{\alpha_{r,m}}$.
Now,
\[
\tilde{\alpha}_{r,m}\big(\varphi_j|_{[-r,r]^m} - \psi_r\big) \leq \alpha_{r,m}\big(\varphi_j|_{[-r,r]^m} - \psi_r\big) \xrightarrow{j \in J} 0,
\]
i.e., $(\varphi_j|_{[-r,r]^m})_{j \in J}$ converges to $\psi_r$ in $C([-r,r]^m)_{\tilde{\alpha}_{r,m}}$.
I claim that if $0 < s \leq r$, then $\psi_r|_{[-s,s]^m} = \psi_s$.
Indeed, since $\tilde{\alpha}$ is increasing,
\begin{align*}
    \tilde{\alpha}_{s,m}\big(\psi_r|_{[-s,s]^m} - \psi_s\big) & \leq \tilde{\alpha}_{s,m}\big(\psi_r|_{[-s,s]^m} - \varphi_j|_{[-s,s]^m}\big) + \tilde{\alpha}_{s,m}\big(\varphi_j|_{[-s,s]^m} - \psi_s\big) \\
    & \leq \tilde{\alpha}_{r,m}\big(\psi_r - \varphi_j|_{[-r,r]^m}\big) + \tilde{\alpha}_{s,m}\big(\varphi_j|_{[-s,s]^m} - \psi_s\big) \xrightarrow{j \in J} 0,
\end{align*}
so $\tilde{\alpha}_{s,m}(\psi_r|_{[-s,s]^m} - \psi_s) = 0$.
Since $\tilde{\alpha}_{s,m}$ is a norm, $\psi_r|_{[-s,s]^m} = \psi_s$, as claimed.
Consequently, there exists a unique $\psi \in \cC_{\alpha}(\R^m)$ such that $\psi|_{[-r,r]^m} = \psi_r$ for all $r > 0$.
After unraveling the definitions, it becomes clear that $(\varphi_j)_{j \in J}$ converges in $\cC_{\alpha}(\R^m)$ to $\psi$.
Thus, $\cC_{\alpha}(\R^m)$ is~complete.

\ref{item.inc} Since $\alpha$ is increasing, the topology of $\cC_{\alpha}(\R^m)$ is induced by the countable family
\[
\cC_{\alpha}(\R^m) \ni \varphi \mapsto \alpha_{n,m}(\varphi|_{[-n,n]^m}) \in [0,\infty) \qquad (n \in \N)
\]
of seminorms.
Consequently, by standard arguments,
\[
d_{\alpha}(\varphi,\psi) \coloneqq \sum_{n=1}^{\infty}2^{-n}\frac{\alpha_{n,m}(\varphi-\psi)}{1+\alpha_{n,m}(\varphi-\psi)} \qquad (\varphi,\psi \in \cC_{\alpha}(\R^m))
\]
is a metric on $\cC_{\alpha}(\R^m)$ that induces the topology of $\cC_{\alpha}(\R^m)$.

\ref{item.transinv} The first statement in this item follows easily from the definition of translation invariant.
To prove the second statement, suppose $\alpha$ is finite on polynomials.
Notice that if $\bmu \in \R^m$ and $\varphi \in C_{\alpha}(\R^m)$, then $\tau_{\bmu}(\varphi) \in C_{\alpha}(\R^m)$.
Indeed, if $P \in \cP_m$, then
\[
\tau_{\bmu}(P) = P(\cdot+\bmu) \in \cP_m \subseteq C_{\alpha}(\R^m).
\]
Consequently, if $(P_j)_{j \in J}$ is a net of polynomials converging to $\varphi$ in $\cC_{\alpha}(\R^m)$, then $(\tau_{\bmu}(P_j))_{j \in J}$ is a net in $\cP_m \subseteq C_{\alpha}(\R^m)$ converging to $\tau_{\bmu}(\varphi)$ in $\cC_{\alpha}(\R^m)$.
Thus, $\tau_{\bmu}(\varphi) \in C_{\alpha}(\R^m)$.

Next, I claim that if $P \in \cP_m$, then $\R^m \ni \bmu \mapsto \tau_{\bmu}(P) \in C_{\alpha}(\R^m)$ is continuous.
Indeed, it suffices to take
\[
P(\blambda) = P_{\gamma}(\blambda) \coloneqq \blambda^{\gamma} = \lambda_1^{\gamma_1}\cdots\lambda_m^{\gamma_m}
\]
for some $\gamma  = (\gamma_1,\ldots,\gamma_m)\in \N_0^m$, in which case the binomial theorem yields
\[
P_{\gamma}(\boldsymbol{\lambda}+\bmu) = (\lambda_1+\mu_1)^{\gamma_1}\cdots(\lambda_m+\mu_m)^{\gamma_m} = \sum_{0 \leq \delta \leq \gamma} \binom{\gamma_1}{\delta_1}\cdots\binom{\gamma_m}{\delta_m}\bmu^{\gamma-\delta}\blambda^{\delta},\pagebreak
\]
where $0 \leq \delta \leq \gamma$ means that $0 \leq \delta_i \leq \gamma_i$ for all $i = 1,\ldots,m$.
Consequently, if $r > 0$, then
\[
\alpha_{r,m}(\tau_{\bmu}(P_{\gamma})-\tau_{\boldsymbol{\nu}}(P_{\gamma})) \leq  \sum_{0 \leq \delta \leq \gamma} \binom{\gamma_1}{\delta_1}\cdots\binom{\gamma_m}{\delta_m}\big|\bmu^{\gamma-\delta}-\boldsymbol{\nu}^{\gamma-\delta}\big|\alpha_{r,m}(P_{\delta}) \xrightarrow{\boldsymbol{\nu} \to \bmu} 0
\]
for each $\bmu \in \R^m$, as claimed.

Finally, suppose $\alpha$ is also increasing so that $C_{\alpha}(\R^m)$ is metrizable by the previous item.
Let $(\bmu_n,\varphi_n)_{n \in \N}$ be a sequence in $\R^m \times C_{\alpha}(\R^m)$ converging to $(\bmu,\varphi) \in \R^m \times C_{\alpha}(\R^m)$ and $\e,r > 0$.
Writing $R \coloneqq \sup \left\{|\bmu_n|_{\infty} : n \in \N\right\} < \infty$, let $P \in \cP_m$ be such that
\[
\alpha_{r+R,m}(P-\varphi) < \frac{\e}{6}.
\]
Now, let $N \in \N$ be such that $n \geq N$ implies
\[
\alpha_{r+R,m}(\varphi_n-\varphi) < \frac{\e}{6} \; \text{ and } \; \alpha_{r,m}(\tau_{\bmu_n}(P)-\tau_{\bmu}(P)) < \frac{\e}{3}.
\]
Since $\alpha$ is translation invariant and increasing, if $n \geq N$, then
\begin{align*}
    \alpha_{r,m}(\tau_{\bmu_n}(\varphi_n) - \tau_{\bmu}(\varphi)) & \leq \alpha_{r,m}(\tau_{\bmu_n}(\varphi_n-P)) + \alpha_{r,m}(\tau_{\bmu_n}(P)-\tau_{\bmu}(P)) + \alpha_{r,m}(\tau_{\bmu}(P-\varphi)) \\
    & \leq \alpha_{r+|\bmu_n|_{\infty},m}(\varphi_n-P) + \alpha_{r,m}(\tau_{\bmu_n}(P)-\tau_{\bmu}(P))  + \alpha_{r+|\bmu|_{\infty},m}(P-\varphi) \\
    & \leq \alpha_{r+R,m}(\varphi_n-P) + \alpha_{r,m}(\tau_{\bmu_n}(P)-\tau_{\bmu}(P))  + \alpha_{r+R,m}(P-\varphi) \\
    & \leq \alpha_{r+R,m}(\varphi_n-\varphi) + \alpha_{r,m}(\tau_{\bmu_n}(P)-\tau_{\bmu}(P))  + 2\,\alpha_{r+R,m}(P-\varphi) < \e.
\end{align*}
Thus, $\tau \colon \R^m \times C_{\alpha}(\R^m) \to C_{\alpha}(\R^m)$ is continuous.

\ref{item.compcomp} Let $r > 0$.
By induction and the definition of compatibility, if $\varphi_i \in C([-r,r])$ for each $i=1,\ldots,m$, then
\[
\alpha_{r,m}(\varphi_1 \otimes \cdots \otimes \varphi_m) \leq \prod_{i=1}^m\alpha_{r,1}(\varphi_i) = \prod_{i=1}^m u_{r,1}(\varphi_i) = \prod_{i=1}^m \norm{\varphi_i}_{\ell^{\infty}([-r,r])} < \infty.
\]
Consequently, if $\varphi \in V([-r,r]_{(m)})$ and $\varphi = \sum_{n=1}^{\infty}\varphi_{1,n} \otimes \cdots \otimes \varphi_{m,n}$ is decomposed as in \eqref{eq.Vardecomp}, then
\[
\sum_{n=1}^{\infty}\alpha_{r,m}(\varphi_{1,n} \otimes \cdots \otimes \varphi_{m,n}) \leq \sum_{n=1}^{\infty}\prod_{i=1}^m \norm{\varphi_{i,n}}_{\ell^{\infty}([-r,r])} < \infty.
\]
Since $\alpha$ is complete, $\big(\sum_{n=1}^N \varphi_{1,n} \otimes \cdots \otimes \varphi_{m,n}\big)_{N \in \N}$ converges in $C([-r,r]^m)_{\alpha_{r,m}}$ to some $\psi \in C([-r,r]^m)_{\alpha_{r,m}}$.
Since $u \leq \alpha$ (and therefore $C([-r,r]^m)_{\alpha_{r,m}} \hookrightarrow C([-r,r]^m)$), it follows that $\psi = \varphi$.
Thus,
\[
\alpha_{r,m}(\varphi) = \alpha_{r,m}(\psi) \leq \sum_{n=1}^{\infty}\prod_{i=1}^m \norm{\varphi_{i,n}}_{\ell^{\infty}([-r,r])}.
\]
Taking the infimum over all decompositions of $\varphi$ yields
\[
\alpha_{r,m}(\varphi) \leq \norm{\varphi}_{V([-r,r]_{(m)})} = \beta_{r,m}(\varphi),
\]
as desired.

\ref{item.SlocCalpha} I claim that if $\varphi \in \cS_{\loc}$ and $(\psi_r)_{r > 0}$ is as in the definition of $\cS_{\loc}$, then $\varphi \in \cC_{\alpha}(\R^m)$, and $(\psi_n)_{n \in \N}$ is a sequence in $\cS$ converging in $\cC_{\alpha}(\R^m)$ to $\varphi$, which implies the result.
Indeed, $\varphi \in \cC_{\alpha}(\R^m)$ because for all $r > 0$,
\[
\alpha_{r,m}(\varphi) = \alpha_{r,m}\big(\varphi|_{[-r,r]^m}\big) = \alpha_{r,m}\big(\psi_r|_{[-r,r]^m}\big) = \alpha_{r,m}(\psi_r) < \infty
\]
since $\psi_r \in \cS \subseteq \cC_{\alpha}(\R^m)$.
Now, $(\psi_n)_{n \in \N}$ converges in $\cC_{\alpha}(\R^m)$ to $\varphi$ because if $r > 0$ and $N$ is any natural number larger than $r$, then
\[
\psi_N|_{[-r,r]^m} = \big(\psi_N|_{[-N,N]^m}\big)|_{[-r,r]^m} = \big(\varphi|_{[-N,N]^m}\big)|_{[-r,r]^m} = \varphi|_{[-r,r]^m}.
\]
Consequently, $\alpha_{r,m}(\psi_N - \varphi) = 0$.
In particular, $\alpha_{r,m}(\psi_n-\varphi) \to 0$ as $n \to \infty$.
Thus, $(\psi_n)_{n \in \N}$ converges in $\cC_{\alpha}(\R^m)$ to $\varphi$, as claimed.
\end{proof}

\subsection{Construction of LFC}\label{subsec.constructLFC}

The purpose of this subsection is to construct lexicographic functional calculus (LFC) with respect to our fixed family $\alpha$ of possibly infinite norms.
I begin by introducing a few more properties of such families.

\begin{definition}[Controlling LFC]\label{def.controlLFC}
Let $\cA$ be a unital $C^*$-algebra and $\cI \sni \cA$.
\begin{enumerate}[label=(\roman*),font=\normalfont]
    \item $\alpha$ \textbf{controls polynomial lexicographic functional calculus} (\textbf{LFC}) \textbf{in $\boldsymbol{\cI}$} if it is finite on polynomials, $\alpha_{\cdot,1} = u_{\cdot,1}$, and for all $m \geq 2$, $P \in \cP_m$, and $b \in \cI^{m-1}$,
    \[
    \sup_{\mathbf{a} \in \cA_{\sa,r}^m}\norm{P_{\sotimes}(\mathbf{a})\sh b}_{\cI} \leq C_{\cI}^{m-2}\alpha_{r,m}(P) \prod_{i=1}^{m-1} \norm{b_i}_{\cI} \qquad (r > 0).
    \]
    Recall that $\cA_{\sa,r} = \{a \in \cA_{\sa} : \norm{a} \leq r\}$.\label{item.controlLFCinI}
    \item $\alpha$ \textbf{controls polynomial LFC in $\boldsymbol{\cA}$ with an insertion from $\boldsymbol{\cI}$} if it controls polynomial LFC in $\cA$ and for all $m \geq 2$, $P \in \cP_m$, $i=1,\ldots,m-1$, and $b \in \cA^{i-1} \times \cI \times \cA^{m-1-i}$,
    \[
    \sup_{\mathbf{a} \in \cA_{\sa,r}^m}\norm{P_{\sotimes}(\mathbf{a})\sh b}_{\cI} \leq \alpha_{r,m}(P) \norm{b_i}_{\cI}\prod_{j \neq i} \norm{b_j} \qquad (r > 0).
    \]
    Observe that, by definition of a Banach ideal, if $\alpha$ controls polynomial LFC with an insertion from $\cI$, then $\alpha$ controls polynomial LFC in $\cI$.\label{item.controlLFCinAiwthI}
    \item $\alpha$ is \textbf{ideal for LFC in $\boldsymbol{\cA}$} if it is complete, compatible, increasing, and translation invariant; $u \leq \alpha$; and for every symmetrically normed ideal $(\cJ,\norm{\cdot}_{\cJ})$ of $\cA$, $\alpha$ controls polynomial LFC in $\cA$ with an insertion from $\cJ$.\label{item.idealinA}
    \item $\alpha$ is \textbf{ideal for LFC} if for every unital $C^*$-algebra $\cB$, $\alpha$ is ideal for LFC in $\cB$.\label{item.ideal}
\end{enumerate}
Observe that if $\alpha$ is ideal for LFC in $\cA$, then $u \leq \alpha \leq \beta$ by Proposition \ref{prop.Calpha}\ref{item.compcomp}.
\end{definition}

In subsection \ref{subsec.commLFC}, I prove that if $\cA$ is commutative, then $u$ is ideal for LFC in $\cA$.
In subsection \ref{subsec.VaropLFC}, I prove that $\beta^{\scA}$ is ideal for LFC in $\cA$; if $\cA$ is finite dimensional, then a ``scaled version'' of $u$ is ideal for LFC in $\cA$; and $\beta$ is ideal for LFC.
Later, in subsection \ref{subsec.MOIfam}, I use MOIs to study a more technical example.

Let us now construct LFC.

\begin{notation}\label{nota.Cbb}
If $V$ and $W$ are normed vector spaces over $\F \in \{\R,\C\}$, then $\Cbb(V;W)$ is the space of continuous functions from $V$ to $W$ that are bounded on bounded sets, endowed with the topology of uniform convergence on bounded sets.
\end{notation}

Since the topology of convergence on bounded sets is induced by the countable family
\[
\Cbb(V;W) \ni F \mapsto \sup_{\norm{v}_V \leq n} \norm{F(v)}_W \in \R \qquad (n \in \N)
\]
of seminorms, $\Cbb(V;W)$ is a metrizable LCTVS over $\F$.
By standard arguments, if $W$ is a Banach space, then $\Cbb(V;W)$ is a Fr\'echet space over $\F$.

To motivate Theorem \ref{thm.LFC} below, the actual construction of LFC, observe that if $\cA$ is a unital $C^*$-algebra, then the polynomial functional calculus map
\[
\cP \ni p \mapsto \ev_{\scA}(p) \coloneqq p_{\scA} \in \Cbb(\cA_{\sa};\cA)
\]
is well defined and linear.
Furthermore, if $p \in \cP$, then
\begin{equation}
    \sup_{a \in \cA_{\sa,r}}\norm{\ev_{\scA}(p)(a)} = \sup_{a \in \cA_{\sa,r}}\norm{p(a)} = \sup_{a \in \cA_{\sa,r}}\norm{p}_{\ell^{\infty}(\sigma(a))} = \|p\|_{\ell^{\infty}([-r,r])} \qquad (r > 0).\label{eq.1dimcase}
\end{equation}
Since $\cP$ is dense in $C(\R)$ by the Stone--Weierstrass theorem and $\Cbb(\cA_{\sa};\cA)$ is a Fr\'echet space, $\ev_{\scA}$ extends uniquely to a continuous linear map from $C(\R)$ to $\Cbb(\cA_{\sa};\cA)$.
This extension is precisely the continuous functional calculus map $C(\R) \ni f \mapsto f_{\scA} \in \Cbb(\cA_{\sa};\cA)$.

\begin{lemma}[Continuity of polynomial LFC]\label{lem.pLFCcont}
Let $\cA$ be a unital $C^*$-algebra and $\cI \sni \cA$.
If $m \geq 2$, $P \in \cP_m$, and $i=1,\ldots,m-1$, then the function
\[
\ev_{\scI}^{m,i}(P) \colon \cA_{\sa}^m \to B_{m-1}(\cA^{i-1} \times \cI \times \cA^{m-1-i};\cI)
\]
defined by
\[
\ev_{\scI}^{m,i}(P)(\mathbf{a})[b] \coloneqq P_{\sotimes}(\mathbf{a})\sh b \qquad \big( \mathbf{a} \in \cA_{\sa}^m, \; b \in \cA^{i-1} \times \cI \times \cA^{m-1-i}\big)
\]
belongs to $\Cbb(\cA_{\sa}^m;B_{m-1}(\cA^{i-1} \times \cI \times \cA^{m-1-i};\cI))$.
Since $\iota_{\scI} \colon (\cI,\norm{\cdot}_{\cI}) \to (\cA,\norm{\cdot})$ is continuous, it follows that the function
\[
\ev_{\scI}^m(P) \colon \cA_{\sa}^m \to B_{m-1}(\cI^{m-1};\cI)
\]
defined by
\[
\ev_{\scI}^m(P)(\mathbf{a})[b] \coloneqq P_{\sotimes}(\mathbf{a})\sh b \qquad \big( \mathbf{a} \in \cA_{\sa}^m, \; b \in \cI^{m-1}\big)
\]
belongs to $\Cbb(\cA_{\sa}^m;B_{m-1}(\cI^{m-1};\cI))$.
\end{lemma}

I leave the proof of Lemma \ref{lem.pLFCcont} to the reader.

\begin{theorem}[Construction of LFC]\label{thm.LFC}
Let $\cA$ be a unital $C^*$-algebra, $\cI \sni \cA$, and $m \geq 2$.
\begin{enumerate}[font=\normalfont,label=(\roman*)]
    \item Suppose $\alpha$ controls polynomial LFC in $\cI$.
    The map
    \[
    \cP_m \ni P \mapsto \ev_{\scI}^m(P) \in \Cbb\big(\scA_{\sa}^m;B_{m-1}\big(\cI^{m-1};\cI\big)\big),
    \]
    where $\ev_{\scI}^m(P)$ is as in Lemma \ref{lem.pLFCcont}, extends uniquely to a continuous linear map $\Phi_{\scI,\alpha}^m \colon C_{\alpha}(\R^m) \to \Cbb(\cA_{\sa}^m;B_{m-1}(\cI^{m-1};\cI))$ called the \textbf{$\boldsymbol{m}$-variate $\boldsymbol{\alpha}$-lexicographic functional calculus in $\boldsymbol{\cI}$}---$m$-variate $\alpha$-LFC in $\cI$ for short.\label{item.LFCinI}
    \item Suppose $\alpha$ controls polynomial LFC in $\cA$ with an insertion from $\cI$.
    In addition, fix an $i=1,\ldots,m-1$.
    The map
    \[
    \cP_m \ni P \mapsto \ev_{\scI}^{m,i}(P) \in \Cbb\big(\scA_{\sa}^m;B_{m-1}\big(\cA^{i-1} \times \cI \times \cA^{m-1-i};\cI\big)\big),
    \]
    where $\ev_{\scI}^{m,i}(P)$ is as in Lemma \ref{lem.pLFCcont}, extends uniquely to a continuous linear map $\Phi_{\scI,\alpha}^{m,i} \colon C_{\alpha}(\R^m) \to \Cbb(\cA_{\sa}^m;B_{m-1}(\cA^{i-1} \times \cI \times \cA^{m-1-i};\cI))$.
    Furthermore,
    \begin{equation}
        \Phi_{\scI,\alpha}^{m,i}(\varphi)(\mathbf{a})[b] = \Phi_{\scI,\alpha}^m(\varphi)(\mathbf{a})[b] = \Phi_{\scA,\alpha}^m(\varphi)(\mathbf{a})[b]\label{eq.agree}
    \end{equation}
    for all $\varphi \in C_{\alpha}(\R^m)$, $\mathbf{a} \in \cA_{\sa}^m$, and $b \in \cI^{m-1} \subseteq \cA^{i-1} \times \cI \times \cA^{m-1-i} \subseteq \cA^{m-1}$.\label{item.LFCinAwI}
\end{enumerate}
\end{theorem}

\begin{proof}
For the first item, note that by definition, if $\alpha$ controls polynomial LFC in $\cI$ and $P \in \cP_m$, then
\[
\sup_{\mathbf{a} \in \cA_{\sa,r}^m}\norm{\ev_{\scI}^m(P)(\mathbf{a})}_{B_{m-1}(\cI^{m-1};\cI)} \leq C_{\cI}^{m-2}\alpha_{r,m}(P) \qquad (r > 0).
\]
In particular, the linear map $\ev_{\scI}^m \colon \cP_m \to \Cbb(\cA_{\sa}^m;B_{m-1}(\cI^{m-1};\cI))$ is continuous when $\cP_m \subseteq C_{\alpha}(\R^m)$ is endowed with the subspace topology.
Recall that $C_{\alpha}(\R^m)$ is a Hausdorff LCTVS and $\Cbb(\cA_{\sa}^m;B_{m-1}(\cI^{m-1};\cI))$ is a Fr\'echet space---in particular, a complete Hausdorff LCTVS.
Consequently, the desired conclusion follows from \cite[Prop.\ 5.5]{Treves1967} (continuous linear maps are uniformly continuous), \cite[Thm.\ 5.1]{Treves1967} (uniformly continuous maps extend continuously to the closure), and the fact that $\cP_m$ is dense in $C_{\alpha}(\R^m)$.

For the second, by definition, if $\alpha$ controls polynomial LFC in $\cA$ with an insertion from $\cI$ and $P \in \cP_m$, then
\[
\sup_{\mathbf{a} \in \cA_{\sa,r}^m}\norm{\ev_{\scI}^{m,i}(P)(\mathbf{a})}_{B_{m-1}(\cA^{i-1} \times \cI \times \cA^{m-1-i};\cI)} \leq \alpha_{r,m}(P) \qquad (r > 0).
\]
By the same results from \cite{Treves1967} cited in the previous paragraph, $\ev_{\scI}^{m,i}$ extends uniquely to a continuous linear map $\Phi_{\scI,\alpha}^{m,i} \colon C_{\alpha}(\R^m) \to \Cbb(\cA_{\sa}^m;B_{m-1}(\cA^{i-1} \times \cI \times \cA^{m-1-i};\cI))$.
To complete the proof, note that
\[
\cS \coloneqq \left\{\varphi \in C_{\alpha}(\R^m) : \text{\eqref{eq.agree} holds for all } \mathbf{a} \in \cA_{\sa}^m \text{ and } b \in \cI^{m-1}\right\}
\]
contains $\cP_m$ and is closed in $C_{\alpha}(\R^m)$ by the continuity properties of $\iota_{\scI}$, $\Phi_{\scA,\alpha}^m$, $\Phi_{\scI,\alpha}^m$, and $\Phi_{\scI,\alpha}^{m,i}$.
Since $\cP_m$ is dense in $C_{\alpha}(\R^m)$, $\cS = C_{\alpha}(\R^m)$, as desired.
\end{proof}

Note that the proof above makes clear that if $\alpha$ controls polynomial LFC in $\cI$ and $\varphi \in C_{\alpha}(\R^m)$, then
\[
\sup_{\mathbf{a} \in \cA_{\sa,r}^m} \norm{\Phi_{\scI,\alpha}^m(\varphi)(\mathbf{a})}_{B_{m-1}(\cI^{m-1};\cI)} \leq C_{\cI}^{m-2}\alpha_{r,m}(\varphi) \qquad (r > 0).
\]
Also, if $\alpha$ controls polynomial LFC in $\cA$ with an insertion from $\cI$ and $\varphi \in C_{\alpha}(\R^m)$, then
\[
\max_{i=1,\ldots,m-1}\sup_{\mathbf{a} \in \cA_{\sa,r}^m} \norm{\Phi_{\scI,\alpha}^m(\varphi)(\mathbf{a})}_{B_{m-1}(\cA^{i-1} \times \cI \times \cA^{m-1-i};\cI)} \leq \alpha_{r,m}(\varphi) \qquad (r > 0).
\]
Such estimates will be essential in subsection \ref{subsec.derivatives}, specifically, in the proof of Theorem \ref{thm.derform}.

\begin{notation}[LFC]\label{nota.LFC}
Let $\cA$ be a unital $C^*$-algebra and $\cI \sni \cA$.
When $\alpha$ controls LFC in $\cI$ and $m \geq 2$, write
\[
\varphi_{\scI,\alpha} \coloneqq \Phi_{\scI,\alpha}^m(\varphi) \in \Cbb\big(\cA_{\sa}^m;B_{m-1}\big(\cI^{m-1};\cI\big)\big) \; \text{ and } \; \varphi_{\scI,\alpha}(\mathbf{a}) \sh b \coloneqq \Phi_{\scI,\alpha}^m(\varphi)(\mathbf{a})[b] \in \cI
\]
for all $\varphi \in C_{\alpha}(\R^m)$, $\mathbf{a} \in \cA_{\sa}^m$, and $b \in \cI^{m-1}$.
\end{notation}

In general, the $\sh$ symbol above is purely formal;
it serves as a reminder of how we constructed LFC and what it means.
However, when $\alpha = \beta^{\scA}$ is the $\cA$-Varopoulos family, $\varphi_{\scI,\beta^{\scA}}(\mathbf{a})\sh b$ will truly be expressible as some object $\varphi_{\sotimes}(\mathbf{a}) \in \cA^{\potimes m}$ acting via the actual $\sh$ operation on $b$;
please see subsection \ref{subsec.VaropLFC} for the details.

Also, suppose $\tilde{\alpha} = (\tilde{\alpha}_{r,m} \colon C([-r,r]^m) \to [0,\infty])_{r > 0, \, m \in \N}$ is a family of possibly infinite norms stronger than $\alpha$.
Observe that if $\alpha$ controls polynomial LFC in $\cI$ (or controls polynomial LFC in $\cA$ with an insertion from $\cI$), then so does $\tilde{\alpha}$.
Moreover, by definition of LFC, $\varphi_{\scI,\tilde{\alpha}} = \varphi_{\scI,\alpha}$ for all $\varphi \in C_{\tilde{\alpha}}(\R^m) \subseteq C_{\alpha}(\R^m)$ because the identity on $\cP_m$ extends to a continuous linear map---the inclusion---from $C_{\tilde{\alpha}}(\R^m)$ to $C_{\alpha}(\R^m)$.

\subsection{The commutative case}\label{subsec.commLFC}

The first example of LFC we shall examine is the somewhat trivial case, serving as a sanity check, of $u$-LFC in a commutative unital $C^*$-algebra.
To begin, I show that in the case of $m \geq 2$ variables, identity \eqref{eq.1dimcase} becomes an inequality that is unfavorable from the point of view of constructing LFC unless the $C^*$-algebra in question is commutative.
Actually, the following proposition is essentially the whole reason $\cA$-specific conditions are necessary to construct LFC-type functional calculi.
(If necessary, please review the notation in Lemma \ref{lem.pLFCcont} at this time.)

\begin{proposition}\label{prop.ellinflbd}
Let $\cA$ be a unital $C^*$-algebra, $m \geq 2$, and $i=1,\ldots,m-1$.
If $P \in \cP_m$ and $\{0\} \neq \cI \sni \cA$, then
\[
u_{r,m}(P) = \norm{P}_{\ell^{\infty}([-r,r]^m)} \leq \sup_{\mathbf{a} \in \cA_{\sa,r}^m}\norm{\ev_{\scI}^{m,i}(P)(\mathbf{a})}_{B_{m-1}(\cA^{i-1} \times \cI \times \cA^{m-1-i};\cI)} \qquad (r > 0)
\]
with equality if $\cA$ is commutative.
\end{proposition}

\begin{proof}
Let $r > 0$, and suppose $P(\blambda) = \sum_{|\delta| \leq d} c_{\delta}\, \blambda^{\delta} \in \cP_m$.
Also, let $\bmu \in [-r,r]^m$ and $c \in \cI$, and define 
\[
\mathbf{a} \coloneqq (\mu_11,\ldots,\mu_m1) \in \cA_{\sa,r}^m \; \text{ and } \; b \coloneqq (\underbrace{1,\ldots,1}_{\mathsmaller{i-1 \text{ times}}},c,\underbrace{1,\ldots,1}_{\mathsmaller{m-1-i \text{ times}}}) \in \cA^{i-1} \times \cI \times \cA^{m-1-i}.
\]
Then
\[
\ev_{\scI}^{m,i}(P)(\mathbf{a})[b] = \sum_{|\delta| \leq d} c_{\delta} \, (\mu_11)^{\delta_1}b_1\cdots (\mu_{m-1}1)^{\delta_{m-1}}b_{m-1}(\mu_m1)^{\delta_m} = P(\bmu) \,c.
\]
In particular, if $c$ is chosen so that $\norm{c}_{\cI} = 1$, which is possible if $\cI$ is not zero, then
\[
|P(\bmu)| = \norm{P(\bmu)\,c}_{\cI} \leq \sup_{\tilde{\mathbf{a}} \in \cA_{\sa,r}^m}\norm{\ev_{\scI}^{m,i}(P)(\tilde{\mathbf{a}})}_{B_{m-1}(\cA^{i-1} \times \cI \times \cA^{m-1-i};\cI)}.
\]
Taking the supremum over all $\bmu \in [-r,r]^m$ gives the desired estimate.
\pagebreak

Now, assume $\cA$ is commutative.
If $\mathbf{a} \in \cA_{\sa,r}^m$ and $b \in \cA^{i-1} \times \cI \times \cA^{m-1-i}$, then
\begin{align*}
    \ev_{\scI}^{m,i}(P)(\mathbf{a})[b] & = \sum_{|\delta| \leq d} c_{\delta} \, a_1^{\delta_1}b_1\cdots a_{m-1}^{\delta_{m-1}}b_{m-1}a_m^{\delta_m} \\
    & = \sum_{|\delta| \leq d} c_{\delta} \, a_1^{\delta_1}\cdots a_m^{\delta_m} \,b_1\cdots b_{m-1} \\
    & = P(\mathbf{a})\,b_1\cdots b_{m-1},
\end{align*}
where $P(\mathbf{a})$ is defined via the continuous functional calculus for $m$-tuples of commuting self-adjoint elements.
Thus,
\begin{align*}
    \norm{\ev_{\scI}^{m,i}(P)(\mathbf{a})[b]}_{\cI} & = \norm{P(\mathbf{a})\,b_1\cdots b_{m-1}}_{\cI} \\
    & \leq \norm{P(\mathbf{a})}\norm{b_i}_{\cI}\prod_{j \neq i} \norm{b_j} \\
    & = \norm{P}_{\ell^{\infty}(\sigma(\mathbf{a}))}\norm{b_i}_{\cI}\prod_{j \neq i} \norm{b_j}
\end{align*}
because $(\cI,\norm{\cdot}_{\cI})$ is a symmetrically normed ideal of $\cA$.
To be clear,
\[
\sigma(\mathbf{a}) \subseteq \sigma(a_1) \times \cdots \times \sigma(a_m) \subseteq [-r,r]^m
\]
is the joint spectrum of the $m$-tuple $\mathbf{a} = (a_1,\ldots,a_m)$.
It follows that
\[
\norm{\ev_{\scI}^{m,i}(P)(\mathbf{a})}_{B_{m-1}(\cA^{i-1} \times \cI \times \cA^{m-1-i};\cI)} \leq \norm{P}_{\ell^{\infty}(\sigma(\mathbf{a}))} \leq \norm{P}_{\ell^{\infty}([-r,r]^m)}.
\]
Taking the supremum over all $\mathbf{a} \in \cA_{\sa,r}^m$ completes the proof.
\end{proof}

This proposition and its proof yield a full description of $u$-LFC in the commutative case.

\begin{lemma}\label{lem.phiAcont}
Let $\cA$ be a commutative unital $C^*$-algebra and $\varphi \in C(\R^m)$.
If $\varphi_{\scA} \colon \cA_{\sa}^m \to \cA$ is the map $\mathbf{a} \mapsto \varphi(\mathbf{a})$ defined via the continuous functional calculus for $m$-tuples of commuting self-adjoint elements, then $\varphi_{\scA} \in \Cbb(\cA_{\sa}^m;\cA)$.
\end{lemma}

\begin{proof}[Sketch of proof]
The conclusion is obvious if $\varphi = P \in \cP_m$ is a polynomial.
In general, by the Stone--Weierstrass theorem, there exists a sequence $(P_n)_{n \in \N}$ in $\cP_m$ converging locally uniformly to $\varphi$.
Since $\norm{\psi(\mathbf{a})} = \norm{\psi}_{\ell^{\infty}(\sigma(\mathbf{a}))}$ for all $\mathbf{a} \in \cA_{\sa}^m$ and $\psi \in C(\sigma(\mathbf{a}))$, the sequence $((P_n)_{\scA})_{n \in \N}$ converges uniformly on bounded sets to $\varphi_{\scA}$.
The result follows.
\end{proof}

\begin{theorem}[LFC in the commutative case]\label{thm.commLFC}
If $\cA$ is a commutative unital $C^*$-algebra, then the uniform family $u$ is ideal for LFC in $\cA$.
Furthermore, if $\cI \sni \cA$, $m \geq 2$, and $\varphi \in C(\R^m) = C_u(\R^m)$, then
\begin{equation}
    \varphi_{\scI,u}(\mathbf{a}) \sh b = \varphi(\mathbf{a}) \,b_1\cdots b_{m-1} \qquad \big(\mathbf{a} \in \cA_{\sa}^m, \; b \in \cI^{m-1}\big),\label{eq.commLFC}
\end{equation}
where $\varphi(\mathbf{a}) \in \cA$ is defined via the continuous functional calculus for $m$-tuples of commuting self-adjoint elements.
\end{theorem}

\begin{proof}
Proposition \ref{prop.ellinflbd} takes care of the only nontrivial part of the verification that $u$ is ideal for LFC in $\cA$.
It therefore remains to verify \eqref{eq.commLFC}.
To this end, define a linear map $\Phi \colon C(\R^m) \to \Cbb(\cA_{\sa}^m;B_{m-1}(\cI^{m-1};\cI))$ by
\[
\Phi(\varphi)(\mathbf{a})[b] \coloneqq \varphi(\mathbf{a})\,b_1\cdots b_{m-1} \qquad \big(\varphi \in C(\R^m), \; \mathbf{a} \in \cA_{\sa}^m, \; b \in \cI^{m-1}\big).
\]
(Lemma \ref{lem.phiAcont} ensures $\Phi$ is well defined.)
Note that $\Phi$ is continuous because for all $r > 0$,
\begin{align*}
    \sup_{\mathbf{a} \in \cA_{\sa,r}^m} \norm{\Phi(\varphi)(\mathbf{a})}_{B_{m-1}(\cI^{m-1};\cI)} & \leq C_{\cI}^{m-2} \sup_{\mathbf{a} \in \cA_{\sa,r}^m} \norm{\varphi(\mathbf{a})} = C_{\cI}^{m-2}\sup_{\mathbf{a} \in \cA_{\sa,r}^m} \norm{\varphi}_{\ell^{\infty}(\sigma(\mathbf{a}))} \\
    & = C_{\cI}^{ m-2}\norm{\varphi}_{\ell^{\infty}([-r,r]^m)} = C_{\cI}^{m-2} u_{r,m}(\varphi).
\end{align*}
In addition, the proof of Proposition \ref{prop.ellinflbd} shows that $\Phi(P) = \ev_{\scI}^m(P)$ for all $P \in \cP_m$.
Thus, $\Phi = \Phi_{\scI,u}^m$, i.e., \eqref{eq.commLFC} holds.
\end{proof}

\subsection{(\texorpdfstring{$\cA$}{}-)Varopoulos LFC}\label{subsec.VaropLFC}

Next, we consider the Varopoulos family $\beta$ and its $\cA$-specific counterpart $\beta^{\scA}$ from Example \ref{ex.betaA}, which give rise to the Varopoulos LFC and the \textbf{$\boldsymbol{\cA}$-Varopoulos LFC}, respectively.
The Varopoulos family gives rise to the smallest well-behaved space of functions---in the sense of Proposition \ref{prop.Calpha}\ref{item.compcomp}---that has a rich, useful LFC in \emph{every} unital $C^*$-algebra.
Considering the $\cA$-specific family $\beta^{\scA}$ enables the proper treatment of the finite-dimensional case, in which LFC is defined for \emph{all} multivariate continuous functions.

Here and throughout, I shall freely identify the Varopoulos algebra $V(\Om_1,\ldots,\Om_m)$ with the projective tensor product $C(\Om_1) \potimes \cdots \potimes C(\Om_m)$ \`a la Corollary \ref{cor.seriesVarop}.

\begin{notation}\label{nota.phitensor}
Let $\cA$ be a unital $C^*$-algebra.
\begin{enumerate}[font=\normalfont,label=(\roman*)]
    \item $\cA_{\nu} \coloneqq \{a \in \cA : a^*a=aa^*\}$ is the set of its normal elements.\label{item.normal}
    \item If $a \in \cA_{\nu}$, then $\Phi_a \colon C(\sigma(a)) \to \cA$ is the continuous functional calculus homomorphism, i.e., $\Phi_a(f) \coloneqq f(a) \in \cA$ for all $f \in C(\sigma(a))$.\label{item.FC}
    \item If $\mathbf{a} = (a_1,\ldots,a_m) \in \cA_{\nu}^m$ and
    \[
    \psi \in C(\sigma(a_1)) \potimes \cdots \potimes C(\sigma(a_m)) = V(\sigma(a_1),\ldots,\sigma(a_m)),
    \]
    define
    \[
    \psi_{\sotimes}(\mathbf{a}) \coloneqq \big(\Phi_{a_1} \potimes \cdots \potimes \Phi_{a_m}\big)(\psi) \in \cA^{\potimes m},
    \]
    where $\Phi_{a_1} \potimes \cdots \potimes \Phi_{a_m} \colon C(\sigma(a_1)) \potimes \cdots \potimes C(\sigma(a_m)) \to \cA^{\potimes m}$ is the projective tensor product of the maps $\Phi_{a_1},\ldots,\Phi_{a_m}$.
    In particular, if $\varphi \colon \R^m \to \C$ satisfies the condition that $\varphi|_{\sigma(a_1) \times \cdots \times \sigma(a_m)} \in V(\sigma(a_1),\ldots,\sigma(a_m)) = C(\sigma(a_1)) \potimes \cdots \potimes C(\sigma(a_m))$ for all $\mathbf{a} = (a_1,\ldots,a_m) \in \cA_{\sa}^m$, then $\varphi$ induces the function
    \[
    \cA_{\sa}^m \ni \mathbf{a} \mapsto \varphi_{\sotimes}(\mathbf{a}) \coloneqq \big(\varphi|_{\sigma(a_1) \times \cdots \times \sigma(a_m)}\big)_{\sotimes}(\mathbf{a}) \in \cA^{\potimes m}
    \]
    from $\cA_{\sa}^m$ to $\cA^{\potimes m}$.\label{item.phitensor}
\end{enumerate}
\end{notation}

I leave it to the reader to check that Notation \ref{nota.phitensor}\ref{item.phitensor} agrees with Notation \ref{nota.algstuff}\ref{item.Ptensor} (restricted to self-adjoint inputs) when $\varphi = P \in \cP_m$.
Here is another important example.
\pagebreak

\begin{example}[$\psi_{\sotimes}(\mathbf{a})$ in finite dimensions]\label{ex.findimLFC}
Let $\cA$ be a finite-dimensional unital $C^*$-algebra.
Note that if $a \in \cA_{\nu}$, then $\#\sigma(a) \leq \dim \cA$, i.e., $\sigma(a)$ has at most $\dim \cA$ elements.
Indeed, $C(\sigma(a)) \cong C^*(a,1)$ via $\Phi_a \colon C(\sigma(a)) \to C^*(a,1) \subseteq \cA$.
Thus,
\[
\dim C(\sigma(a)) = \dim C^*(a,1) \leq \dim \cA < \infty,
\]
from which it follows that $\#\sigma(a) \leq \dim \cA$.
Now, write
\[
P_{\lambda}^a \coloneqq 1_{\{\lambda\}}(a) \in \Proj\left(\cA\right) \coloneqq \{p \in \cA_{\sa} : p^2=p\} \qquad (\lambda \in \C).
\]
If $\mathbf{a} \in \cA_{\sa}^m$ and $\psi \in C(\sigma(a_1) \times \cdots \times \sigma(a_m)) = \C^{\sigma(a_1) \times \cdots \times \sigma(a_m)}$, then
\begin{align*}
    \psi & = \sum_{\blambda \in \sigma(a_1) \times \cdots \times \sigma(a_m)} \psi(\blambda)\,1_{\{\blambda\}} \\
    & = \sum_{\blambda \in \sigma(a_1) \times \cdots \times \sigma(a_m)} \psi(\blambda)\,1_{\{\lambda_1\}} \otimes \cdots \otimes 1_{\{\lambda_m\}}.
\end{align*}
Consequently,
\begin{align*}
    \psi_{\sotimes}(\mathbf{a}) & = \sum_{\blambda \in \sigma(a_1) \times \cdots \times \sigma(a_m)} \psi(\blambda)\,\Phi_{a_1}(1_{\{\lambda_1\}}) \otimes \cdots \otimes \Phi_{a_m}(1_{\{\lambda_m\}}) \\
    & = \sum_{\blambda \in \sigma(a_1) \times \cdots \times \sigma(a_m)} \psi(\blambda)\,P_{\lambda_1}^{a_1} \otimes \cdots \otimes P_{\lambda_m}^{a_m}. 
\end{align*}
In particular, if $m \geq 2$, then
\begin{equation}
    \psi_{\sotimes}(\mathbf{a})\sh b = \sum_{\blambda \in \sigma(a_1) \times \cdots \times \sigma(a_m)} \psi(\blambda)\,P_{\lambda_1}^{a_1} b_1 \cdots P_{\lambda_{m-1}}^{a_{m-1}}b_{m-1}P_{\lambda_m}^{a_m}\label{eq.findimLFC}
\end{equation}
for all $b \in \cA^{m-1}$.
\end{example}

Let us now begin the development of $\cA$-Varopoulos LFC in general.

\begin{lemma}\label{lem.bdphitensor}
Let $\cA$ be a unital $C^*$-algebra, and suppose $\varphi \colon \R^m \to \C$ satisfies the condition
\[
\varphi|_{\sigma(a_1) \times \cdots \times \sigma(a_m)} \in V(\sigma(a_1),\ldots,\sigma(a_m)) \; \text{ for all } \; (a_1,\ldots,a_m) \in \cA_{\sa}^m,
\]
e.g., $\varphi \in \cC_{\beta^{\scA}}(\R^m)$.
If $\mathbf{a} = (a_1,\ldots,a_m) \in \cA_{\sa}^m$, then
\[
\norm{\varphi_{\sotimes}(\mathbf{a})}_{\cA^{\potimes m}} \leq \norm{\varphi|_{\sigma(a_1) \times \cdots \times \sigma(a_m)}}_{V(\sigma(a_1),\ldots,\sigma(a_m))}.
\]
In particular,
\[
\sup_{\mathbf{a} \in \cA_{\sa,r}^m} \norm{\varphi_{\sotimes}(\mathbf{a})}_{\cA^{\potimes m}} \leq \beta_{r,m}^{\scA}(\varphi) \leq \beta_{r,m}(\varphi) \qquad (r > 0).
\]
\end{lemma}

\begin{proof}
Recall that $\Phi_{a_i} \colon C(\sigma(a_i)) \to \cA$ has operator norm one for all $i=1,\ldots,m$.
Thus, the tensor product $\Phi_{a_1} \potimes \cdots \potimes \Phi_{a_m}$ has operator norm one.
Consequently,
\begin{align*}
    \norm{\varphi_{\sotimes}(\mathbf{a})}_{\cA^{\potimes m}} & = \norm{\big(\Phi_{a_1} \potimes \cdots \potimes \Phi_{a_m}\big)\big(\varphi|_{\sigma(a_1) \times \cdots \times \sigma(a_m)}\big)}_{\cA^{\potimes m}} \\
    & \leq \norm{\varphi|_{\sigma(a_1) \times \cdots \times \sigma(a_m)}}_{C(\sigma(a_1)) \potimes \cdots \potimes C(\sigma(a_m))} \\
    & = \norm{\varphi|_{\sigma(a_1) \times \cdots \times \sigma(a_m)}}_{V(\sigma(a_1),\ldots,\sigma(a_m))},
\end{align*}
as desired.
\end{proof}

\begin{lemma}\label{lem.phitensorcont}
If $\cA$ is a unital $C^*$-algebra and $\varphi \in C_{\beta^{\scA}}(\R^m)$, then $\varphi_{\sotimes} \in \Cbb\big(\cA_{\sa}^m;\cA^{\potimes m}\big)$.
\end{lemma}

\begin{proof}
If $P(\blambda) = \sum_{|\delta| \leq d} c_{\delta}\blambda^{\delta} \in \cP_m$, then
\[
P_{\sotimes}(\mathbf{a}) = \sum_{|\delta| \leq d} c_{\delta}\,a_1^{\delta_1} \otimes \cdots \otimes a_m^{\delta_m} \qquad \big(\mathbf{a} \in \cA_{\sa}^m\big).
\]
Consequently, it is clear that $P_{\sotimes} \in \Cbb\big(\cA_{\sa}^m;\cA^{\potimes m}\big)$.
Now, if $\varphi \in C_{\beta^{\scA}}(\R^m)$, let $(P_n)_{n \in \N}$ be a sequence in $\cP_m$ converging in $C_{\beta^{\scA}}(\R^m)$ to $\varphi$.
By Lemma \ref{lem.bdphitensor}, if $r > 0$, then
\[
\sup_{\mathbf{a} \in \cA_{\sa,r}^m} \norm{(P_n)_{\sotimes}(\mathbf{a}) - \varphi_{\sotimes}(\mathbf{a})}_{\cA^{\potimes m}} \leq \beta_{r,m}^{\scA}(P_n - \varphi) \xrightarrow{n \to \infty} 0.
\]
Thus, $((P_n)_{\sotimes})_{n \in \N}$ converges to $\varphi_{\sotimes}$ uniformly on bounded sets.
The result follows.
\end{proof}

\begin{theorem}[$\cA$-Varopoulos LFC]\label{thm.AVaropLFC}
Let $\cA$ be a unital $C^*$-algebra.
The $\cA$-Varopoulos family $\beta^{\scA}$ is ideal for LFC in $\cA$.
Furthermore, if $\cI \sni \cA$, $m \geq 2$, and $\varphi \in C_{\beta^{\scA}}(\R^m)$, then
\begin{equation}
    \varphi_{\scI,\beta^{\scA}}(\mathbf{a}) \sh b = \varphi_{\sotimes}(\mathbf{a})\sh b \qquad \big(\mathbf{a} \in \cA_{\sa}^m, \; b \in \cI^{m-1}\big),\label{eq.VaropLFC}
\end{equation}
where the $\sh$ on the left-hand side is formal and the $\sh$ on the right-hand side is the actual $\sh = \sh_k$ operation introduced in subsection \ref{subsec.HFC}.
\end{theorem}

\begin{proof}
As noted in Example \ref{ex.Varopstrongerthanuniform}, $u \leq \beta^{\scA}$.
Also, as noted in Example \ref{ex.betaA}, $\beta^{\scA}$ is complete, compatible, increasing, and translation invariant.
Since it is clear that $\beta_{\cdot,1}^{\scA} = u_{\cdot,1}$, in order to prove that $\beta^{\scA}$ is ideal for LFC in $\cA$, it suffices to show that $\beta^{\scA}$ controls polynomial LFC in $\cA$ with an insertion from $\cI$ whenever $\cI \sni \cA$.
We shall do so and establish \eqref{eq.VaropLFC}~simultaneously.

To this end, let $i=1,\ldots,m-1$, and write $\sh^i \colon \cA^{\potimes m} \to B_{m-1}(\cA^{i-1} \times \cI \times \cA^{m-1-i};\cI)$ for the map $u \mapsto (b \mapsto u\sh b)$.
By Proposition \ref{prop.hashonI}, $\sh^i$ is a well-defined bounded linear map with operator norm at most one.
Now, Lemma \ref{lem.phitensorcont} implies that
\[
\sh^i\varphi_{\sotimes} = \sh^i\circ \varphi_{\sotimes} \in \Cbb(\cA_{\sa}^m; B_{m-1}(\cA^{i-1} \times \cI \times \cA^{m-1-i};\cI)) \qquad \big(\varphi \in C_{\beta^{\scA}}(\R^m)\big).
\]
Define $\Phi \colon C_{\beta^{\scA}}(\R^m) \to \Cbb(\cA_{\sa}^m;B_{m-1}(\cA^{i-1} \times \cI \times \cA^{m-1-i};\cI))$ by $\varphi \mapsto \sh^i\varphi$.
If $r > 0$, then
\begin{align*}
    \sup_{\mathbf{a} \in \cA_{\sa,r}^m} \norm{\Phi(\varphi)(\mathbf{a})} & = \sup_{\mathbf{a} \in \cA_{\sa,r}^m} \norm{\sh^i\varphi_{\sotimes}(\mathbf{a})}_{B_{m-1}(\cA^{i-1} \times \cI \times \cA^{m-1-i};\cI)} \\
    & \leq \sup_{\mathbf{a} \in \cA_{\sa,r}^m} \norm{\varphi_{\sotimes}(\mathbf{a})}_{\cA^{\potimes m}} \leq \beta_{r,m}^{\scA}(\varphi)
\end{align*}
by Lemma \ref{lem.bdphitensor}.
Since $\Phi(P) = \ev_{\scI}^{m,i}(P)$ for all $P \in \cP_m$, it follows---after applying the above bound with $\cI = \cA$ as well---that $\beta^{\scA}$ controls polynomial LFC in $\cA$ with an insertion from $\cI$.
Furthermore, since $\iota_{\scI}$ is continuous, the map $\Psi \colon C_{\beta^{\scA}}(\R^m) \to \Cbb(\cA_{\sa}^m ;B_{m-1}(\cI^{m-1};\cI))$ defined by
\[
\Psi(\varphi)(\mathbf{a}) \coloneqq \Phi(\varphi)(\mathbf{a})|_{\cI^{m-1}} \qquad \big(\varphi \in C_{\beta^{\scA}}(\R^m), \; \mathbf{a} \in \cA_{\sa}^m\big)
\]
is well defined, linear, and continuous.
Since $\Psi(P) = \ev_{\scI}^m(P)$ for all $P \in \cP_m$, it follows that $\Psi = \Phi_{\scI,\beta^{\scA}}^m$, i.e., \eqref{eq.VaropLFC} holds.
\end{proof}

\begin{corollary}\label{cor.VaropLFC}
Let $\cA$ be a unital $C^*$-algebra.
\begin{enumerate}[font=\normalfont,label=(\roman*)]
    \item The Varopoulos family $\beta$ is ideal for LFC, and \eqref{eq.VaropLFC} holds for all $\varphi \in VC(\R^m)$ with $\varphi_{\scI,\beta}(\mathbf{a}) \sh b $ in place of $\varphi_{\scI,\beta^{\scA}}(\mathbf{a}) \sh b$.\label{item.VaropLFC}
    \item If $\alpha$ is ideal for LFC in $\cA$, then $\alpha \leq \beta$, and \eqref{eq.VaropLFC} holds for all $\varphi \in VC(\R^m) \subseteq C_{\alpha}(\R^m)$ with $\varphi_{\scI,\alpha}(\mathbf{a}) \sh b $ in place of $\varphi_{\scI,\beta^{\scA}}(\mathbf{a}) \sh b$.\label{item.VaropLFCconverse}
\end{enumerate}
\end{corollary}

\begin{proof}
For the first item, recall from Example \ref{ex.Varopstrongerthanuniform} that $u \leq \beta$ and from Example \ref{ex.uniformandVaropoulos} that $\beta$ is complete, compatible, increasing, and translation invariant.
With these in mind, the rest of the first item follows at once from Theorem \ref{thm.AVaropLFC} and the comments following Notation~\ref{nota.LFC}.

For the second item, combine the definition of ideal for LFC in $\cA$, Proposition \ref{prop.Calpha}\ref{item.compcomp}, the first item, and the comments following Notation \ref{nota.LFC}.
\end{proof}

Now, I use Theorem \ref{thm.IPTPV} to demonstrate that there are many ways to compute $\varphi_{\sotimes}(\mathbf{a})$ and thus the ($\cA$-)Varopoulos LFC.

\begin{proposition}\label{prop.tens}
Retain the setup of Theorem \ref{thm.IPTPV}.
For each $i = 1,\ldots,m$, let $V_i$ be a complex Banach space and $T_i \colon C(\Om_i) \to V_i$ be a bounded linear map.
Also, write
\[
T \coloneqq T_1 \potimes \cdots \potimes T_m \colon V(\Om_1,\ldots,\Om_m) = C(\Om_1) \potimes \cdots \potimes C(\Om_m) \to V_1 \potimes \cdots \potimes V_m
\]
and $\Om \coloneqq \Om_1 \times \cdots \times \Om_m$.
\begin{enumerate}[label=(\roman*),font=\normalfont]
    \item If $\varphi$ is as in \eqref{eq.integphi}, then
    \[
    T(\varphi) = \int_{\Sigma} T_1(\varphi_1(\cdot,\sigma))\otimes \cdots \otimes T_m(\varphi_m(\cdot,\sigma))\,\rho(\d\sigma),
    \]
    where the right-hand side is a Bochner integral in $V_1 \potimes \cdots \potimes V_m$.\label{item.tensint}
    \item Suppose $\cA \coloneqq V_1 = \cdots = V_m$ is a unital $C^*$-algebra, $\cI \sni \cA$, $i=1,\ldots,m-1$, and $b \in \cA^{i-1}\times \cI \times \cA^{m-1-i}$.
    If $\varphi$ is as in \eqref{eq.integphi}, then
    \[
    T(\varphi) \sh b = \int_{\Sigma} T_1(\varphi_1(\cdot,\sigma))\,b_1\cdots T_{m-1}(\varphi_{m-1}(\cdot,\sigma))\,b_{m-1}\,T_m(\varphi_m(\cdot,\sigma))\,\rho(\d\sigma) \in \cI,
    \]
    where the right-hand side is a Bochner integral in $(\cI,\norm{\cdot}_{\cI})$.\label{item.hashtensint}
\end{enumerate}
\end{proposition}

\begin{proof}
By Theorem \ref{thm.IPTPV},
\[
\varphi = \int_{\Sigma} \varphi_1(\cdot,\sigma) \otimes \cdots \otimes \varphi_m(\cdot,\sigma)\,\rho(\d\sigma)
\]
is a Bochner integral in $C(\Om_1) \potimes \cdots \potimes C(\Om_m)$.
Since $T$ is a bounded linear map from $C(\Om_1) \potimes \cdots \potimes C(\Om_m)$ to $V_1 \potimes \cdots \potimes V_m$, it follows from basic properties of the Bochner integral that
\begin{align*}
    T(\varphi) & = \int_{\Sigma} T(\varphi_1(\cdot,\sigma) \otimes \cdots \otimes \varphi_m(\cdot,\sigma))\,\rho(\d\sigma) \\
    & = \int_{\Sigma} T_1(\varphi_1(\cdot,\sigma))\otimes \cdots \otimes T_m(\varphi_m(\cdot,\sigma))\,\rho(\d\sigma),
\end{align*}
as claimed in the first item.
\pagebreak

For the second item, let $\sh^i$ be as in the proof of Theorem \ref{thm.AVaropLFC}, and write $S \colon \cA^{\potimes m} \to \cI$ for the bounded linear map $u \mapsto (\sh^iu)[b] = u \sh b$.
Then the first item and, again, basic properties of the Bochner integral yield that
\begin{align*}
    T(\varphi)\sh b & = S(T(\varphi)) = S\left(\int_{\Sigma} T_1(\varphi_1(\cdot,\sigma))\otimes \cdots \otimes T_m(\varphi_m(\cdot,\sigma))\,\rho(\d\sigma)\right)\\
    & = \int_{\Sigma} S\left(T_1(\varphi_1(\cdot,\sigma))\otimes \cdots \otimes T_m(\varphi_m(\cdot,\sigma))\right)\rho(\d\sigma) \\
    & = \int_{\Sigma} T_1(\varphi_1(\cdot,\sigma))\,b_1\cdots T_{m-1}(\varphi_{m-1}(\cdot,\sigma))\,b_{m-1}\,T_m(\varphi_m(\cdot,\sigma))\,\rho(\d\sigma) \in \cI,
\end{align*}
as claimed in the second item.
\end{proof}

Of course, the situation of present interest is when $\Om_i = \sigma(a_i)$ with $a_i \in \cA_{\sa}$ and $T_i = \Phi_{a_i}$ for all $i=1,\ldots,m$.
Here is an example in this situation that will show up later.

\begin{example}[A calculation of $\varphi_{\sotimes}(\mathbf{a})$]\label{ex.genWiener}
Return to the setup of Example \ref{ex.multivarWiener}.
In particular,
\[
\varphi(\blambda) = \int_{\Sigma}e^{i\,\boldsymbol{\xi}(\sigma)\cdot \blambda} \, \mu(\d\sigma) \qquad (\blambda \in \R^m).
\]
Next, let $\cA$ be a unital $C^*$-algebra.
If $\mathbf{a} \in \cA_{\sa}^m$, then, using the decomposition of $\varphi$ from Example \ref{ex.multivarWiener}, Proposition \ref{prop.tens}\ref{item.tensint} yields
\[
\varphi_{\sotimes}(\mathbf{a}) = \int_{\Sigma} e^{i\xi_1(\sigma)\,a_1} \otimes \cdots \otimes e^{i\xi_m(\sigma)\,a_m}\,\mu(\d\sigma) \in \cA^{\potimes m}.
\]
If, in addition, $\cI \sni \cA$, $i=1,\ldots,m-1$, and $b \in \cA^{i-1} \times \cI \times \cA^{m-1-i}$, then
\[
\varphi_{\scI,\beta}(\mathbf{a})\sh b = \varphi_{\sotimes}(\mathbf{a})\sh b = \int_{\Sigma} e^{i\xi_1(\sigma)\,a_1} b_1 \cdots e^{i\xi_{m-1}(\sigma)\,a_{m-1}} b_{m-1}e^{i\xi_m(\sigma)\,a_m}\,\mu(\d\sigma) \in \cI
\]
by Theorem \ref{thm.AVaropLFC} and Proposition \ref{prop.tens}\ref{item.hashtensint}.
\end{example}

To end this subsection, I explore the finite-dimensional case.
Aside from Theorem \ref{thm.AVaropLFC}, the key result to its analysis is that if $\cA$ is a finite-dimensional unital $C^*$-algebra, then the $\cA$-Varopoulos family $\beta^{\scA}$ is equivalent to the uniform family $u$.

\begin{lemma}\label{lem.betaAuA}
Let $\cA$ be a finite-dimensional unital $C^*$-algebra.
The family
\[
u^{\scA} = \big(u_{r,m}^{\scA} \coloneqq (\dim \cA)^{m-1}u_{r,m}\colon C([-r,r]^m) \to [0,\infty]\big)_{r > 0, \, m \in \N}
\]
satisfies $\beta^{\scA} \leq u^{\scA}$.
In particular, $C_{\beta^{\scA}}(\R^m) = C(\R^m)$ for all $m \in \N$.
\end{lemma}

\begin{proof}
Note that $\beta_{\cdot,1}^{\scA} = u_{\cdot,1}^{\scA} = u_{\cdot,1}$, so it suffices to take $m \geq 2$.
In this case, let $r > 0$ and $\mathbf{a} = (a_1,\ldots,a_m) \in \cA_{\sa,r}^m$.
If $\varphi \in C([-r,r]^m)$, then
\begin{align*}
    \varphi|_{\sigma(a_1) \times \cdots \times \sigma(a_m)} & = \sum_{\blambda \in \sigma(a_1) \times \cdots \times \sigma(a_m)} \varphi(\blambda)\, 1_{\{\blambda\}} \\
    & = \sum_{\bmu \in \sigma(a_2) \times \cdots \times \sigma(a_m)} \left(\sum_{\lambda \in \sigma(a_1)} \varphi(\lambda,\bmu)\,1_{\{\lambda\}}\right) \otimes 1_{\{\mu_1\}} \otimes \cdots \otimes 1_{\{\mu_{m-1}\}} \\
    & = \sum_{\bmu \in \sigma(a_2) \times \cdots \times \sigma(a_m)} \varphi(\cdot,\bmu)|_{\sigma(a_1)} \otimes 1_{\{\mu_1\}} \otimes \cdots \otimes 1_{\{\mu_{m-1}\}}.
\end{align*}
Consequently,
\begin{align*}
    \norm{\varphi|_{\sigma(a_1) \times \cdots \times \sigma(a_m)}}_{V(\sigma(a_1),\ldots,\sigma(a_m))} & \leq \sum_{\bmu \in \sigma(a_2) \times \cdots \times \sigma(a_m)} \norm{\varphi(\cdot,\bmu)}_{\ell^{\infty}(\sigma(a_1))} \prod_{i=1}^{m-1}\norm{1_{\{\mu_i\}}}_{\ell^{\infty}(\sigma(a_{i+1}))} \\
    & = \sum_{\bmu \in \sigma(a_2) \times \cdots \times \sigma(a_m)} \norm{\varphi(\cdot,\bmu)}_{\ell^{\infty}(\sigma(a_1))} \\
    & \leq \left(\max_{i=2,\ldots,m} \#\sigma(a_i)\right)^{m-1}\max_{\blambda \in \sigma(a_1)\times \cdots \times \sigma(a_m)}|\varphi(\blambda)| \\
    & \leq (\dim \cA)^{m-1} \max_{\blambda \in \sigma(a_1)\times \cdots \times \sigma(a_m)}|\varphi(\blambda)| \\
    & \leq (\dim \cA)^{m-1}u_{r,m}(\varphi) = u_{r,m}^{\scA}(\varphi).
\end{align*}
Taking the supremum over all $\mathbf{a} \in \cA_{\sa,r}^m$ yields the result.
\end{proof}

\begin{corollary}[LFC in finite dimensions]\label{cor.findimLFC}
If $\cA$ is a finite-dimensional unital $C^*$-algebra and $\cI \sni \cA$, then $u^{\scA}$ controls polynomial LFC in $\cA$ with an insertion from $\cI$.
Furthermore, if $m \geq 2$ and $\varphi \in C(\R^m) = C_{u^{\scA}}(\R^m)$, then
\[
\varphi_{\cI,u^{\scA}}(\mathbf{a})\sh b = \sum_{\blambda \in \sigma(a_1) \times \cdots \times \sigma(a_m)} \varphi(\blambda)\,P_{\lambda_1}^{a_1}b_1\cdots P_{\lambda_{m-1}}^{a_{m-1}}b_{m-1} P_{\lambda_m}^{a_m}
\]
for all $\mathbf{a} \in \cA_{\sa}^m$ and $b \in \cI^{m-1}$.
\end{corollary}

\begin{proof}
Combine Lemma \ref{lem.betaAuA}, Theorem \ref{thm.AVaropLFC}, the comments following Notation \ref{nota.LFC}, and identity \eqref{eq.findimLFC} from Example \ref{ex.findimLFC}.
\end{proof}

It is worth observing that the constant $(\dim \cA)^{m-1}$ in $u^{\scA}$ can be improved when $\cI = \cA$.
To be more precise, Corollary \ref{cor.findimLFC} says that if $\cA$ is a finite-dimensional unital $C^*$-algebra and $\varphi \in C(\R^m)$, then
\begin{equation}
    \sup_{\mathbf{a} \in \cA_{\sa,r}^m}\norm{\varphi_{\scA,u^{\scA}}(\mathbf{a})}_{B_{m-1}(\cA^{m-1};\cA)} \leq (\dim \cA)^{m-1}\norm{\varphi}_{\ell^{\infty}([-r,r]^m)} \qquad (r > 0). \label{eq.subop}
\end{equation}
Using the full classification of finite-dimensional $C^*$-algebras and sharp estimates on the operator norms of multilinear Schur multipliers, specifically, \cite[eq.\ (4.1.6)]{ST2019}, inequality \eqref{eq.subop} can be improved substantially.
Before stating the result, I recall some information about the classification of finite-dimensional $C^*$-algebras.

\begin{fact}\label{fact.findimclass}
If $\cA$ is a finite-dimensional $C^*$-algebra, then there exist unique $k \in \N_0$ and unique-up-to-permutations $(n_1,\ldots,n_k) \in \N^k$ such that
\[
\cA \cong \mathrm{M}_{n_1}(\C) \oplus \cdots \oplus \mathrm{M}_{n_k}(\C).
\]
In fact, if $Z(\cA) = \{a \in \cA : ab=ba$ for all $b \in \cA\}$ is the center of $\cA$, then $k = \dim Z(\cA)$, and there is a unique collection $\{p_1,\ldots,p_k\} \subseteq \Proj\left(Z(\cA)\right) \setminus \{0\}$ of $k$ minimal central projections such that
\[
n_i = \sqrt{\dim \left(\cA p_i\right)} \qquad (i=1,\ldots,k).
\]
\end{fact}

\begin{proof}[Sketch of proof]
I merely state the highlights of the proof.
Please see \cite[\S{I}.11]{Takesaki1979} for the details of the argument.
If $\cA$ is finite dimensional, then $Z(\cA)$ is a finite-dimensional commutative $C^*$-algebra. Thus, $Z(\cA) \cong C(\Om)$ for some finite discrete space $\Om = \{\om_1,\ldots,\om_k\}$.
If $p_1,\ldots,p_k \in \Proj\left(Z(\cA)\right)$ are the projections corresponding, respectively, to the functions $1_{\{\om_1\}},\ldots,1_{\{\om_k\}} \in C(\Om)$, then $\cA \cong \bigoplus_{i=1}^k\cA p_i$, and $\cA p_i$ is a simple $C^*$-algebra for all $i=1,\ldots,k$.
Also, since $\Proj\left(Z(\cA)\right) = \{$element corresponding to $1_F : F \subseteq \Om\}$, $\{p_1,\ldots,p_k\}$ is clearly the unique collection of size $k$ of non-zero minimal central projections in $\cA$. Now, if $\cB$ is a finite-dimensional simple $C^*$-algebra, then $\cB \cong \MnC$ for some $n \in \N$. Of course, $n$ is uniquely determined by $\cB$ because $\dim \cB = n^2$.
It follows that $\cA \cong \bigoplus_{i=1}^k\mathrm{M}_{n_i}(\C)$, where $n_i = \sqrt{\dim \left(\cA p_i\right)}$ for all $i=1,\ldots,k$.
\end{proof}

\begin{proposition}\label{prop.optimestim}
Let $\cA$ be a non-zero (and therefore unital) finite-dimensional $C^*$-algebra, and define
\[
n(\cA) \coloneqq \max\left\{\dim\left(\cA p\right) : p \in \Proj\left(Z(\cA)\right) \text{ is minimal}\right\} \in \{1,\ldots,\dim\cA\}.
\]
If $m \geq 2$ and $\varphi \in C(\R^m)$, then
\[
\|\varphi_{\scA,u^{\scA}}(\mathbf{a})\|_{B_{m-1}(\cA^{m-1};\cA)} \leq n(\cA)^{\frac{m-1}{4}}\norm{\varphi}_{\ell^{\infty}(\sigma(a_1) \times \cdots \times \sigma(a_m))} \qquad (\mathbf{a} \in \cA_{\sa}^m).
\]
In particular, the family $\big(n(\cA)^{(m-1)/4}u_{r,m})_{r > 0, \, m \in \N}$ controls polynomial LFC in $\cA$.
\end{proposition}

\begin{proof}
To begin, let us prove the claimed estimate when $\cA = \MnC$ for $n \in \N$.
Let $m \geq 2$ and $\varphi \in C(\R^m)$.
By Corollary \ref{cor.findimLFC},
\[
\varphi_{\scA,u^{\scA}}(\mathbf{a}) \sh b = \sum_{\blambda \in \sigma(a_1) \times \cdots \times \sigma(a_m)} \varphi(\blambda) \, P_{\lambda_1}^{a_1} b_1\cdots P_{\lambda_{m-1}}^{a_{m-1}} b_{m-1} P_{\lambda_m}^{a_m}
\]
for all $\mathbf{a} \in \cA_{\sa}^m = \MnC_{\sa}^m$ and $b \in \cA^{m-1} = \MnC^{m-1}$.
Consequently, by \cite[Cor.\ 4.1.4]{ST2019}, which relies on the aforementioned sharp operator-norm estimates on multilinear Schur multipliers (vid.\ \cite[eq.\ (4.1.6)]{ST2019}), 
\[
\norm{\varphi_{\scA,u^{\scA}}(\mathbf{a})}_{B_{m-1}(\MnC^{m-1};\MnC)} \leq n^{\frac{m-1}{2}}\norm{\varphi}_{\ell^{\infty}(\sigma(a_1)\times \cdots \times \sigma(a_m))}
\]
for all $\mathbf{a} \in \cA_{\sa}^m = \MnC_{\sa}^m$.
Since $n(\cA) = n(\MnC) = n^2$, this is the desired estimate.

In general, Fact \ref{fact.findimclass} implies that it suffices to assume $\cA = \mathrm{M}_{n_1}(\C) \oplus \cdots \oplus \mathrm{M}_{n_k}(\C)$ for some $k \in \N$ and $n_1,\ldots,n_k \in \N$, in which case $n(\cA) = \max\left\{n_i^2 : i=1,\ldots,k\right\}$.
Recall that the $C^*$-algebraic direct sum $\cA = \mathrm{M}_{n_1}(\C) \oplus \cdots \oplus \mathrm{M}_{n_k}(\C)$ is given componentwise operations and the norm
\[
\norm{(a_1,\ldots,a_k)} = \max_{i=1,\ldots,k}\norm{a_i}_{\mathrm{M}_{n_i}(\C)}.
\]
Consequently, if $A = (a_1,\ldots,a_k) \in \cA$ and $f \colon \sigma(A) \to \C$ is a function, then
\[
\sigma(A) = \bigcup_{i = 1}^k \sigma(a_i) \; \text{ and } \; f(A) = (f(a_1),\ldots,f(a_k)).
\]
In particular,
\[
P_{\lambda}^A = (P_{\lambda}^{a_1},\ldots,P_{\lambda}^{a_k}) \qquad (\lambda \in \C).
\]
Now, fix $\boldsymbol{A} = (A_1,\ldots,A_m) \in \cA_{\sa}^m$ and $B = (B_1,\ldots,B_{m-1}) \in \cA^{m-1}$.
For $j=1,\ldots,m$ and $i = 1,\ldots, m-1$, write $A_j = (a_{j,1},\ldots,a_{j,k})$ and $B_i = (b_{i,1},\ldots,b_{i,k})$.
Next, if $\ell=1,\ldots,k$, then by Corollary \ref{cor.findimLFC} (twice) and the above,
\begin{align*}
    (\varphi_{\scA,u^{\scA}}(\boldsymbol{A}) \sh B)_{\ell} & = \left(\sum_{\blambda \in \sigma(A_1) \times \cdots \times \sigma(A_m)} \varphi(\blambda) \, P_{\lambda_1}^{A_1}B_1\cdots P_{\lambda_{m-1}}^{A_{m-1}}B_{m-1} P_{\lambda_m}^{A_m}\right)_{\ell} \\
    & = \sum_{\blambda \in \sigma(A_1) \times \cdots \times \sigma(A_m)} \varphi(\blambda) \, P_{\lambda_1}^{a_{1,\ell}}b_{1,\ell}\cdots P_{\lambda_{m-1}}^{a_{m-1,\ell}}b_{m-1,\ell}P_{\lambda_m}^{a_{m,\ell}} \\
    & = \sum_{\blambda \in \sigma(a_{1,\ell}) \times \cdots \times \sigma(a_{m,\ell})} \varphi(\blambda) \, P_{\lambda_1}^{a_{1,\ell}}\,b_{1,\ell}\cdots P_{\lambda_{m-1}}^{a_{m-1,\ell}}\,b_{m-1,\ell}\,P_{\lambda_m}^{a_{m,\ell}} \\
    & = \varphi_{\mathsmaller{\mathrm{M}_{n_{\ell}}(\C)},u^{\mathsmaller{\mathrm{M}_{n_{\ell}}(\C)}}}(a_{1,\ell},\ldots,a_{m,\ell})\sh[b_{1,\ell},\ldots,b_{m-1,\ell}].
\end{align*}
Therefore, by the previous paragraph,
\begin{align*}
    \norm{\varphi_{\scA,u^{\scA}}(\boldsymbol{A}) \sh B} & = \max_{\ell=1,\ldots,k}\norm{\varphi_{\mathsmaller{\mathrm{M}_{n_{\ell}}(\C)},u^{\mathsmaller{\mathrm{M}_{n_{\ell}}(\C)}}}(a_{1,\ell},\ldots,a_{m,\ell})\sh[b_{1,\ell},\ldots,b_{m-1,\ell}]}_{\mathrm{M}_{n_{\ell}}(\C)} \\
    & \leq \norm{B_1}\cdots\norm{B_{m-1}} \max_{\ell=1,\ldots,k}\left(n_{\ell}^{\frac{m-1}{2}}\norm{\varphi}_{\ell^{\infty}(\sigma(a_{1,\ell}) \times \cdots \times \sigma(a_{m,\ell}))}\right) \\
    & \leq n(\cA)^{\frac{m-1}{4}}\norm{\varphi}_{\ell^{\infty}(\sigma(A_1) \times \cdots \times \sigma(A_m))}\norm{B_1}\cdots\norm{B_{m-1}}.
\end{align*}
This completes the proof.
\end{proof}

\subsection{An example using multiple operator integrals}\label{subsec.MOIfam}

Finally, for the reader sufficiently familiar with multiple operator integrals (MOIs), I use some MOI theory to introduce and study one more interesting family of possibly infinite norms and its associated LFC.
I shall freely use the notation, terminology, and results from \cite[\S2.2 \& \S4.2]{Nikitopoulos2023n}.
The reader who is uninterested in or insufficiently familiar with MOIs may safely skip to the next section. 

Below, if $\Xi$ is a Polish space, then $\ell^{\infty}(\Xi,\cB_{\Xi})^{\iotimes m}$ denotes the $m$-fold integral projective tensor product $\ell^{\infty}(\Xi,\cB_{\Xi}) \iotimes \cdots \iotimes \ell^{\infty}(\Xi,\cB_{\Xi})$.

\begin{notation}\label{nota.nufam}
If $r > 0$, $m \in \N$, and $\varphi \in C([-r,r]^m)$, then
\[
\nu_{r,m}(\varphi) \coloneqq \norm{\varphi}_{\ell^{\infty}([-r,r],\cB_{[-r,r]})^{\iotimes m}} \in [0,\infty].
\]
As usual, write $\nu \coloneqq (\nu_{r,m} \colon C([-r,r]^m) \to [0,\infty])_{r > 0, \, m \in \N}$.
\end{notation}

The space $C_{\nu}^k(\R)$ is precisely the space $NC^k(\R)$ of ``noncommutative $C^k$ functions'' I introduced in \cite{Nikitopoulos2023n}.

\begin{proposition}\label{prop.nufam}
The family $\nu$ of possibly infinite norms satisfies $u \leq \nu$.
Moreover, $\nu$ satisfies properties \ref{item.complete}--\ref{item.consistent} from Definition \ref{def.norms1}.
In particular, $\nu \leq \beta$ by Proposition \ref{prop.Calpha}\ref{item.compcomp}.
\end{proposition}

\begin{proof}[Sketch of proof]
All the assertions in the statement follow either directly from the definitions or by standard arguments from \cite[Prop.\ 2.2.3]{Nikitopoulos2023n}.
Note that it is also easy to see directly that $\nu \leq \beta$.
\end{proof}

Of course, $C_{\nu}(\R) = \cC(\R) = C(\R) = VC(\R)$.
It is unclear to me whether $\nu \neq \beta$, $C_{\nu}(\R^m) \subsetneq \cC_{\nu}(\R^m)$, $VC(\R^m) \subsetneq C_{\nu}(\R^m)$, or $VC(\R^m) \subsetneq \cC_{\nu}(\R^m)$ whenever $m \geq 2$.

\begin{theorem}[MOIs as LFC]\label{thm.nufam}
Let $\cM$ be a von Neumann algebra and $\cA\subseteq \cM$ be a unital $C^*$-subalgebra;
for example, $\cA$ could be any unital $C^*$-algebra and $\cM$ its double dual $\cA^{**}$.
The family $\nu$ controls polynomial LFC in $\cA$.
Furthermore, if $m \geq 2$ and $\varphi \in C_{\nu}(\R^m)$, then
\begin{equation}
    \varphi_{\scA,\nu}(\mathbf{a})\sh b = \int_{\sigma(a_m)} \cdots \int_{\sigma(a_1)} \varphi(\blambda) \, P^{a_1}(\d\lambda_1)\,b_1\cdots P^{a_{m-1}}(\d\lambda_{m-1})\,b_{m-1}\,P^{a_m}(\d\lambda_m) \label{eq.MOILFC}
\end{equation}
for all $\mathbf{a} \in \cA_{\sa}^m$ and $b \in \cA^{m-1}$.
Moreover, if $(\cI,\norm{\cdot}_{\cI}) \sni \cM$ is an integral symmetrically normed ideal of $\cM$ (vid.\ \cite[Def.\ 3.1.3(b)]{Nikitopoulos2023h}), then $\nu$ controls polynomial LFC in $\cM$ with an insertion from $\cI$, and \eqref{eq.MOILFC} holds with $\varphi_{\scI,\nu}(\mathbf{a})\sh b$ in place of $\varphi_{\scA,\nu}(\mathbf{a}) \sh b$ for all $\varphi \in C_{\nu}(\R^m)$, $\mathbf{a} \in \cM_{\sa}^m$, and $b \in \cI^{m-1}$.
\end{theorem}

\begin{proof}[Sketch of proof]
Let $r > 0$ and $P \in \cP_m$.
The representation $P(\blambda) = \sum_{|\delta| \leq d} c_{\delta}\,\blambda^{\delta}$ of $P$ is an $\ell^{\infty}$-integral projective decomposition of $P|_{[-r,r]^m}$.
From this observation and the definition of the MOI
\[
(I^{\mathbf{a}}P)[b] = \int_{\sigma(a_m)} \cdots \int_{\sigma(a_1)} P(\blambda) \, P^{a_1}(\d\lambda_1)\,b_1\cdots P^{a_{m-1}}(\d\lambda_{m-1})\,b_{m-1}\,P^{a_m}(\d\lambda_m),
\]
it follows that
\begin{equation}
    (I^{\mathbf{a}}P)[b] = \sum_{|\delta| \leq d} c_{\delta} \, a_1^{\delta_1}b_1\cdots a_{m-1}^{\delta_{m-1}}b_{m-1}a_m^{\delta_m} = \ev_{\scA}^m(P)(\mathbf{a})[b] \quad \big(\mathbf{a} \in \cA_{\sa,r}^m, \; b \in \cA^{m-1}\big).\label{eq.MOIagree}
\end{equation} 
In addition, by \cite[Thm.\ 4.2.4(iii)]{Nikitopoulos2023n}, if $\mathbf{a} \in \cA_{\sa,r}^m$, then
\begin{align*}
    \norm{\ev_{\scA}^m(P)(\mathbf{a})}_{B_{m-1}(\cA^{m-1};\cA)} & = \norm{(I^{\mathbf{a}}P)|_{\cA^{m-1}}}_{B_{m-1}(\cA^{m-1};\cA)} \\
    & \leq \norm{P}_{\ell^{\infty}(\sigma(a_1),\cB_{\sigma(a_1)}\hspace{-0.1mm}) \iotimes \cdots \iotimes \ell^{\infty}(\sigma(a_m),\cB_{\sigma(a_m)}\hspace{-0.1mm})} \leq \nu_{r,m}(P).
\end{align*}
Thus, $\nu$ controls polynomial LFC in $\cA$.
Now, let
\[
\cS \coloneqq \left\{\varphi \in C_{\nu}(\R^m) : \text{\eqref{eq.MOILFC} holds for all } \mathbf{a} \in \cA_{\sa}^m \text{ and } b \in \cA^{m-1}\right\}.
\]
By \cite[Thm.\ 4.2.4(iii)]{Nikitopoulos2023n}, if $\varphi \in C_{\nu}(\R^m)$, $r > 0$, $\mathbf{a} \in \cA_{\sa,r}^m$, and $b \in \cA^{m-1}$, then
\[
\norm{(I^{\mathbf{a}}\varphi)[b]} \leq \norm{\varphi}_{\ell^{\infty}(\sigma(a_1),\cB_{\sigma(a_1)}\hspace{-0.1mm}) \iotimes \cdots \iotimes \ell^{\infty}(\sigma(a_m),\cB_{\sigma(a_m)}\hspace{-0.1mm})} \prod_{i=1}^{m-1}\norm{b_i} \leq \nu_{r,m}(\varphi)\prod_{i=1}^{m-1}\norm{b_i}.
\]
Since $\Phi_{\scA,\nu}^m \colon C_{\nu}(\R^m) \to \Cbb(\cA_{\sa}^m;B_{m-1}(\cA^{m-1};\cA))$ is continuous, it follows that $\cS$ is closed in $C_{\nu}(\R^m)$.
By \eqref{eq.MOIagree}, $\cP_m \subseteq \cS$.
Since $\cP_m$ is dense in $C_{\nu}(\R^m)$, it follows that $\cS = C_{\nu}(\R^m)$, as desired.

Let $(\cI,\norm{\cdot}_{\cI})$ be an integral symmetrically normed ideal of $\cM$.
To prove that $\nu$ controls polynomial LFC in $\cM$ with an insertion from $\cI$ and to establish that \eqref{eq.MOILFC} holds with $\varphi_{\scI,\nu}(\mathbf{a})\sh b$ in place of $\varphi_{\scA,\nu}(\mathbf{a}) \sh b$ for all $\varphi \in C_{\nu}(\R^m)$, $\mathbf{a} \in \cM_{\sa}^m$, and $b \in \cI^{m-1}$, repeat the same argument as in the previous paragraph with $\cA = \cM$, $\ev_{\scI}^{m,i}(P)$ or $\ev_{\scI}^m(P)$ (as needed) in place of $\ev_{\scA}^m(P)$, \cite[Prop.\ 4.1.7]{Nikitopoulos2023h} in place of \cite[Thm.\ 4.2.4(iii)]{Nikitopoulos2023n}, and $b \in \cI^{m-1}$ or $b \in \cM^{i-1} \times \cI \times \cM^{m-1-i}$ (as needed) in place of $b \in \cA^{m-1}$.
\end{proof}

\begin{remark}[$\cA$-specific refinement]
Let $\cM$ be a von Neumann algebra and $\cA \subseteq \cM$ be a unital $C^*$-subalgebra.
(Respectively, let $(\cI,\norm{\cdot}_{\cI})$ be an integral symmetrically normed ideal of $\cM$.)
The proof above actually shows that if
\[
\nu_{r,m}^{\scA}(\varphi) \coloneqq \sup_{\mathbf{a} \in \cA_{\sa,r}^m}\norm{\varphi|_{\sigma(a_1)\times \cdots \times \sigma(a_m)}}_{\ell^{\infty}(\sigma(a_1),\cB_{\sigma(a_1)}\hspace{-0.1mm}) \iotimes \cdots \iotimes \ell^{\infty}(\sigma(a_m),\cB_{\sigma(a_m)}\hspace{-0.1mm})} \in [0,\infty]
\]
for all $r > 0$, $m \in \N$, and $\varphi \in C([-r,r]^m)$, then $\nu^{\scA} \coloneqq \big(\nu_{r,m}^{\scA}\big)_{r > 0, \, m \in \N}$ controls polynomial LFC in $\cA$ (respectively, controls polynomial LFC in $\cM$ with an insertion from $\cI$), and \eqref{eq.MOILFC} holds with $\varphi_{\scA,\nu^{\scA}}(\mathbf{a})\sh b$ in place of $\varphi_{\scA,\nu}(\mathbf{a})\sh b$ for all $\varphi \in C_{\nu^{\scA}}(\R^m)$, $\mathbf{a} \in \cA_{\sa}^m$, and $b \in \cA^{m-1}$ (respectively, \eqref{eq.MOILFC} holds with $\varphi_{\scI,\nu^{\scA}}(\mathbf{a})\sh b$ in place of $\varphi_{\scA,\nu}(\mathbf{a})\sh b$ for all $\varphi \in C_{\nu^{\scA}}(\R^m)$, $\mathbf{a} \in \cM_{\sa}^m$, and $b \in \cI^{m-1}$).
Also, note that if $\cA$ is finite dimensional, then $\nu^{\scA} = \beta^{\scA}$.
\end{remark}

It is not clear to me whether $\nu$ is ideal for LFC.

\section{Application to functional calculus calculus}\label{sec.apptofunkycalccalc}

For the duration of this section, let $\alpha$ be a family of possibly infinite norms as in \eqref{eq.alphaintro}.

\subsection{The spaces \texorpdfstring{$\cC_{\alpha}^k(\R)$}{} and \texorpdfstring{$C_{\alpha}^k(\R)$}{}}\label{subsec.Ckalpha}

This subsection's purpose is to conduct a detailed study of the spaces $\cC_{\alpha}^k(\R)$ and $C_{\alpha}^k(\R)$ introduced in subsection \ref{subsec.results}.
First in this study is a collection of basic properties.
Before stating them, let us review a few facts about the ``$C^k$ topology'' on $C^k(\R)$, that of locally uniform convergence of derivatives of order strictly less than $k+1$.

Let $k \in \N_0 \cup \{\infty\}$.
The \textbf{$\boldsymbol{C^k}$ topology} on $C^k(\R)$ is the Fr\'echet-space topology induced by the seminorms
\[
C^k(\R) \ni f \mapsto \big\|f^{(i)}\big\|_{\ell^{\infty}([-r,r])} \in [0,\infty) \qquad (r > 0, \; 0 \leq i < k+1).
\]
By \cite[Cor.\ 2.1.4]{Nikitopoulos2023n}, if $f \in C^k(\R)$, then
\[
u_{r,i+1}\big(f^{[i]}\big) = \big\|f^{[i]}\big\|_{\ell^{\infty}([-r,r]^{i+1})} = \frac{1}{i!}\big\|f^{(i)}\big\|_{\ell^{\infty}([-r,r])} \qquad (r > 0, \; 0 \leq i < k+1).
\]
Consequently, the topology on $\cC_u^k(\R) = C^k(\R)$ defined in subsection \ref{subsec.results} is precisely the $C^k$ topology.
Since it is elementary to show that polynomials are dense in $C^k(\R)$ with the $C^k$ topology (vid.\ \cite[Lem.\ A.1.5]{Nikitopoulos2023n}), $C_u^k(\R) = \cC_u^k(\R) = C^k(\R)$.

\begin{proposition}[Properties of $\cC_{\alpha}^k(\R)$ and $C_{\alpha}^k(\R)$]\label{prop.Ckalpha}
Let $k \in \N_0 \cup \{\infty\}$ and
\[
\tilde{\alpha} = (\tilde{\alpha}_{r,m} \colon C([-r,r]^m) \to [0,\infty])_{r > 0, \, m \in \N}
\]
be another family of possibly infinite norms.
\begin{enumerate}[label=(\roman*),font=\normalfont]
    \item If $\tilde{\alpha}$ is stronger than $\alpha$, then $\cC_{\tilde{\alpha}}^k(\R) \hookrightarrow \cC_{\alpha}^k(\R)$.
    If, in addition, $\tilde{\alpha}$ is finite on polynomials, then so is $\alpha$, and $C_{\tilde{\alpha}}^k(\R) \hookrightarrow C_{\alpha}^k(\R)$.
    In particular, if $\alpha$ satisfies $u \leq \alpha \leq \beta$, then $C_{\beta}^k(\R) \hookrightarrow C_{\alpha}^k(\R) \hookrightarrow \cC_{\alpha}^k(\R) \hookrightarrow C^k(\R)$.\label{item.strongerCk}
    \item Suppose $0 \leq j < k+1$.
    If $\alpha$ is submultiplicative and compatible, then
    \[
    \alpha_{r,j+1}\big((fg)^{[j]}\big) \leq \alpha_{r,1}(1)^j\sum_{i=0}^j\alpha_{r,i+1}\big(f^{[i]}\big)\,\alpha_{r,j-i+1}\big(g^{[j-i]}\big) \qquad (r > 0)
    \]
    for all $f,g \in C^k(\R)$.
    (As usual, $0 \cdot \infty \coloneqq 0$.)
    Consequently, if $f,g \in \cC_{\alpha}^k(\R)$, then $fg \in \cC_{\alpha}^k(\R)$, and the product operation on $\cC_{\alpha}^k(\R)$ is continuous.
    If, in addition, $\alpha$ is $\ast$-isometric, then $\cC_{\alpha}^k(\R)$ is a $\ast$-algebra with a continuous $\ast$-operation.
    Finally, if $\alpha$ is ($\ast$-isometric,) submultiplicative, compatible, and finite on polynomials, then $C_{\alpha}^k(\R)$ is a ($\ast$-)subalgebra of $\cC_{\alpha}^k(\R)$.\label{item.submultstarisomCk}
    \item If $u \leq \alpha$ and $\alpha$ is complete, then $\cC_{\alpha}^k(\R)$ is complete.\label{item.compincCk}
    \item If $\alpha$ is increasing, then $\cC_{\alpha}^k(\R)$ is metrizable.
    Consequently, by the previous item, if $u \leq \alpha$ and $\alpha$ is complete and increasing, then $\cC_{\alpha}^k(\R)$ is a Fr\'{e}chet space.\label{item.incCk}
    \item If $\cS \subseteq \cC_{\alpha}^k(\R)$, then
    \[
    \cS_{\loc} \subseteq \overline{\cS} \subseteq \cC_{\alpha}^k(\R),
    \]
    where the closure above takes place in the space $\cC_{\alpha}^k(\R)$.
    (Please see Proposition \ref{prop.Calpha}\ref{item.SlocCalpha} for the definition of $\cS_{\loc}$.)\label{item.SlocCk}
\end{enumerate}
\end{proposition}

\begin{proof}
The first item follows readily from the definitions.
I take the remaining items in~turn.

\ref{item.submultstarisomCk} Let $r > 0$.
The product rule for divided differences, \cite[Prop.\ 2.1.3(iii)]{Nikitopoulos2023n}, says that if $f,g \in C^k(\R)$ and $0 \leq j < k+1$, then
\[
(fg)^{[j]} = \sum_{i=0}^j \big(f^{[i]} \otimes 1^{\otimes (j-i)}\big) \big(1^{\otimes i} \otimes g^{[j-i]}\big).
\]
Thus,
\begin{align*}
    \alpha_{r,j+1}\big((fg)^{[j]}\big) & \leq \sum_{i=0}^j \alpha_{r,j+1}\big(\big(f^{[i]} \otimes 1^{\otimes (j-i)}\big) \big(1^{\otimes i} \otimes g^{[j-i]}\big)\big) \\
    & \leq \sum_{i=0}^j \alpha_{r,j+1}\big(f^{[i]} \otimes 1^{\otimes (j-i)}\big) \,\alpha_{r,j+1}\big(1^{\otimes i} \otimes g^{[j-i]}\big) \tag{submultiplicativity} \\
    & \leq \alpha_{r,1}(1)^j \sum_{i=0}^j \alpha_{r,i+1}\big(f^{[i]}\big) \,\alpha_{r,j-i+1}\big(g^{[j-i]}\big), \tag{compatibility}
\end{align*}
as claimed.
The rest of this item's claims follow readily from the definitions.

\ref{item.compincCk} Let $(f_j)_{j \in J}$ be a Cauchy net in $\cC_{\alpha}^k(\R)$.
Since $\cC_{\alpha}^k(\R) \hookrightarrow \cC_u^k(\R) = C^k(\R)$, the net $(f_j)_{j \in J}$ is also Cauchy in $C^k(\R)$.
Since $C^k(\R)$ is complete, there exists an $f \in C^k(\R)$ such that $(f_j)_{j \in J}$ converges to $f$ in $C^k(\R)$ and therefore $\big(f_j^{[i]}\big)_{j \in J}$ converges to $f^{[i]}$ in $C(\R^{i+1})$ whenever $0 \leq i < k+1$.
Now, if $0 \leq i < k+1$, then $\big(f_j^{[i]}\big)_{j \in J}$ is Cauchy and therefore, by Proposition \ref{prop.Calpha}\ref{item.compinc}, convergent in $\cC_{\alpha}(\R^{i+1})$.
Since we already know that $\big(f_j^{[i]}\big)_{j \in J}$ converges to $f^{[i]}$ in $C(\R^{i+1})$ and $\cC_{\alpha}(\R^{i+1}) \hookrightarrow C(\R^{i+1})$, it follows that $f^{[i]} \in \cC_{\alpha}(\R^{i+1})$ and $\big(f_j^{[i]}\big)_{j \in J}$ converges to $f^{[i]}$ in $\cC_{\alpha}(\R^{i+1})$.
This proves that $f \in \cC_{\alpha}^k(\R)$ and $(f_j)_{j \in J}$ converges to $f$ in $\cC_{\alpha}^k(\R)$.
Thus, $\cC_{\alpha}^k(\R)$ is complete.
\pagebreak

\ref{item.incCk} If $\alpha$ is increasing, then the topology of $\cC_{\alpha}^k(\R)$ is induced by the countable family
\[
\cC_{\alpha}^k(\R) \ni f \mapsto \alpha_{n,i+1}\big(f^{[i]}\big) \in [0,\infty)  \qquad (n \in \N, \; 0 \leq i < k+1)
\]
of seminorms, from which it follows that $\cC_{\alpha}^k(\R)$ is metrizable.

\ref{item.SlocCk} Let $f \in \cS_{\loc}$, and for each $r > 0$, let $g_r \in \cS$ be such that $g_r|_{[-1-r,r+1]} = f|_{[-1-r,r+1]}$.
I claim that $g \in \cC_{\alpha}^k(\R)$ and $(g_n)_{n \in \N}$ is a sequence in $\cS$ converging in $\cC_{\alpha}^k(\R)$ to $f$, which implies the result.
Indeed, if $r > 0$ and $0 \leq i < k+1$, then
\[
\alpha_{r,i+1}\big(f^{[i]}\big) = \alpha_{r,i+1}\big(f^{[i]}|_{[-r,r]^{i+1}}\big) = \alpha_{r,i+1}\big(g_r^{[i]}|_{[-r,r]^{i+1}}\big) = \alpha_{r,i+1}\big(g_r^{[i]}\big) < \infty
\]
because $g_r \in \cS \subseteq \cC_{\alpha}^k(\R)$;
thus, $f \in \cC_{\alpha}^k(\R)$. 
Now, $(g_n)_{n \in \N}$ converges in $\cC_{\alpha}^k(\R)$ to $f$ because if $r > 0$, $N$ is any natural number larger than $r$, and $0 \leq i < k+1$, then
\[
g_N^{[i]}|_{[-r,r]^{i+1}} = \big(g_N^{[i]}|_{[-N,N]^{i+1}}\big)|_{[-r,r]^{i+1}} = \big(f^{[i]}|_{[-N,N]^{i+1}}\big)|_{[-r,r]^{i+1}} = f^{[i]}|_{[-r,r]^{i+1}}.
\]
Consequently, $\alpha_{r,i+1}\big((g_N - f)^{[i]}\big) = 0$.
In particular, $\alpha_{r,i+1}\big((g_n - f)^{[i]}\big) \to 0$ as $n \to \infty$.
Thus, $(g_n)_{n \in \N}$ converges in $\cC_{\alpha}^k(\R)$ to $f$, as claimed.
\end{proof}

Now, suppose $\alpha$ is finite on polynomials.
Certainly, $C_{\alpha}^k(\R) \subseteq \cC_{\alpha}^k(\R)$.
However, $C_{\alpha}^k(\R)$ is actually contained in a smaller closed subspace of $\cC_{\alpha}^k(\R)$.

\begin{notation}\label{nota.mathbfCkalpha}
Let $k \in \N_0 \cup \{\infty\}$, and suppose $\alpha$ is finite on polynomials.
Write
\[
\mathbf{C}_{\alpha}^k(\R) \coloneqq \big\{ f \in C^k(\R) : f^{[i]} \in C_{\alpha}\big(\R^{i+1}\big) \text{ whenever } 0 \leq i < k+1\big\}.
\]
Observe that $C_{\alpha}^k(\R) \subseteq \mathbf{C}_{\alpha}^k(\R) \subseteq \cC_{\alpha}^k(\R)$.
\end{notation}

As we shall see, it is often substantially less onerous to show that a function belongs to $\mathbf{C}_{\alpha}^k(\R)$ than to show it belongs to $C_{\alpha}^k(\R)$.
However, the main result of the next subsection concerns functions in $C_{\alpha}^k(\R)$.
It is therefore useful to know when $C_{\alpha}^k(\R) = \mathbf{C}_{\alpha}^k(\R)$.
The first main result of this subsection, Theorem \ref{thm.mathbfCkalpha=Ckalpha} below, gives a criterion ensuring this.

\begin{definition}[Wiener space]\label{def.Wk}
Let $k \in \N_0$.
The \textbf{$\boldsymbol{k^{\text{\textbf{th}}}}$ Wiener space} $W_k(\R)$ is the set of functions $f \colon \R \to \C$ with the property that there exists a (necessarily unique) Borel complex measure $\mu$ on $\R$ such that
\[
\mu_{(k)} \coloneqq \int_{\R} |\xi|^k \,|\mu|(\d \xi) < \infty \; \text{ and } \; f(\lambda) = \int_{\R} e^{i\lambda \xi} \, \mu(\d\xi) \qquad (\lambda \in \R).
\]
Recall that $|\mu|$ is the total variation of $\mu$.
Also, define $W_{\infty}(\R) \coloneqq \bigcap_{k \in \N} W_k(\R)$.
\end{definition}

\begin{theorem}\label{thm.mathbfCkalpha=Ckalpha}
Suppose $u \leq \alpha \leq \beta$ and $\alpha$ is increasing, complete, and translation invariant, e.g., $\alpha$ is ideal for LFC in some unital $C^*$-algebra.
Let $k \in \N_0 \cup \{\infty\}$.
\begin{enumerate}[font=\normalfont,label=(\roman*)]
    \item $\mathbf{C}_{\alpha}^k(\R) = C_{\alpha}^k(\R)$, i.e., polynomials are dense in $\mathbf{C}_{\alpha}^k(\R)$;
    actually,
    \[
    C_{\alpha}^k(\R) = \big\{f \in C^k(\R) : f^{[k]} \in C_{\alpha}\big(\R^{k+1}\big)\big\}.
    \]
    In particular, if, in addition, $\cC_{\alpha}(\R^m) = C_{\alpha}(\R^m)$ for all $m \in \N$, then $\cC_{\alpha}^k(\R) = C_{\alpha}^k(\R)$, i.e., polynomials are dense in $\cC_{\alpha}^k(\R)$.
    \item $W_k(\R) \subseteq \mathbf{C}_{\alpha}^k(\R)$, and $W_k(\R)$ is dense in $\mathbf{C}_{\alpha}^k(\R)$.
\end{enumerate}
\end{theorem}
\pagebreak

In particular,
\[
\cC_u^k(\R) = \mathbf{C}_u^k(\R) = C_u^k(\R) = C^k(\R),
\]
as I already pointed out earlier, and
\[
\cC_{\beta}^k(\R) = \mathbf{C}_{\beta}^k(\R) = C_{\beta}^k(\R) = \big\{f \in C^k(\R) : f^{[k]} \in VC\big(\R^{k+1}\big)\big\},
\]
which I claimed (in part) in subsection \ref{subsec.results}.
Also,
\[
\mathbf{C}_{\beta^{\scA}}^k(\R) = C_{\beta^{\scA}}^k(\R) = \big\{ f \in C^k(\R) : f^{[k]} \in C_{\beta^{\scA}}\big(\R^{k+1}\big)\big\}
\]
for every unital $C^*$-algebra $\cA$.
Finally,
\[
\mathbf{C}_{\nu}^k(\R) = C_{\nu}^k(\R) = \big\{ f \in C^k(\R) : f^{[k]} \in C_{\nu}\big(\R^{k+1}\big)\big\},
\]
which is a new characterization of my space $NC^k(\R) = C_{\nu}^k(\R)$ of ``noncommutative $C^k$ functions'' from \cite{Nikitopoulos2023n}.
(Please see subsection \ref{subsec.MOIfam} for the family $\nu$.)

The proof of Theorem \ref{thm.mathbfCkalpha=Ckalpha} comprises two key steps:
showing that (1) $W_k(\R)$ and $\cP$ have the same closures, namely $C_{\alpha}^k(\R)$, in $\mathbf{C}_{\alpha}^k(\R)$ and (2) every function in $\mathbf{C}_{\alpha}^k(\R)$ can be approximated via mollification by functions in $W_k(\R)_{\loc} \subseteq C_{\alpha}^k(\R)$.
Both steps make heavy use of a handy expression from \cite{Nikitopoulos2023n}, which I reproduce momentarily, for the divided differences of $C^k$ functions.

\begin{notation}\label{nota.simplices}
Write $\R_+ \coloneqq [0,\infty)$.
If $k \in \N$, then
\begin{align*}
     \Sigma_k & \coloneqq \big\{s = (s_1,\ldots,s_k) \in \R_+^k : |s| = s_1+\cdots+s_k \leq 1\big\}, \; \text{ and} \\
     \Delta_k & \coloneqq \big\{\mathbf{t} = (t_1,\ldots,t_{k+1}) \in \R_+^{k+1} : t_1+\cdots+t_{k+1} = 1\big\}.
\end{align*}
Also, $\rho_k$ is the pushforward of the $k$-dimensional Lebesgue measure on $\Sigma_k$ by the homeomorphism $\Sigma_k \ni s \mapsto (s,1-|s|) \in \Delta_k$.
Explicitly, $\rho_k$ is the Borel measure on $\Delta_k$ characterized~by
\[
\int_{\Delta_k} \varphi(\mathbf{t}) \, \rho_k(\d\mathbf{t}) = \int_{\Sigma_k} \varphi(s,1-|s|) \, \d s
\]
for all bounded Borel $\varphi \colon \Delta_k \to \C$.
In particular, $\rho_k(\Delta_k) = 1/k!$, as the reader may verify.
\end{notation}

Let $k \in \N_0$.
If $g \in C^k(\R)$, then \cite[Prop.\ 2.1.3(ii)]{Nikitopoulos2023n} says that
\begin{equation}
    g^{[k]}(\blambda) = \int_{\Delta_k} g^{(k)}(\mathbf{t} \boldsymbol{\cdot} \blambda)\,\rho_k(\d\mathbf{t}) \qquad \big(\blambda \in \R^{k+1}\big).\label{eq.divdiff}
\end{equation}
This results in a nice formula for the divided differences of functions in the $k^{\text{th}}$ Wiener space.
Indeed, if $g = \int_{\R} e^{i \boldsymbol{\cdot}\xi}\,\mu(\d\xi) \in W_k(\R)$ and $j=0,\ldots,k$, then
\[
g^{(j)}(\lambda) = \int_{\R} e^{i\lambda \xi}(i\xi)^j\,\mu(\d\xi) \qquad (\lambda \in \R).
\]
Consequently, 
\begin{equation}
    g^{[j]}(\blambda) = \int_{\Delta_j}\int_{\R} e^{i(\mathbf{t} \boldsymbol{\cdot} \blambda)\xi} (i\xi)^j\,\mu(\d\xi)\,\rho_j(\d\mathbf{t}) \qquad \big(\blambda \in \R^{j+1}\big)\label{eq.divdiffWk}
\end{equation}
by \eqref{eq.divdiff}.

\begin{lemma}\label{lem.sameclosures}
If $\alpha \leq \beta$ and $k \in \N_0 \cup \{\infty\}$, then $W_k(\R) \subseteq C_{\alpha}^k(\R)$.
Since $C^{k+1}(\R) \subseteq W_k(\R)_{\loc}$ by \cite[Lem.\ 3.2.3(iii)]{Nikitopoulos2023n}, it follows from Proposition \ref{prop.Ckalpha}\ref{item.SlocCk} that $C^{k+1}(\R)$ is contained in the closure of $W_k(\R)$ in $C_{\alpha}^k(\R)$.
Thus, since polynomials are smooth, $W_k(\R)$ is dense in $C_{\alpha}^k(\R)$.
\end{lemma}
\pagebreak

\begin{proof}
It suffices to prove that $W_k(\R) \subseteq C_{\beta}^k(\R)$.
To this end, let $f = \int_{\R} e^{i\boldsymbol{\cdot}\xi}\,\mu(\d\xi) \in W_k(\R)$.
Now, for each $n \in \N$, define $\mu_n(\d\xi) \coloneqq 1_{[-n,n]}(\xi) \, \mu(\d\xi)$ and
\[
f_n(\lambda) \coloneqq \int_{\R} e^{i\lambda\xi}\,\mu_n(\d\xi) = \int_{\R} e^{i\lambda \xi}1_{[-n,n]}(\xi) \, \mu(\d\xi) \qquad (\lambda \in \R).
\]
Then $f_n \in W_k(\R)$, and $\supp |\mu_n| \subseteq [-n,n]$ for all $n \in \N$.
If $0 \leq j < k+1$, then \eqref{eq.divdiffWk} applied to the function $g = f-f_n \in W_k(\R)$ yields
\[
(f-f_n)^{[j]}(\blambda) = \int_{\Delta_j}\int_{\R} e^{i(\mathbf{t} \boldsymbol{\cdot} \blambda)\xi}(i\xi)^j(1-1_{[-n,n]}(\xi))\,\mu(\d\xi)\,\rho_j(\d\mathbf{t}) \qquad \big(\blambda \in \R^{j+1}\big).
\]
Consequently, by inequality \eqref{eq.multivarWiener} from Example \ref{ex.multivarWiener} applied to the function $\varphi = f-f_n$ and the dominated convergence theorem,
\begin{align*}
    \sup_{r > 0}\beta_{r,j+1}\big((f-f_n)^{[j]}\big) & \leq  \int_{\Delta_j}\int_{\R}|\xi|^j(1-1_{[-n,n]}(\xi)) \, |\mu|(\d\xi)\,\rho_j(\d\mathbf{t}) \\
    & = \frac{1}{j!}\int_{\R}|\xi|^j(1-1_{[-n,n]}(\xi)) \, |\mu|(\d\xi) \xrightarrow{n \to \infty} 0.
\end{align*}
In particular, $f_n \to f$ in $\cC_{\beta}^k(\R)$ as $n \to \infty$.
It therefore suffices to assume $\supp |\mu|$ is compact.

Suppose $R > 0$ and $\supp |\mu| \subseteq [-R,R]$.
Then $\int_{\R} |f|\,d|\mu| \leq \mu_{(0)}\,\norm{f}_{\ell^{\infty}([-R,R])}$ for all Borel-measurable functions $f \colon \R \to \C$.
In particular, $\mu_{(m)} \leq R^m\mu_{(0)} < \infty$ for all $m \in \N$.
Therefore, for any $n \in \N$,
\[
q_n(\lambda) \coloneqq \int_{\R}\sum_{m=0}^n\frac{(i\lambda \xi)^m}{m!} \, \mu(\d\xi) = \sum_{m=0}^n\frac{(i\lambda)^m}{m!}\int_{\R} \xi^m \, \mu(\d\xi) \in \cP
\]
is well defined.
I claim that $q_n \to f$ in $\cC_{\beta}^k(\R)$ as $n \to \infty$.
Indeed, since
\[
e^{i\lambda \xi} = \sum_{m=0}^{\infty}\frac{(i\lambda\xi)^m}{m!} \; \text{ and } \; \int_{\R} \sum_{m=0}^{\infty}\frac{|\lambda\xi|^m}{m!}\,\mu(\d\xi) = \int_{\R} e^{|\lambda\xi|}\,|\mu|(\d\xi) \leq e^{|\lambda|R}\mu_{(0)},
\]
the dominated convergence theorem yields
\[
f(\lambda) - q_n(\lambda) = \int_{\R}\sum_{m=n+1}^{\infty}\frac{(i\lambda \xi)^m}{m!} \, \mu(\d\xi) = \sum_{m=n+1}^{\infty}\frac{(i\lambda)^m}{m!}\int_{\R}\xi^m \, \mu(\d\xi) \qquad (\lambda \in \R).
\]
Consequently, by \cite[eq.\ (10)]{Nikitopoulos2023n} (the formula for divided differences of polynomials) and a simple limiting argument, if $0 \leq j < k+1$, then
\[
(f-q_n)^{[j]}(\blambda) = \sum_{m=n+1}^{\infty}\frac{i^m}{m!}\int_{\R}\xi^m \, \mu(\d\xi)\sum_{|\gamma| = m-j} \blambda^{\gamma} \qquad \big(\blambda \in \R^{j+1}\big).
\]
Therefore, using the fact that $\big\{\gamma \in \N_0^{j+1} : |\gamma| = m-j\big\}$ has $\binom{m}{m-j} \leq 2^m$ elements,
\[
\beta_{r,j+1}\big((f-q_n)^{[j]}\big) \leq \sum_{m=n+1}^{\infty}\binom{m}{m-j}\frac{r^{m-j}}{m!}\mu_{(m)} \leq \frac{\mu_{(0)}}{r^j}\sum_{m=n+1}^{\infty}\frac{(2rR)^m}{m!} \xrightarrow{n \to \infty} 0 \qquad (r > 0).
\]
Thus, $q_n \to f$ in $\cC_{\beta}^k(\R)$ as $n \to \infty$.
In particular, $f \in C_{\beta}^k(\R)$, as desired.
\end{proof}

\begin{lemma}\label{lem.mollify}
Let $\eta \in C_c^{\infty}(\R)$ be such that $\int_{\R} \eta(x)\,\d x = 1$, and suppose $u \leq \alpha$ and $\alpha$ is finite on polynomials, increasing, complete, and translation invariant.
If $m \in \N$, $\varphi \in C_{\alpha}(\R^m)$,~and
\[
\varphi_{\e}(\blambda) \coloneqq \int_{\R} \eta(x) \,\varphi(\lambda_1-\e x, \ldots, \lambda_m - \e x)\,\d x \qquad (\blambda \in \R^m, \; \e > 0),
\]
then $\varphi_{\e} \in C_{\alpha}(\R^m)$ for all $\e > 0$, and $\varphi_{\e} \to \varphi$ in $C_{\alpha}(\R^m)$ as $\e \searrow 0$.
\end{lemma}

\begin{proof}
I shall freely use basic facts and terminology concerning strong/Bochner measurability and integrability of functions with values in Fr\'echet spaces.
Please see \cite[\S1.1]{Nikitopoulos2024} for the requisite background.

Recall that if $S$ is a set and $s \in S$, then $s_{(m)}$ is the element of $S^m$ with all components equal to $s$.
Let $\e > 0$.
By Proposition \ref{prop.Calpha}\ref{item.transinv}, the function
\[
\R \ni x \mapsto F_{\e}(x) \coloneqq \eta(x)\,\varphi\big(\cdot - (\e x)_{(m)}\big) \in C_{\alpha}(\R^m)
\]
is well defined and continuous because $\alpha$ is translation invariant, increasing, and finite on polynomials.
Consequently, $F_{\e}$ is strongly measurable.
Proposition \ref{prop.Calpha}\ref{item.transinv} also implies that $F_{\e} \to F \coloneqq \eta(\cdot) \, \varphi$ pointwise (as functions from $\R$ to $C_{\alpha}(\R^m)$) as $\e \searrow 0$.
In addition, if $R > 0$, $\supp \eta \subseteq [-R,R]$, and $r > 0$, then
\begin{align*}
    \int_{\R} \sup_{0 < \delta \leq \e}\alpha_{r,m}(F_{\delta}(x))\,\d x & = \int_{-R}^R |\eta(x)| \,\sup_{0 < \delta \leq \e}\alpha_{r,m}\big(\varphi\big(\cdot - (\delta x)_{(m)}\big)\big) \,\d x \\
    & \leq \alpha_{r+\e R, m}(\varphi)\, \norm{\eta}_{L^1} < \infty
\end{align*}
because $\alpha$ is translation invariant and increasing.
Since $(\alpha_{r,m})_{r > 0}$ induces the topology of the Fr\'echet space $C_{\alpha}(\R^m)$, the inequality above implies that $F_{\e}$ is strongly integrable and, by the dominated convergence theorem for (Fr\'echet space--valued) Bochner integrals, $\int_{\R} F_{\e}(x) \to \int_{\R} F(x)\,\d x = \varphi \int_{\R} \eta(x) \,\d x = \varphi$ in $C_{\alpha}(\R^m)$ as $\e \searrow 0$.
Finally, by applying the evaluation functionals $C_{\alpha}(\R^m) \ni \psi \mapsto \psi(\blambda) \in \C$ ($\blambda \in \R^m$), which are continuous because $u \leq \alpha$, to the Bochner integral $\int_{\R} F_{\e}(x)\,\d x$, it follows that $\varphi_{\e} = \int_{\R} F_{\e}(x) \,\d x$ for all $\e > 0$.
This completes the proof.
\end{proof}

\begin{proof}[Proof of Theorem \ref{thm.mathbfCkalpha=Ckalpha}]
Let $\eta \in C_c^{\infty}(\R)$ be such that $\int_{\R} \eta(x)\,\d x = 1$, and define
\[
\eta_{\e}(x) \coloneqq \frac1\e\eta\left(\frac{x}{\e}\right) \qquad (x \in \R, \; \e > 0).
\]
If $f \in \mathbf{C}_{\alpha}^k(\R)$, then $f \ast \eta_{\e} \in C^{\infty}(\R)$ for all $\e > 0$.
In particular, by Lemma \ref{lem.sameclosures}, $f \ast \eta_{\e} \in C_{\alpha}^k(\R)$ for all $\e > 0$.

Next, I argue that $f \ast \eta_{\e} \to f$ in $\mathbf{C}_{\alpha}^k(\R)$ as $\e \searrow 0$.
To this end, note that if $g \in C(\R)$, $i \in \N$, $\e > 0$, and $\blambda \in \R^{i+1}$, then
\begin{align*}
    \int_{\Delta_i} (g \ast \eta_{\e})(\mathbf{t} \boldsymbol{\cdot} \blambda) \, \rho_i(\d\mathbf{t}) & = \int_{\Delta_i}\int_{\R} g(\mathbf{t} \boldsymbol{\cdot} \blambda-x)\, \eta_{\e}(x)\,\d x \, \rho_i(\d\mathbf{t}) \\
    & = \int_{\Delta_i}\int_{\R} g\big(\mathbf{t} \boldsymbol{\cdot} \big(\blambda-x_{(i+1)}\big)\big)\, \eta_{\e}(x)\,\d x \, \rho_i(\d\mathbf{t}) \\
    & = \int_{\R} \eta_{\e}(x) \int_{\Delta_i}  g\big(\mathbf{t} \boldsymbol{\cdot} \big(\blambda-x_{(i+1)}\big)\big) \, \rho_i(\d\mathbf{t}) \,\d x \\
    & = \int_{\R} \eta(y) \int_{\Delta_i}  g\big(\mathbf{t} \boldsymbol{\cdot} \big(\blambda-(\e y)_{(i+1)}\big)\big) \, \rho_i(\d\mathbf{t})\,\d y
\end{align*}
by Fubini's theorem and the change of variable $y \coloneqq x/\e$.
It follows from \eqref{eq.divdiff} (twice) that if $0 \leq i < k+1$ and $\blambda \in \R^{i+1}$, then
\begin{align*}
    (f \ast \eta_{\e})^{[i]}(\blambda) & = \int_{\Delta_i} (f \ast \eta_{\e})^{(i)}(\mathbf{t} \boldsymbol{\cdot} \blambda) \, \rho_i(\d\mathbf{t}) = \int_{\Delta_i} \big(f^{(i)} \ast \eta_{\e}\big)(\mathbf{t} \boldsymbol{\cdot} \blambda) \, \rho_i(\d\mathbf{t}) \\
    & = \int_{\R} \eta(y) \int_{\Delta_i}  f^{(i)}\big(\mathbf{t} \boldsymbol{\cdot} \big(\blambda-(\e y)_{(i+1)}\big)\big) \, \rho_i(\d\mathbf{t})\,\d y \\
    & = \int_{\R} \eta(y)\,f^{[i]}\big(\blambda - (\e y)_{(i+1)}\big)\,\d y.
\end{align*}
Therefore, by Lemma \ref{lem.mollify}, $(f \ast \eta_{\e})^{[i]} \to f^{[i]}$ in $C_{\alpha}\big(\R^{i+1}\big)$ as $\e \searrow 0$.
In other words, $f \ast \eta_{\e} \to f$ in $\mathbf{C}_{\alpha}^k(\R)$ as $\e \searrow 0$, as desired.

Finally, since $C^{k+1}(\R) \subseteq C_{\beta}^k(\R)$ by Lemma \ref{lem.sameclosures}, it is automatic that if $f \in C^k(\R)$, then
\[
f^{[i]} \in VC\big(\R^{i+1}\big) = C_{\beta}\big(\R^{i+1}\big) \subseteq C_{\alpha}\big(\R^{i+1}\big) \qquad (0 \leq i < k).
\]
Consequently, $C_{\alpha}^k(\R) = \mathbf{C}_{\alpha}^k(\R) = \big\{f \in C^k(\R) : f^{[k]} \in C_{\alpha}\big(\R^{k+1}\big)\big\}$.
\end{proof}

I wrap up this subsection by proving Theorem \ref{thm.BesovHolderVaropoulos}, the cheeky statement of which is that if $f$ is ``slightly better than $C^k$,'' then $f$ is Varopoulos $C^k$, i.e., $f \in VC^k(\R) = C_{\beta}^k(\R) = \cC_{\beta}^k(\R)$.

\begin{definition}[Besov spaces]\label{def.Besov}
Let $\eta \in C_c^{\infty}(\R)$ be such that $0 \leq \eta \leq 1$, $\supp \eta \subseteq [-2,2]$, and $\eta \equiv 1$ on $[-1,1]$.
Define
\[
\eta_i(\xi) \coloneqq \eta\big(2^{-i}\xi\big) - \eta\big(2^{-i+1}\xi\big) \qquad (i \in \Z, \; \xi \in \R).
\]
Now, for $(s,p,q) \in \R \times [1,\infty]^2$ and $f \in \sS'(\R) \coloneqq \{$tempered distributions on $\R\}$, define
\begin{align*}
    \norm{f}_{\dot{B}_q^{s,p}} & \coloneqq \big\|\big(2^{is}\|\wch{\eta}_i \ast f\|_{L^p}\big)_{i \in \Z}  \big\|_{\ell^q(\Z)} \in [0,\infty] \; \text{ and} \\
    \norm{f}_{B_q^{s,p}} & \coloneqq \|\wch{\eta} \ast f\|_{L^p} + \big\|\big(2^{is}\|\wch{\eta}_i \ast f\|_{L^p}\big)_{i \in \N}  \big\|_{\ell^q(\N)} \in [0,\infty].
\end{align*}
Above,
\[
\wch{g}(x) \coloneqq \frac{1}{2\pi} \int_{\R} e^{ix\xi}g(\xi)\,\d \xi \qquad (x \in \R)
\]
is the inverse Fourier transform of $g$.
The space
\[
\dot{B}_q^{s,p}(\R) \coloneqq \big\{f \in \sS'(\R) : \norm{f}_{\dot{B}_q^{s,p}} < \infty\big\}
\]
is called the \textbf{homogeneous $\boldsymbol{(s,p,q)}$-Besov space}, and the space
\[
B_q^{s,p}(\R) \coloneqq \big\{f \in \sS'(\R) : \norm{f}_{B_q^{s,p}} < \infty\big\}
\]
is called the \textbf{inhomogeneous $\boldsymbol{(s,p,q)}$-Besov space}.
\end{definition}

Please see \cite[\S3]{Nikitopoulos2023n} and the references therein for background on Besov spaces.
In particular, as is explained therein, $W_k(\R) \subseteq B_1^{k,\infty}(\R) \subseteq \dot{B}_1^{k,\infty}(\R) \subseteq C^k(\R)$ for all $k \in \N$.
The final key to upgrading the latter containment to $\dot{B}_1^{k,\infty}(\R) \subseteq VC^k(\R)$, the first part of Theorem \ref{thm.BesovHolderVaropoulos}, is the following celebrated result of Peller.

\begin{theorem}[Peller]\label{thm.Peller}
If $k \in \N$, then there exists a constant $c_k < \infty$ such that for all $f \in B_1^{k,\infty}(\R)$, there exist a $\sigma$-finite measure space $(\Sigma,\sH,\rho)$ and measurable functions $\varphi_1,\ldots,\varphi_{k+1} \colon \R \times \Sigma \to \C$ such that $\varphi_i(\cdot,\sigma) \in C(\R)$ for all $i =1,\ldots,k+1$ and $\sigma \in \Sigma$,
\[
\int_{\Sigma} \prod_{i=1}^{k+1}\norm{\varphi_i(\cdot,\sigma)}_{\ell^{\infty}(\R)}\,\rho(\d\sigma) \leq c_k\norm{f}_{B_1^{k,\infty}}, \; \text{ and } \; f^{[k]}(\blambda) = \int_{\Sigma}\prod_{i=1}^{k+1}\varphi_i(\lambda_i,\sigma) \,\rho(\d\sigma)
\]
for all $\blambda \in \R^{k+1}$.
\end{theorem}

Slightly stronger forms of this result are stated as \cite[Thm.\ 5.5]{Peller2006}, \cite[Thm.\ 2.2.1]{Peller2016}, and \cite[Thm.\ 4.3.13]{Nikitopoulos2023h}.
Please see \cite[App.\ A]{Nikitopoulos2023h} for a detailed and mostly self-contained proof.

\begin{corollary}\label{cor.BesovVCk}
If $k \in \N$, then $B_1^{k,\infty}(\R) \subseteq VC^k(\R)$.
\end{corollary}

\begin{proof}
Let $f \in B_1^{k,\infty}(\R)$.
By Theorems \ref{thm.Peller} and \ref{thm.IPTPV} (and Corollary \ref{cor.polydenseinVC}), $f^{[k]} \in VC\big(\R^{k+1}\big)$.
Consequently, by Theorem \ref{thm.mathbfCkalpha=Ckalpha} with $\alpha = \beta$, $f \in C_{\beta}^k(\R) = VC^k(\R)$.
\end{proof}

\begin{proof}[Proof of Theorem \ref{thm.BesovHolderVaropoulos}]
Let $f \in \dot{B}_1^{k,\infty}(\R)$, and define $g \coloneqq (\wch{\eta} + \wch{\eta}_1)\ast f \in \sS'(\R)$.
Since $g$ is a tempered distribution with compactly supported Fourier transform, $g \in C^{\infty}(\R)$ by the Paley--Wiener theorem.
It is not difficult to show (vid.\ \cite[Lem.\ 3.3.6]{Nikitopoulos2023n}) that $f-g \in B_1^{k,\infty}(\R)$ as well.
Since $C^{\infty}(\R) \subseteq VC^k(\R)$ by Lemma \ref{lem.sameclosures} with $\alpha = \beta$ and $B_1^{k,\infty}(\R) \subseteq VC^k(\R)$ by Corollary \ref{cor.BesovVCk}, $f = (f-g)+g \in VC^k(\R)$.

If $\e > 1$, then $C_{\loc}^{k,\e}(\R) \subseteq \cP$, so it is clear that $C_{\loc}^{k,\e}(\R) \subseteq VC^k(\R)$.
Suppose $0 < \e \leq 1$.
It is argued in the proof of \cite[Thm.\ 3.3.11]{Nikitopoulos2023n} that $C_{\loc}^{k,\e}(\R) \subseteq B_1^{k,\infty}(\R)_{\loc}$.
Consequently, the containment $C_{\loc}^{k,\e}(\R) \subseteq VC^k(\R)$ follows via Proposition \ref{prop.Ckalpha}\ref{item.SlocCk} from Corollary \ref{cor.BesovVCk}.
\end{proof}

\begin{remark}[Alternate proof that $\dot{B}_1^{k,\infty}(\R) \subseteq VC^k(\R)$]\label{rem.BesovVCkwodensitythm}
It is possible to prove that $\dot{B}_1^{k,\infty}(\R) \subseteq VC^k(\R)$ without the full force of Theorem \ref{thm.mathbfCkalpha=Ckalpha}.
Indeed, note that
\[
\sup_{r > 0}\beta_{r,k+1}\big(f^{[k]}\big) \leq c_k\norm{f}_{B_1^{k,\infty}} \qquad \big(f \in B_1^{k,\infty}(\R)\big)
\]
by Theorems \ref{thm.Peller} and \ref{thm.IPTPV}.
Consequently, if $f \in \dot{B}_1^{k,\infty}(\R)$ and $(f_n)_{n \in \N}$ is the inhomogeneous Littlewood--Paley sequence of $f$, then \cite[Lem.\ 3.3.6]{Nikitopoulos2023n} says (that $f_n \in C^{\infty}(\R)$ and) $\norm{f_n-f}_{B_1^{k,\infty}} \to 0$ as $n \to \infty$.
Since $B_1^{i,\infty}(\R) \hookrightarrow B_1^{k,\infty}(\R)$ for all $i=1,\ldots,k$ and $C^{\infty}(\R) \subseteq VC^k(\R)$ by Lemma \ref{lem.sameclosures}, we conclude that $f_n \in VC^k(\R)$ for all $n \in \N$, $f \in \cC_{\beta}^k(\R)$, and $f_n \to f$ in $\cC_{\beta}^k(\R)$ as $n \to \infty$.
Thus, $f \in VC^k(\R)$.
\end{remark}

\begin{corollary}\label{cor.VCksubsetneqCk}
If $k \in \N$, then $W_k(\R)_{\loc} \subsetneq VC^k(\R) \subsetneq C^k(\R)$.
\end{corollary}

\begin{proof}
I produced a specific function belonging to $C_{\loc}^{k,1/4}(\R) \setminus W_k(\R)_{\loc}$ in \cite[Thm.\ 3.4.1]{Nikitopoulos2023n}.\footnote{There are minor misprints in the proof of \cite[Thm.\ 3.4.1]{Nikitopoulos2023n}.
Specifically, the $O(\zeta^{-3/2})$ in \cite[eq.\ (19)]{Nikitopoulos2023n} should be $O(\zeta^{-1})$, which means the $O(\xi^{-3/2})$ in the next display equation should be $O(\xi^{-5/4})$.
The same typo is present in the second and last display equations in \cite[\S{A}.2]{Nikitopoulos2023n}, in which $O(\zeta^{-3/2})$ should be $O(\zeta^{-1})$ as well.}
Since $C_{\loc}^{k,1/4}(\R) \subseteq VC^k(\R)$ by Theorem \ref{thm.BesovHolderVaropoulos}, the containment $W_k(\R)_{\loc} \subseteq VC^k(\R)$ is strict.
Since $\nu \leq \beta$, $VC^k(\R) = C_{\beta}^k(\R) \subseteq C_{\nu}^k(\R) = NC^k(\R)$.
(Recall that $\nu$ is defined and studied in subsection \ref{subsec.MOIfam}.)
Since $NC^k(\R) \subsetneq C^k(\R)$ by \cite[Thm.\ 4.4.1]{Nikitopoulos2023n}, $VC^k(\R) \subsetneq C^k(\R)$ as well.
\end{proof}

\subsection{Higher derivatives of maps induced by functional calculus}\label{subsec.derivatives}

We now harvest the fruits of our labors by proving the following massive generalization of Theorem \ref{thm.derformintro}.
For the duration of this subsection, let $k \in \N$, $\cA$ be a unital $C^*$-algebra, and $\cI \sni \cA$.
Also, write $\cI_{\sa} \coloneqq \cI \cap \cA_{\sa}$, considered a real normed vector space under $\norm{\cdot}_{\cI}$.

\begin{theorem}\label{thm.derform}
Let $a \in \cA_{\sa}$.
If $\alpha$ controls polynomial LFC in $\cI$ and $f \in C_{\alpha}^k(\R)$, then $f(a+b)-f(a) \in \cI$ for all $b \in \cI_{\sa}$, and the function
\[
\cI_{\sa} \ni b \mapsto f_{a,\scI}(b) \coloneqq f(a+b) - f(a) \in \cI
\]
is $C^k$ with respect to $\norm{\cdot}_{\cI}$.
Furthermore,
\[
\partial_{b_k}\cdots\partial_{b_1}f_{a,\scI}(b) = \sum_{\pi \in S_k} f_{\scI,\alpha}^{[k]}(\underbrace{a+b,\ldots,a+b}_{\mathsmaller{k+1\,\mathrm{times}}})\sh\big[b_{\pi(1)},\ldots,b_{\pi(k)}\big] 
\]
for all $b,b_1,\ldots,b_k \in \cI_{\sa}$.
\end{theorem}

If $(\cI,\norm{\cdot}_{\cI}) = (\cA,\norm{\cdot})$ and $a=0$ in Theorem \ref{thm.derform}, then $f_{a,\scI} = f_{0,\scA} = f_{\scA} - f(0)$, so the conclusions of Theorem \ref{thm.derform} take the form:
If $f \in C_{\alpha}^k(\R)$, then $f_{\scA} \in C^k(\cA_{\sa};\cA)$, and
\[
\partial_{b_k}\cdots\partial_{b_1}f_{\scA}(a) = \sum_{\pi \in S_k} f_{\scA,\alpha}^{[k]}(\underbrace{a,\ldots,a}_{\mathsmaller{k+1\,\mathrm{times}}})\sh\big[b_{\pi(1)},\ldots,b_{\pi(k)}\big] 
\]
for all $a,b_1,\ldots,b_k \in \cA_{\sa}$.
Thus, Theorem \ref{thm.derform} generalizes Theorem \ref{thm.derformintro}.

After proving Theorem \ref{thm.derform}, we shall go through several corollaries resulting from our calculations of various lexicographic functional calculi.

\begin{lemma}[Perturbation formula]\label{lem.pertform}
If $\alpha$ controls polynomial LFC in $\cI$, $f \in C_{\alpha}^1(\R)$, $a \in \cA_{\sa}$, and $b \in \cI_{\sa}$, then
\[
f(a+b) - f(b) = f_{\scI,\alpha}^{[1]}(a+b,a)\sh b.
\]
In particular, $f_{a,\scI}(b) = f(a+b) - f(a) \in \cI$.
\end{lemma}

\begin{proof}
Let $a \in \cA_{\sa}$ and $b \in \cI_{\sa}$.
By definition of polynomial LFC and the perturbation formula in the holomorphic case (Proposition \ref{prop.pertformholo}), we already know that
\[
p(a+b) - p(a) = p_{\sotimes}^{[1]}(a+b,a)\sh b = p_{\scI,\alpha}^{[1]}(a+b,a)\sh b \qquad (p \in \cP).
\]
For general $f \in C_{\alpha}^1(\R)$, choose a net $(p_j)_{j \in J}$ in $C_{\alpha}^1(\R)$ converging to $f$ in $C_{\alpha}^1(\R)$.
Since $\alpha_{\cdot,1} = u_{\cdot,1}$, $(p_j)_{j \in J}$ converges locally uniformly to $f$.
Thus, $(p_j(a+b))_{j \in J}$ and $(p_j(a))_{j \in J}$ converge, respectively, to $f(a+b)$ and $f(a)$ in $\cA$.
Since $\big(p_j^{[1]}\big)_{j \in J}$ converges to $f^{[1]}$ in $C_{\alpha}(\R^2)$, the continuity of bivariate $\alpha$-LFC in $\cI$ guarantees that
\[
(p_j(a+b) - p_j(a))_{j\in J} = \big(\big(p_j^{[1]}\big)_{\scI,\alpha}(a+b,a)\sh b\big)_{j \in J}
\]
converges to $f_{\scI,\alpha}^{[1]}(a+b,a)\sh b$ in $\cI$.
Since $\iota_{\scI} \colon (\cI,\norm{\cdot}_{\cI}) \to (\cA,\norm{\cdot})$ is continuous, we obtain~that
\begin{align*}
    f(a+b) - f(a) & = \cA\text{-}\lim_{j \in J} \left(p_j(a+b)-p_j(a)\right) = \cI\text{-}\lim_{j \in J} \left(p_j(a+b)-p_j(a)\right) \\
    & = \cI\text{-}\lim_{j \in J}\big(p_j^{[1]}\big)_{\scI,\alpha}(a+b,a)\sh b = f_{\scI,\alpha}^{[1]}(a+b,a)\sh b,
\end{align*}
as desired.
\end{proof}

\begin{proof}[Proof of Theorem \ref{thm.derform}]
By definition of polynomial LFC and the formula for the derivatives of maps induced by the holomorphic functional calculus (Theorem \ref{thm.derivinidealHFC}), we already know the result when $f=p \in \cP$.

Now, let $V$ and $W$ be normed vector spaces.
If $F \colon V \to W$ is a $k$-times Fr\'echet-differentiable function, write $D^kF \colon V \to B_k(V^k;W)$ for the $k^{\text{th}}$ Fr\'echet derivative of $F$;
please see \cite[Ch.\ 1]{HJ2014} for background on Fr\'echet derivatives.
Also, if $T \in B_k(V^k;W)$, write
\[
S(T)[v_1,\ldots,v_k] \coloneqq \sum_{\pi \in S_k} T(v_{\pi(1)},\ldots,v_{\pi(k)}) \qquad (v_1,\ldots,v_k \in V).
\]
Note that $\norm{S(T)}_{B_k(V^k;W)} \leq k!\norm{T}_{V \to W}$.

For a general $f \in C_{\alpha}^k(\R)$, choose a net $(p_j)_{j \in J}$ in $\cP$ converging to $f$ in $C_{\alpha}^k(\R)$, let $r > 0$, and write
\[
R \coloneqq \norm{a} + C_{\cI}r+1 \; \text{ and } \; \cI_{\sa,s} \coloneqq \{b \in \cI_{\sa} : \norm{b}_{\cI} \leq s\} \qquad (s \geq 0).
\]
By Lemma \ref{lem.pertform},
\begin{align*}
    \sup_{b \in \cI_{\sa,r}} \norm{(p_j)_{a,\scI}(b) - f_{a,\scI}(b)}_{\cI} & = \sup_{b \in \cI_{\sa,r}} \norm{\big(p_j^{[1]}\big)_{\scI,\alpha}(a+b,a)\sh b - f_{\scI,\alpha}^{[1]}(a+b,a)\sh b}_{\cI} \\
    & \leq r\sup_{b \in \cI_{\sa,r}} \norm{\big(p_j^{[1]}\big)_{\scI,\alpha}(a+b,a) - f_{\scI,\alpha}^{[1]}(a+b,a)}_{\cI \to \cI} \\
    & \leq r\sup_{\mathbf{a} \in \cA_{\sa,R}^2} \norm{\big(p_j^{[1]}\big)_{\scI,\alpha}(\mathbf{a}) - f_{\scI,\alpha}^{[1]}(\mathbf{a})}_{\cI \to \cI} \\
    & \leq r \,\alpha_{R,2}\big((p_j-f)^{[1]}\big) \xrightarrow{j \in J} 0.
\end{align*}
Next, let $i=1,\ldots,k$.
For each $g \in C_{\alpha}^k(\R)$, write $T_ig \colon \cI_{\sa} \to B_i(\cI_{\sa}^i;\cI)$ for the function
\[
T_ig(b) \coloneqq S\big(g_{\scI,\alpha}^{[i]}\big((a+b)_{(i+1)}\big)\big) \qquad (b \in \cI_{\sa}).
\]
This is the putative $i^{\text{th}}$ Fr\'echet derivative of $g_{a,\scI}$.
By the continuity properties of the ($i+1$)-variate $\alpha$-LFC and the fact that $(\cI,\norm{\cdot}_{\cI}) \hookrightarrow (\cA,\norm{\cdot})$, $T_ig \in \Cbb(\cI_{\sa};B_i(\cI_{\sa}^i;\cI))$.
By the first paragraph of the proof, if $r > 0$, then
\begin{align*}
    \sup_{b \in \cI_{\sa,r}} \big\|D^i(p_j)_{a,\scI}(b) - T_if&(b)\big\|_{B_i(\cI_{\sa}^i;\cI)} = \sup_{b \in \cI_{\sa,r}} \norm{T_ip_j(b) - T_if(b)}_{B_i(\cI_{\sa}^i;\cI)} \\
    & \leq i! \sup_{b \in \cI_{\sa,r}} \norm{\big(p_j^{[i]}\big)_{\scI,\alpha}\big((a+b)_{(i+1)}\big) - f_{\scI,\alpha}^{[i]}\big((a+b)_{(i+1)}\big)}_{B_i(\cI_{\sa}^i;\cI)} \\
    & \leq i! \sup_{\mathbf{a} \in \cA_{\sa,R}^{i+1}} \norm{\big(p_j^{[i]}\big)_{\scI,\alpha}(\mathbf{a}) - f_{\scI,\alpha}^{[i]}(\mathbf{a})}_{B_i(\cI^i;\cI)} \\
    & \leq i! \,\alpha_{R,i+1}\big((p_j-f)^{[i]}\big) \xrightarrow{j \in J} 0.
\end{align*}
In other words, the net $((p_j)_{a,\scI})_{j \in J}$ converges to $f_{a,\scI}$ uniformly on bounded sets, and for all $i=1,\ldots,k$, the net $\big(D^i(p_j)_{a,\scI}\big)_{j \in J}$ converges uniformly on bounded sets to $T_if$.
It then follows from \cite[Thm.\ 1.85]{HJ2014} that $f_{a,\scI} \in C^k(\cI_{\sa};\cI)$ and $D^if_{a,\scI} = T_if$ for all $i=1,\ldots,k$.
\end{proof}

\begin{remark}[Alternate proof]\label{rem.altpertformproof}
I call the method deployed above to prove Theorem \ref{thm.derform} the method of polynomial approximation:
Establish the desired conclusion on polynomials, and extend it by approximation in the appropriate topology to the desired class of functions.
This was actually the first-ever method, used by Daletskii and Krein in \cite{DK1956}, to compute the higher derivatives of maps induced by functional calculus.
In particular, the LFC formalism enabled us to prove a new, extremely general result using effectively classical techniques.
This was also my approach in \cite{Nikitopoulos2023n}.

These days, the more common method of proving such results is that of (higher-order) perturbation formulas, due to its efficacy in situations in which unbounded operators arise;
please see, e.g., \cite{dPS2004,Peller2006,ACDS2009,AP2016,Peller2016,CLMSS2019,LMS2020,LMM2021,Nikitopoulos2023h}.
Since I also used the method of perturbation formulas to study holomorphic functional calculus calculus in the appendix, it is worth noting that the method can be adapted to prove Theorem \ref{thm.derform}.
To explain how requires some notation.
If $S$ is a set, $n \in \N$, $i = 1,\ldots,n+1$, and $s = (s_1,\ldots,s_n) \in S^n$, then $s_{i-} \coloneqq (s_1,\ldots,s_{i-1}) \in S^{i-1}$ and $s_{i+} \coloneqq (s_i,\ldots,s_n) \in S^{n-i+1}$, where $s_{1-}$ and $s_{(n+1)+}$ are both the empty list.
Also, if $\pi \in S_n$, then $s^{\pi} \coloneqq (s_{\pi(1)},\ldots,s_{\pi(n)}) \in S^n$.
In this notation, the higher-order perturbation formulas read:
If $f \in C_{\alpha}^k(\R)$, $a \in \cA_{\sa}$, $\mathbf{a} \in \cA_{\sa}^k$, $c \in \cI_{\sa}$, $b \in \cI^k$, and $i=1,\ldots,k+1$, then
\[
f_{\scI,\alpha}^{[k]}(\mathbf{a}_{i-},a+c,\mathbf{a}_{i+})\sh b - f_{\scI,\alpha}^{[k]}(\mathbf{a}_{i-},a,\mathbf{a}_{i+}) \sh b = f_{\scI,\alpha}^{[k+1]}(\mathbf{a}_{i-},a+c,a,\mathbf{a}_{i+})\sh [b_{i-},c,b_{i+}].
\]
These formulas themselves can be proven, e.g., by the method of polynomial approximation.
Regardless, once they are proven, Theorem \ref{thm.derform} follows via arguments like those in the proofs of Theorem \ref{thm.derivinidealHFC} or \cite[Thm.\ 4.4.6]{Nikitopoulos2023h}.
\end{remark}

\begin{corollary}[Commutative case]\label{cor.commcase}
Suppose $\cA$ is commutative, and let $a \in \cA_{\sa}$.
If $f \in C^k(\R)$, then $f_{a,\scI} \in C^k(\cI_{\sa};\cI)$, and $\partial_{b_k}\cdots\partial_{b_1} f_{a,\scI}(b) = f^{(k)}(a+b)\,b_1\cdots b_k$ for all $b,b_1,\ldots,b_k \in \cI_{\sa}$.
\end{corollary}

\begin{proof}
By Theorems \ref{thm.derform} and \ref{thm.commLFC}, if $f \in C_u^k(\R) = C^k(\R)$, then $f_{a,\scI} \in C^k(\cI_{\sa};\cI)$, and
\begin{align*}
    \partial_{b_k}\cdots\partial_{b_1} f_{a,\scI}(b) & = \sum_{\pi \in S_k} f^{[k]}\big((a+b)_{(k+1)}\big)\,b_{\pi(1)}\cdots b_{\pi(k)} \\
    & = k!\,f^{[k]}\big((a+b)_{(k+1)}\big)\,b_1\cdots b_k.\numberthis\label{eq.almosttherecomm}
\end{align*}
Now, it is a standard functional-calculus exercise to show that if $m \in \N$, $\psi \in C(\R^m)$, and $g(\lambda) \coloneqq \psi(\lambda_{(m)})$ for all $\lambda \in \R$, then $\psi(a_{(m)}) = g(a)$ for all $a \in \cA_{\sa}$.
Since, by \eqref{eq.divdiff},
\[
f^{[k]}\big(\lambda_{(k+1)}\big) = f^{(k)}(\lambda)\,\rho_k(\Delta_k) = \frac{f^{(k)}(\lambda)}{k!} \qquad (\lambda \in \R),
\]
the claimed formula follows at once from \eqref{eq.almosttherecomm}.
\end{proof}

\begin{corollary}[$\cA$-Varopoulos case]\label{cor.AVarop}
Let $a \in \cA_{\sa}$.
If $f \in C_{\beta^{\scA}}^k(\R)$, e.g., $f \in VC^k(\R)$, then $f_{a,\scI} \in C^k(\cI_{\sa};\cI)$, and
\[
\partial_{b_k}\cdots\partial_{b_1} f_{a,\scI}(b) = \sum_{\pi \in S_k} f_{\sotimes}^{[k]}(\underbrace{a+b,\ldots,a+b}_{\mathsmaller{k+1\,\mathrm{times}}})\sh\big[b_{\pi(1)},\ldots,b_{\pi(k)}\big] 
\]
for all $b,b_1,\ldots,b_k \in \cI_{\sa}$.
\end{corollary}

\begin{proof}
Combine Theorems \ref{thm.derform} and \ref{thm.AVaropLFC}.
\end{proof}
\pagebreak

Using results of Peller from \cite{Peller2006} and the theory MOIs from \cite{Nikitopoulos2023m}, I proved in \cite{Nikitopoulos2023n} that if $f \in \dot{B}_1^{k,\infty}(\R)$ or $f \in C_{\loc}^{k,\e}(\R)$ for some $\e > 0$, then $f_{\scA} \in C^k(\cA_{\sa};\cA)$.
In \cite{ACDS2009}, it was proven that if $\cA$ is a von Neumann algebra represented on a separable Hilbert space, $(\cI,\norm{\cdot}_{\cI})$ is a symmetrically normed ideal of $\cA$ with property (F), $a$ is a self-adjoint operator affiliated with $\cA$, and $f \in W_{k+1}(\R)$, then $f_{a,\scI} \in C^k(\cI_{\sa};\cI)$.
In \cite{Nikitopoulos2023h}, I proved that if $\cA$ is a von Neumann algebra, $(\cI,\norm{\cdot}_{\cI})$ is an integral symmetrically normed ideal of $\cA$, $a$ is a self-adjoint operator affiliated with $\cA$, and $f \in \dot{B}_1^{1,\infty}(\R) \cap \dot{B}_1^{k,\infty}(\R)$ is such that $f'$ is bounded, then $f_{a,\scI} \in C^k(\cI_{\sa};\cI)$.
Together with Theorem \ref{thm.BesovHolderVaropoulos}, Corollary \ref{cor.AVarop} massively generalizes all these results on the higher differentiability of $f_{a,\scI}$ when $a$ is a bounded operator \emph{without} the use of MOI theory.
In particular, the regularity $f \in \dot{B}_1^{k,\infty}(\R)$ or $f \in C_{\loc}^{k,\e}(\R)$ is now known to be sufficient to guarantee the $k$-times continuous differentiability of $f_{a,\scI}$ for \emph{all} symmetrically normed ideals $\cI$ of $\cA$, not just $\cA$ itself.

\begin{corollary}[Finite-dimensional case]\label{cor.findimcase}
Suppose $\cA$ is finite dimensional, and let $a \in \cA_{\sa}$.
If $f \in C^k(\R)$, then $f_{a,\scI} \in C^k(\cI_{\sa};\cI)$, and
\[
\partial_{b_k}\cdots\partial_{b_1} f_{a,\scI}(b) = \sum_{\pi \in S_k}\sum_{\blambda \in \sigma(a+b)^{k+1}} f^{[k]}(\blambda)\,P_{\lambda_1}^{a+b}b_{\pi(1)}\cdots P_{\lambda_k}^{a+b}b_{\pi(k)}P_{\lambda_{k+1}}^{a+b}
\]
for all $b,b_1,\ldots,b_k \in \cI_{\sa}$.
\end{corollary}

\begin{proof}
Combine Theorem \ref{thm.derform} and Corollary \ref{cor.findimLFC}.
(Note that $C_{u^{\scA}}^k(\R) = C^k(\R)$.)
\end{proof}

Corollary \ref{cor.findimcase} massively generalizes the matrix case, i.e., the case $\cA = \cI = \MnC$, from \cite{DK1956} as stated in more modern language and notation in \cite{Hiai2010,Nikitopoulos2023n}.

The final corollary uses notation and terminology from the theory of MOIs and their applications;
please see subsection \ref{subsec.MOIfam} and the references therein.

\begin{corollary}\label{cor.ISNIcase}
Suppose $\cM$ is a von Neumann algebra with $\cA$ as a unital $C^*$-subalgebra.
If $f \in NC^k(\R)$, then $f_{\scA} \in C^k(\cA_{\sa};\cA)$, and
\begin{align*}
    \partial_{b_k}\cdots\partial_{b_1} f_{\scA}(a) = \sum_{\pi \in S_k}\underbrace{\int_{\sigma(a)}\cdots\int_{\sigma(a)}}_{\mathsmaller{k+1\,\mathrm{times}}} f^{[k]}(\blambda)\,P^{a}(\d\lambda_1)\,b_{\pi(1)}\cdots P^{a}(\d\lambda_k)\,b_{\pi(k)}\,P^{a}(\d\lambda_{k+1})
\end{align*}
for all $a,b_1,\ldots,b_k \in \cA_{\sa}$.
Now, suppose $\cA = \cM$ and $(\cI,\norm{\cdot}_{\cI})$ is an integral symmetrically normed ideal of $\cA = \cM$, and let $a \in \cA_{\sa}$.
If $f \in NC^k(\R)$, then $f_{a,\scI} \in C^k(\cI_{\sa};\cI)$, and
\begin{align*}
    \partial_{b_k}&\cdots\partial_{b_1} f_{a,\scI}(b) \\
    & = \sum_{\pi \in S_k}\underbrace{\int_{\sigma(a+b)}\cdots\int_{\sigma(a+b)}}_{\mathsmaller{k+1\,\mathrm{times}}} f^{[k]}(\blambda)\,P^{a+b}(\d\lambda_1)\,b_{\pi(1)}\cdots P^{a+b}(\d\lambda_k)\,b_{\pi(k)}\,P^{a+b}(\d\lambda_{k+1})
\end{align*}
for all $b,b_1,\ldots,b_k \in \cI_{\sa}$.
\end{corollary}

\begin{proof}
Combine Theorems \ref{thm.derform} and \ref{thm.nufam}.
(Recall that $C_{\nu}^k(\R) = NC^k(\R)$.)
\end{proof}

The first result in Corollary \ref{cor.ISNIcase} is precisely \cite[Thm.\ 1.2.3]{Nikitopoulos2023n}.
The second recovers and slightly generalizes \cite[Thm.\ 1.2.2]{Nikitopoulos2023h} restricted to the case of bounded operators.
\pagebreak

As we have seen, Theorem \ref{thm.derform} recovers, improves, or generalizes a large number of results from the literature about the higher differentiability of $f_{a,\scI}$.
There is one important result it does not immediately recover.
Let $H$ be a separable complex Hilbert space and $\cS_q(H)$ be the ideal of Schatten $q$-class operators on $H$ endowed with the Schatten $q$-norm $\norm{\cdot}_q$.
In \cite{LMS2020}, it is proven that if $1 < q < \infty$, $a \in B(H)_{\sa}$, and $f \in C^k(\R)$, then $f_{a,\mathsmaller{\cS_q(H)}} \in  C^k(\cS_q(H)_{\sa};\cS_q(H))$.
Theorem \ref{thm.derform} does not obviously recover this result because I am not sure whether a scaling of the uniform family controls polynomial LFC in $\cS_q(H)$.
A weaker fact, however, \emph{is} true and is the key to the quoted result from \cite{LMS2020}.
Indeed, as a consequence of \cite[Thm.\ 5.6]{PST2013}, there exists a constant $c_{k,q} < \infty$ such that for all $p \in \cP$ and $r > 0$,
\[
\sup_{\mathbf{a} \in B(H)_{\sa,r}^{k+1}}\Big\|\ev_{\mathsmaller{\cS_q(H)}}^{k+1}\big(p^{[k]}\big)(\mathbf{a})\Big\|_{B_k(\cS_q(H)^k;\cS_q(H))} \leq c_{k,q}\big\|p^{(k)}\big\|_{\ell^{\infty}([-r,r])} = k!\,c_{k,q}u_{r,k+1}\big(p^{[k]}\big).
\]
In particular, if one were to change the entire formalism of LFC to require only control of $\ev_{\scI}^{k+1}(P)$ for multivariate polynomials $P \in \cP_{k+1}$ of the form $P = p^{[k]}$ for some single-variate polynomial $p \in \cP$, then the formalism would recover the quoted result from \cite{LMS2020}.
It would be interesting to know whether a scaling of the uniform family does, in fact, control polynomial LFC in $\cS_q(H)$ when $1 < q < \infty$.

Finally, it is worth noting that there is related work, e.g., \cite{dPS2004,CCGP2026}, on the (higher) Gateaux differentiability of maps on noncommutative $L^p$ spaces induced by functional calculus.
Though very important, the results in these works are of a fundamentally different flavor than those in the present paper.
For example, Fr\'echet differentiability is simply false in the settings of \cite{dPS2004,CCGP2026}.

\appendix
\section{Holomorphic functional calculus calculus}\label{app.HFC}

The purpose of this appendix is to prove Theorem \ref{thm.derivinidealHFC}, a holomorphic version of the main result of subsection \ref{subsec.derivatives} generalizing Theorems \ref{thm.HFCintro} and \ref{thm.HFCderformintro}.
In doing so, I shall use material from subsections \ref{subsec.HFC} and \ref{subsec.SNI} as well as the theory of strong/Bochner integrals in Fr\'echet spaces (vid.\ \cite[\S1.1]{Nikitopoulos2024}).
Throughout, $\cB$ is a unital Banach algebra with norm $\norm{\cdot}$, $k \in \N_0$, $m \in \N$, $U,U_1,\ldots,U_m \subseteq \C$ are open sets, and $V \coloneqq U_1 \times \cdots \times U_m$.

\subsection{``Baby'' multivariate holomorphic functional calculus}\label{sec.multivarHFC}

This subsection develops a ``baby'' multivariate holomorphic functional calculus sufficient to formulate and prove Theorem \ref{thm.derivinidealHFC} in the next subsection.
A proper treatment of ``adult'' multivariate holomorphic functional calculi would take us needlessly far afield;
I refer the interested reader to \cite{Curto1988} for more information and references.

Here and throughout, if $W \subseteq \C^m$ is open, then $\Hol\left(W\right)$ is given the topology of locally uniform convergence.

\begin{notation}\label{nota.RUHol0}
$\cR_U \subseteq \Hol\left(U\right)$ is the set of rational functions with poles outside of $U$.
Also,
\[
\cR_V \coloneqq \spn\left\{r_1 \otimes \cdots \otimes r_m : r_1 \in \cR_{U_1},\ldots,r_m \in \cR_{U_m}\right\} \subseteq \Hol \left(V\right),
\]
and $\Hol_0\left(V\right)$ is the closure of $\cR_V$ in $\Hol\left(V\right)$.
\end{notation}

By Runge's theorem, $\cR_U$ is dense in $\Hol\left(U\right)$, i.e., $\Hol_0\left(U\right) = \Hol\left(U\right)$.
Consequently,
\[
\spn\{f_1 \otimes \cdots \otimes f_m : f_1 \in \Hol\left(U_1\right),\ldots,f_m \in \Hol\left(U_m\right)\} \subseteq \Hol_0\left(V\right).
\]
If $m \geq 2$, then $\Hol_0\left(V\right) \subsetneq \Hol\left(V\right)$ in general.
This is part of what complicates holomorphic functional calculus in the multivariate case.
To avoid this and other complications, I construct a functional calculus defined only on $\Hol_0\left(V\right)$.

\begin{theorem}[``Baby'' multivariate holomorphic functional calculus]\label{thm.babymultivarHFC}
Suppose
\[
\mathbf{a} = (a_1,\ldots,a_m) \in \cB_{U_1} \times \cdots \times \cB_{U_m} \; \text{ and } \; a_ia_j=a_ja_i \qquad (i,j \in \{1,\ldots,m\}).
\]
There exists a unique continuous, unital algebra homomorphism $H_{\mathbf{a}}^V \colon \Hol_0\left(V\right) \to \cB$ that maps the coordinate function $1^{\otimes(i-1)} \otimes \iota_{\mathsmaller{U_i}} \otimes 1^{\otimes(m-i)}$ to $a_i$ for each $i = 1,\ldots,m$.
The map $H_{\mathbf{a}}^V$ is called the (\textbf{baby}) \textbf{holomorphic functional calculus for $\mathbf{a}$}, and $\varphi(\mathbf{a}) \coloneqq H_{\mathbf{a}}^V(\varphi)$ for all $\varphi \in \Hol_0\left(V\right)$.
\end{theorem}

\begin{proof}
By simple algebraic arguments, if $\Phi,\Psi \colon \Hol_0\left(V\right) \to \cB$ are two unital algebra homomorphisms sending the function $\mathbf{z} \mapsto z_i$ to $a_i$ for each $i  = 1,\ldots,m$, then $\Phi = \Psi$ on $\cR_V$.
Since $\cR_V$ is dense in $\Hol_0\left(V\right)$, if $\Phi$ and $\Psi$ are also continuous, then $\Phi = \Psi$ on $\Hol_0\left(V\right)$.

Now, for each $i = 1,\ldots,m$, let $\Gamma_i$ be a cycle surrounding $\sigma(a_i)$ in $U_i$, and define
\[
H(\varphi) \coloneqq \frac{1}{(2\pi i)^m}\int_{\Gamma_m}\cdots\int_{\Gamma_1}\varphi(\mathbf{z})\,(z_1-a_1)^{-1}\cdots (z_m-a_m)^{-1}\,\d z_1 \cdots \d z_m \in \cB
\]
for all $\varphi \in \Hol_0\left(V\right)$.
Clearly, $H \colon \Hol_0\left(V\right) \to \cB$ is linear.
By the dominated convergence theorem, $H$ is continuous.
By the construction of the single-variate holomorphic functional calculus, if $f_i \in \Hol\left(U_i\right)$ for all $i = 1,\ldots,m$, then
\begin{align*}
    H&(f_1 \otimes \cdots \otimes f_m) \\
    & = \frac{1}{(2\pi i)^m}\int_{\Gamma_m}\cdots\int_{\Gamma_1} f_1(z_1)\cdots f_m(z_m)\,(z_1-a_1)^{-1}\cdots (z_m-a_m)^{-1}\,\d z_1 \cdots \d z_m \\
    & = \left(\frac{1}{2\pi i}\int_{\Gamma_1} f_1(z)\,(z-a_1)^{-1}\,\d z\right)\cdots \left(\frac{1}{2\pi i}\int_{\Gamma_m} f_m(z)\,(z-a_m)^{-1}\,\d z\right) \\
    & = f_1(a_1)\cdots f_m(a_m).\numberthis\label{eq.tensorHFC}
\end{align*}
In particular, $H$ is unital and maps $1^{\otimes (i-1)} \otimes \iota_{\mathsmaller{U_i}} \otimes 1^{\otimes(m-i)}$ to $a_i$ for each $i = 1,\ldots,m$.

Finally, let us show that $H$ is multiplicative.
To this end, recall that if $a \in \cB_{U_1}$, $b \in \cB_{U_2}$, and $ab=ba$, then $f(a)\,g(b) = g(b)\,f(a)$ for all $f \in \Hol\left(U_1\right)$ and $g \in \Hol\left(U_2\right)$.
Consequently, if $f_i,g_i \in \Hol\left(U_i\right)$ for all $i = 1,\ldots,m$, $\varphi \coloneqq f_1 \otimes \cdots \otimes f_m$, and $\psi \coloneqq g_1 \otimes \cdots \otimes g_m$, then
\begin{align*}
    H(\varphi\psi) & = H(f_1g_1 \otimes \cdots \otimes f_mg_m) = (f_1g_1)(a_1)\cdots (f_mg_m)(a_m) \\
    & = f_1(a_1)\,g_1(a_1)\cdots f_m(a_m)\,g_m(a_m) = f_1(a_1)\cdots f_m(a_m)\,g_1(a_1)\cdots g_m(a_m) \\
    & = H(f_1 \otimes \cdots \otimes f_m)\,H(g_1 \otimes \cdots \otimes g_m) = H(\varphi)\,H(\psi)
\end{align*}
by \eqref{eq.tensorHFC} and the properties of the single-variate holomorphic functional calculus.
By the linearity of $H$, it follows that $H$ is multiplicative on the subalgebra $\cR_V \subseteq \Hol_0\left(V\right)$.
Since $\cR_V$ is dense in $\Hol_0\left(V\right)$ and $H$ is continuous, we are done.
\end{proof}

\begin{remark}\label{rem.adultmultivarHFC}
It is actually the case that if
\[
\Phi(\varphi) \coloneqq \frac{1}{(2\pi i)^m}\int_{\Gamma_m}\cdots\int_{\Gamma_1} \varphi(\mathbf{z})\,(z_1-a_1)^{-1}\cdots (z_m-a_m)^{-1}\,\d z_1 \cdots \d z_m
\]
for all $\varphi \in \Hol\left(V\right)$, then $\Phi \colon \Hol\left(V\right) \to \cB$ is a unital, continuous algebra homomorphism sending the coordinate function $\mathbf{z} \mapsto z_i$ to $a_i$ for each $i = 1,\ldots,m$, but it takes slightly more work to prove that $\Phi$ is multiplicative.
More seriously, uniqueness is a delicate issue for functional calculi defined on all of $\Hol\left(V\right)$.
Unlike the single-variate case, the proper formulation of uniqueness results, e.g., \cite[Thm.\ 1]{Putinar1983}, requires the introduction of more refined notions of joint spectrum for $m$-tuples of commuting elements, e.g., the Taylor joint spectrum from \cite{Taylor1970s,Taylor1970f} or the Harte spectrum from \cite{Harte1972a,Harte1972b}.
Once again, I refer the interested reader to \cite{Curto1988} for more information and references.
\end{remark}

Note that by the uniqueness part of Theorem \ref{thm.babymultivarHFC}, if $V_i \subseteq U_i$ is open and $\sigma(a_i) \subseteq V_i$ for each $i=1,\ldots,m$, then
\[
H_{U_1 \times \cdots \times U_m}^{\mathbf{a}}(\varphi) = H_{V_1 \times \cdots \times V_m}^{\mathbf{a}}\big(\varphi|_{V_1 \times \cdots \times V_m}\big) \qquad \big(\varphi \in \Hol_0\left(U_1 \times \cdots \times U_m\right)\big).
\]
Consequently, the absence of $V$ in the notation $\varphi(\mathbf{a}) = H_{\mathbf{a}}^V(\varphi)$ is justified.

Let us wrap up this subsection with an important calculation, that of $f^{[k]}(\mathbf{a}) \in \cB$ for $f \in \Hol\left(U\right)$ and $\mathbf{a} = (a_1,\ldots,a_{k+1}) \in \cB_U^{k+1}$ such that $a_ia_j=a_ja_i$ for all $i,j \in \{1,\ldots,k+1\}$.

\begin{lemma}\label{lem.divdiffinHol0}
If $f \in \Hol\left(U\right)$, then $f^{[k]} \in \Hol_0\big(U^{k+1}\big)$.
\end{lemma}

\begin{proof}[Sketch of proof]
Let $z_0 \in \C \setminus U$, and define $r_{z_0}(\lambda) \coloneqq (z_0-\lambda)^{-1}$ for all $\lambda \in U$.
Of course, $r_{z_0} \in \cR_U$.
By a brief calculation,
\begin{equation}
    r_{z_0}^{[k]}(\blambda) = \frac{1}{(z_0-\lambda_1)\cdots(z_0-\lambda_{k+1})} \qquad \big(\blambda \in U^{k+1}\big).\label{eq.divdiffresolv}
\end{equation}
By combining \eqref{eq.divdiffresolv} with the fact that divided differences of polynomials are polynomials (vid.\ \cite[Ex.\ 2.1.5]{Nikitopoulos2023n}), the product rule for divided differences (vid.\ \cite[Prop.\ 2.1.3(iii)]{Nikitopoulos2023n} for the real-$C^k$ case---the holomorphic case is the same), and the fundamental theorem of algebra, it is not difficult to see that if $r \in \cR_U$, then there exist polynomials $p_1,\ldots,p_{k+1} \in \cP$ with roots outside of $U$ and a multivariate polynomial $P\in \cP_{k+1}$ such that
\[
r^{[k]}(\blambda) = \frac{P(\blambda)}{p_1(\lambda_1) \cdots p_{k+1}(\lambda_{k+1})} \qquad \big(\blambda \in U^{k+1}\big).
\]
In particular, $r^{[k]} \in \cR_{U^{k+1}}$.

Now, write $\cD_k \colon \Hol\left(U\right) \to \Hol\big(U^{k+1}\big)$ for the $k^{\text{th}}$ divided difference map $f \mapsto f^{[k]}$.
By the previous paragraph, $\cD_k\cR_U \subseteq \cR_{U^{k+1}} \subseteq \Hol_0\big(U^{k+1}\big)$.
Also, it is not difficult to see from \eqref{eq.divdiffCIF} that $\cD_k \colon \Hol\left(U\right) \to \Hol\big(U^{k+1}\big)$ is continuous.
Since $\cR_U$ is dense in $\Hol\left(U\right)$ and $\Hol_0\big(U^{k+1}\big)$ is closed in $\Hol\big(U^{k+1}\big)$, we conclude that $\cD_k \Hol\left(U\right) \subseteq \overline{\cD_k\cR_U} \subseteq \Hol_0\big(U^{k+1}\big)$, as desired.
\end{proof}

\begin{theorem}[Calculating $f^{[k]}(\mathbf{a})$]\label{thm.mvarHFCdivdiff}
Let $\mathbf{a} = (a_1,\ldots,a_{k+1}) \in \cB_U^{k+1}$ be such that $a_ia_j = a_ja_i$ for all $i,j \in \{1,\ldots,k+1\}$.
If $f \in \Hol\left(U\right)$ and $\Gamma$ is a cycle surrounding $\sigma(a_1) \cup \cdots \cup \sigma(a_{k+1})$ in $U$, then
\[
f^{[k]}(\mathbf{a}) = H_{\mathbf{a}}^{U^{k+1}}\big(f^{[k]}\big) = \frac{1}{2\pi i}\int_{\Gamma} f(z) \, (z-a_1)^{-1}\cdots (z-a_{k+1})^{-1} \,\d z.
\]
\end{theorem}

\begin{proof}
Let $f \in \Hol\left(U\right)$ and $W \coloneqq \{z \in U \setminus \Gamma^* : W_{\Gamma}(z) = 1\}$.
By \eqref{eq.divdiffCIF},
\begin{equation}
    f^{[k]}\big|_{W^{k+1}} = \frac{1}{2\pi i}\int_{\Gamma} f(z) \, \big((z-\iota_{\mathsmaller{W}})^{-1}\big)^{\otimes(k+1)}\,\d z \in \Hol_0\big(W^{k+1}\big),\label{eq.divdiffholoHol0}
\end{equation}
where $\iota_{\mathsmaller{W}} \colon W \to \C$ is the inclusion and the right-hand side is a Bochner integral in the Fr\'echet space $\Hol_0\big(W^{k+1}\big)$.
Applying the continuous homomorphism $H_{\mathbf{a}}^{W^{k+1}}$ to both sides of \eqref{eq.divdiffholoHol0} and appealing to the observation following Remark \ref{rem.adultmultivarHFC} complete the proof.
\end{proof}

\subsection{Perturbation and derivative formulas}\label{sec.pertHFC}

For the duration of this subsection, let $(\cI,\norm{\cdot}_{\cI})$ be a symmetrically normed ideal of $\cB$ and $a \in \cB_U$.
Also, write
\[
\cI_{U,a} \coloneqq \{b \in \cI : a+b \in \cB_U\},
\]
and observe that $\cI_{U,a}$ is open in the space $(\cI,\norm{\cdot}_{\cI})$ because it is the inverse image of the open set $\cB_U \subseteq \cB$ under the map $\cI \ni b \mapsto a+b \in \cB$, which is continuous because the inclusion $\iota_{\scI} \colon (\cI, \norm{\cdot}_{\cI}) \to (\cB,\norm{\cdot})$ is continuous.

\begin{notation}\label{nota.prepertform}
Recall that $V = U_1 \times \cdots \times U_m$.
\begin{enumerate}[font=\normalfont,label=(\roman*)]
    \item If $\mathbf{a} = (a_1,\ldots,a_m) \in \cB_{U_1} \times \cdots \times \cB_{U_m}$, then
    \[
    \tilde{a}_i \coloneqq 1^{\otimes (i-1)} \otimes a_i \otimes 1^{\otimes (m-i)} \in \cB^{\potimes m} \qquad (i=1,\ldots,m),
    \]
    and for all $\varphi \in \Hol_0\left(V\right)$, define
    \[
    \varphi_{\sotimes}(\mathbf{a}) \coloneqq \varphi(\tilde{a}_1,\ldots,\tilde{a}_m) \in \cB^{\potimes m}
    \]
    via the ``baby'' multivariate holomorphic functional calculus (HFC).\label{item.varphiotimesHFC}
    \item For each $i=1,\ldots,k$, write $\sh_{k,i} \colon \cB^{\potimes (k+1)} \to B\big(\cB;\cB^{\potimes k}\big)$ for the bounded linear map determined by
    \[
    \sh_{k,i}(a_1 \otimes \cdots \otimes a_{k+1})c = a_1\otimes \cdots \otimes a_{i-1} \otimes a_ica_{i+1} \otimes a_{i+2} \otimes \cdots \otimes a_{k+1}
    \]
    for all $a_1,\ldots,a_{k+1},c \in \cB$.
    (Replace the $i^{\text{th}}$ tensor sign with $c$.)
    Also, for $u \in \cB^{\potimes(k+1)}$ and $c \in \cB$, write $u\sh_{k,i}c \coloneqq \sh_{k,i}(u)[c] \in \cB^{\potimes k}$.
    Observe that $\sh_{k,i}$ has operator norm at most one.
\end{enumerate}
\end{notation}

Let $f \in \Hol\left(U\right)$.
The goal of this subsection is to use ``perturbation formulas'' to compute the derivatives of the map
\[
\cI_{U,a} \ni b \mapsto f(a+b) - f(a) \in \cI.
\]
In the present context, perturbation formulas are formulas for $f_{\sotimes}^{[i]}(a_1,\ldots,a_{i+1})$ resembling the recursive definition of $f^{[k+1]}$ in terms of $f^{[k]}$.

\begin{proposition}[Perturbation formulas]\label{prop.pertformholo}
If $f \in \Hol\left(U\right)$, then
\[
f(a) - f(b) = f_{\sotimes}^{[1]}(a,b) \sh[a-b] \qquad (a,b \in \cB_U).\pagebreak
\]
Also, if $k \geq 1$, then
\[
f_{\sotimes}^{[k]}(\mathbf{a}) - f_{\sotimes}^{[k]}(\mathbf{b}) = \sum_{i=1}^{k+1} f_{\sotimes}^{[k+1]}(a_1,\ldots,a_i,b_i,\ldots,b_{k+1})\sh_{k+1,i}[a_i-b_i] \qquad \big(\mathbf{a},\mathbf{b} \in \cB_U^{k+1}\big),
\]
where $\mathbf{a} = (a_1,\ldots,a_{k+1})$ and $\mathbf{b} = (b_1,\ldots,b_{k+1})$, as usual.
\end{proposition}

\begin{proof}
Let $a,b \in \cB_U$ and $\mathbf{a},\mathbf{b} \in \cB_U^{k+1}$.
Also, let $\Gamma$ be a cycle surrounding the compact set
\[
\sigma(a) \cup \sigma(b) \cup \bigcup_{i=1}^{k+1}(\sigma(a_i) \cup \sigma(b_i))
\]
in $U$.
Finally, write
\[
R_z(a) \coloneqq (z-a)^{-1} \qquad ( z \in \rho(a) = \C \setminus \sigma(a)).
\]
By construction of the single-variate holomorphic functional calculus, the resolvent identity, and Theorem \ref{thm.mvarHFCdivdiff},
\begin{align*}
    f(a) - f(b) & = \frac{1}{2\pi i}\int_{\Gamma} f(z) \, (R_z(a) - R_z(b))\,\d z \\
    & = \frac{1}{2\pi i}\int_{\Gamma} f(z) \, R_z(a)(a-b)R_z(b)\,\d z \\
    & = \frac{1}{2\pi i}\int_{\Gamma} f(z) \, (R_z(a) \otimes R_z(b))\sh[a-b]\,\d z \\
    & = \Bigg(\frac{1}{2\pi i}\int_{\Gamma} f(z) \, R_z(a) \otimes R_z(b)\,\d z\Bigg)\sh[a-b] \\
    & = f_{\sotimes}^{[1]}(a,b)\sh[a-b].
\end{align*}
This is the first desired formula.
Next, by Theorem \ref{thm.mvarHFCdivdiff} (twice) and the resolvent identity once again,
{\small
\begin{align*}
    &f_{\sotimes}^{[k]}(\mathbf{a}) - f_{\sotimes}^{[k]}(\mathbf{b}) = \sum_{j=1}^{k+1}\big(f_{\sotimes}^{[k]}(a_1,\ldots,a_j,b_{j+1},\ldots,b_{k+1}) - f_{\sotimes}^{[k]}(a_1,\ldots,a_{j-1},b_j,\ldots,b_{k+1})\big) \\
    & = \frac{1}{2\pi i}\sum_{j=1}^{k+1}\int_{\Gamma} f(z) \, R_z(a_1) \otimes \cdots \otimes R_z(a_{j-1}) \otimes (R_z(a_j) - R_z(b_j)) \otimes R_z(b_{j+1}) \otimes \cdots \otimes R_z(b_{k+1}) \, \d z \\
    & = \frac{1}{2\pi i}\sum_{j=1}^{k+1}\int_{\Gamma} f(z) \, R_z(a_1) \otimes \cdots \otimes R_z(a_{j-1}) \otimes R_z(a_j)(a_j-b_j)R_z(b_j) \otimes R_z(b_{j+1}) \otimes \cdots \otimes R_z(b_{k+1}) \, \d z \\
    & = \frac{1}{2\pi i}\sum_{j=1}^{k+1}\int_{\Gamma} f(z) \, (R_z(a_1) \otimes \cdots \otimes R_z(a_j) \otimes R_z(b_j) \otimes \cdots \otimes R_z(b_{k+1}))\sh_{k+1,j}[a_j-b_j] \, \d z \\
    & = \sum_{j=1}^{k+1}\Bigg(\frac{1}{2\pi i}\int_{\Gamma} f(z) \, R_z(a_1) \otimes \cdots \otimes R_z(a_j) \otimes R_z(b_j) \otimes \cdots \otimes R_z(b_{k+1}) \, \d z\Bigg)\sh_{k+1,j}[a_j-b_j] \\
    & = \sum_{j=1}^{k+1} f_{\sotimes}^{[k+1]}(a_1,\ldots,a_j,b_j,\ldots,b_{k+1})\sh_{k+1,j}[a_j-b_j],
\end{align*}
}as desired.
\end{proof}

\begin{corollary}\label{cor.pertopfuncholowelldef}
If $a,b \in \cB_U$, $f \in \Hol\left(U\right)$, and $a-b \in \cI$, then $f(a) - f(b) \in \cI$. 
\end{corollary}

\begin{proof}
By Proposition \ref{prop.hashonI}, $u \sh c \in \cI$ whenever $u \in \cB \potimes \cB$ and $c \in \cI$, so the conclusion follows from the first formula in Proposition \ref{prop.pertformholo}.
\end{proof}

\begin{proposition}[Continuity of $\varphi_{\sotimes}$]\label{prop.contpertholo}
If $\varphi \in \Hol_0\left(V\right)$, then $\varphi_{\sotimes} \colon \cB_{U_1} \times \cdots \times \cB_{U_m} \to \cB^{\potimes m}$ is continuous.
\end{proposition}

\begin{proof}
Let $\mathbf{a} \in \cB_{U_1} \times \cdots \times \cB_{U_m}$, and define
\[
\e(\mathbf{b}) \coloneqq \max_{i=1,\ldots,m}\norm{a_i-b_i} \qquad \big(\mathbf{b} \in \cB_{U_1} \times \cdots \times \cB_{U_m}\big).
\]
Also, let $\delta > 0$ be such that $K_i \coloneqq \{z \in \C : \operatorname{dist}\left(z,\sigma(a_i)\right) \leq \delta\} \subseteq U_i$ for all $i=1,\ldots,m$.
Observe that if $\e(\mathbf{b})$ is sufficiently small, then $\sigma(b_i) \subseteq K_i$ for all $i=1,\ldots,m$.

We shall make use of the construction of $H_{\mathbf{a}}^V$ from the proof of Theorem \ref{thm.babymultivarHFC}.
Let $i=1,\ldots,m$ and $\Gamma_i$ be a cycle surrounding $K_i$ in $U_i$.
In particular, if $\e(\mathbf{b})$ is sufficiently small, then $\Gamma_i$ surrounds $\sigma(a_i) \cup \sigma(b_i)$ in $U_i$.
In this case,
\[
R_z(b_i) = R_z(a_i)(1-(b_i-a_i)R_z(a_i))^{-1} \qquad (z \in \Gamma_i^*).
\]
As a result, recalling that $\|(1-a)^{-1}\| < (1-\norm{a})^{-1}$ whenever $\norm{a} < 1$ and writing
\[
c_i \coloneqq \sup_{z \in \Gamma_i^*}\norm{R_z(a_i)}  < \infty,
\]
if $\e(\mathbf{b}) < 1/c_i$ and $\e(\mathbf{b})$ is small enough that $\Gamma_i$ surrounds $K_i$ in $U_i$, then
\begin{align*}
    \sup_{z \in \Gamma_i^*}\norm{R_z(b_i)} & = \sup_{z \in \Gamma_i^*} \norm{R_z(a_i)(1-(b_i-a_i)R_z(a_i))^{-1}} \\
    & \leq \sup_{z \in \Gamma_i^*} \frac{\norm{R_z(a_i)}}{1-\norm{(b_i-a_i)R_z(a_i)}} \\
    & \leq \frac{c_i}{1-c_i\norm{a_i-b_i}} \leq \frac{c_i}{1-c_i\e(\mathbf{b})}.
\end{align*}
Also, for all fixed $z_i \in \Gamma_i^*$, $R_{z_i}(b_i) \to R_{z_i}(a_i)$ in $\cB$ as $b_i \to a_i$ by the bound above and the resolvent identity.
Finally, recall that
\begin{align*}
    \varphi_{\sotimes}(\mathbf{b}) & = \frac{1}{(2\pi i)^m}\int_{\Gamma_m}\cdots\int_{\Gamma_1} \varphi(\mathbf{z})\,R_{z_1}\big(\tilde{b}_1\big) \cdots R_{z_m}\big(\tilde{b}_m\big)\,\d z_1 \cdots \d z_m \\
    & = \frac{1}{(2\pi i)^m}\int_{\Gamma_m}\cdots\int_{\Gamma_1} \varphi(\mathbf{z})\,R_{z_1}(b_1) \otimes \cdots \otimes R_{z_m}(b_m)\,\d z_1 \cdots \d z_m.
\end{align*}
Consequently, we conclude from another application of the bound above and the dominated convergence theorem for Bochner integrals that $\varphi_{\sotimes}(\mathbf{b}) \to \varphi_{\sotimes}(\mathbf{a})$ in $\cB^{\potimes m}$ as $\mathbf{b} \to \mathbf{a}$.
\end{proof}

\begin{theorem}[Derivatives of maps induced by HFC]\label{thm.derivinidealHFC}
Let $(\cI,\norm{\cdot}_{\cI})$ be a symmetrically normed ideal of $\cB$ and $a \in \cB_U$.
If $f \in \Hol\left(U\right)$, then the map
\[
\cI_{U,a} \ni b \mapsto f_{a,\scI}(b) \coloneqq f(a+b)-f(a) \in \cI
\]
is holomorphic with respect to $\norm{\cdot}_{\cI}$.
(This map is well defined by Corollary \ref{cor.pertopfuncholowelldef}.)
Furthermore, if $b \in \cI_{U,a}$ and $b_1,\ldots,b_k \in \cI$, then
\begin{align*}
    \partial_{b_k}\cdots\partial_{b_1}f_{a,\scI}(b) & = \sum_{\pi \in S_k} f_{\sotimes}^{[k]}(\underbrace{a+b,\ldots,a+b}_{\mathsmaller{k+1\,\mathrm{times}}})\sh \big[b_{\pi(1)},\ldots,b_{\pi(k)}\big] \\
    & = \frac{1}{2\pi i}\sum_{\pi \in S_k} \int_{\Gamma} f(z)\,(z-a-b)^{-1}b_{\pi(1)}\cdots (z-a-b)^{-1}b_{\pi(k)}(z-a-b)^{-1}\,\d z,
\end{align*}
where $\Gamma$ is any cycle surrounding $\sigma(a+b)$ in $U$.
\end{theorem}

\begin{proof}
Let $b \in \cI_{U,a}$, and let $h \in \cI$ be such that $b+h \in \cI_{U,a}$.
The proof of the derivative formula proceeds by induction on $k$.
For the base case, note that
\begin{align*}
    \e(h) & \coloneqq \frac{1}{\norm{h}_{\cI}}\big\|f_{a,\scI}(b+h) - f_{a,\scI}(b) - f_{\sotimes}^{[1]}(a+b,a+b)\sh h\big\|_{\cI} \\
    & = \frac{1}{\norm{h}_{\cI}}\norm{f(a+b+h) - f(a+b) - f_{\sotimes}^{[1]}(a+b,a+b)\sh h}_{\cI} \\
    & = \frac{1}{\norm{h}_{\cI}}\norm{f_{\sotimes}^{[1]}(a+b+h,a+b)\sh h - f_{\sotimes}^{[1]}(a+b,a+b)\sh h}_{\cI} \\
    & \leq \norm{f_{\sotimes}^{[1]}(a+b+h,a+b) - f_{\sotimes}^{[1]}(a+b,a+b)}_{\cB \potimes \cB} \xrightarrow{\norm{h}_{\cI} \to 0} 0
\end{align*}
by Propositions \ref{prop.pertformholo}, \ref{prop.hashonI}, and \ref{prop.contpertholo} and the fact that $\iota_{\scI} \colon (\cI,\norm{\cdot}_{\cI}) \to (\cB,\norm{\cdot})$ is continuous.
Now, assume the claimed derivative formula for the $k^{\text{th}}$ derivative.
Also, if $S$ is a set and $s \in S$, write $s_{(m)}$ for the member of $S^m$ with all components equal to $s$.
If $b_1,\ldots,b_k \in \cI$ and $b_{k+1} \coloneqq h$, then
{\small
\begin{align*}
    \e&(b_1,\ldots,b_{k+1}) \coloneqq \frac{1}{\norm{h}_{\cI}}\Bigg\|\partial_{b_k}\cdots\partial_{b_1}f_{a,\scI}(b+h) - \partial_{b_k}\cdots\partial_{b_1}f_{a,\scI}(b) \\
    & \hspace{45mm} - \sum_{\sigma \in S_{k+1}} f_{\sotimes}^{[k+1]}\big((a+b)_{(k+2)}\big)\sh\big[b_{\sigma(1)},\ldots,b_{\sigma(k+1)}\big]\Bigg\|_{\cI} \\
    & = \frac{1}{\norm{h}_{\cI}}\Bigg\|\sum_{\pi \in S_k} \big(f_{\sotimes}^{[k]}\big((a+b+h)_{(k+1)}\big) - f_{\sotimes}^{[k]}\big((a+b)_{(k+1)}\big)\big)\sh \big[b_{\pi(1)},\ldots,b_{\pi(k)}\big] \\
    & \hspace{35mm} - \sum_{\sigma \in S_{k+1}} f_{\sotimes}^{[k+1]}\big((a+b)_{(k+2)}\big)\sh\big[b_{\sigma(1)},\ldots,b_{\sigma(k+1)}\big]\Bigg\|_{\cI} \\
    & = \frac{1}{\norm{h}_{\cI}}\Bigg\|\sum_{\pi \in S_k}\sum_{i=1}^{k+1} f_{\sotimes}^{[k+1]}\big((a+b+h)_{(i)}, (a+b)_{(k+2-i)}\big)\sh\big[b_{\pi(1)},\ldots,b_{\pi(i-1)},h,b_{\pi(i)},\ldots,b_{\pi(k)}\big] \\
    & \hspace{35mm} - \sum_{\pi \in S_k}\sum_{i=1}^{k+1} f_{\sotimes}^{[k+1]}\big((a+b)_{(k+2)}\big)\sh\big[b_{\pi(1)},\ldots,b_{\pi(i-1)},h,b_{\pi(i)},\ldots,b_{\pi(k)}\big]\Bigg\|_{\cI} \\
    & \leq k! \, C_{\cI}^k\prod_{i=1}^k\norm{b_i}_{\cI} \sum_{i=1}^{k+1} \norm{f_{\sotimes}^{[k+1]}\big((a+b+h)_{(i)}, (a+b)_{(k+2-i)}\big) - f_{\sotimes}^{[k+1]}\big((a+b)_{(k+2)}\big)}_{\cB^{\potimes(k+2)}}
\end{align*}
}by the induction hypothesis and Propositions \ref{prop.pertformholo} and \ref{prop.hashonI}.
Writing $D^k$ for the $k^{\text{th}}$ Fr\'echet derivative---please consult \cite[Ch.\ 1]{HJ2014} for background on Fr\'echet derivatives---and
\[
F(a)[b_1,\ldots,b_{k+1}] \coloneqq \sum_{\sigma \in S_{k+1}} f_{\sotimes}^{[k+1]}\big(a_{(k+2)}\big)\sh\big[b_{\sigma(1)},\ldots,b_{\sigma(k+1)}\big] \quad (a \in \cB_U, \; b_1,\ldots,b_{k+1} \in \cI),
\]
we then conclude from Proposition \ref{prop.contpertholo} that
{\small
\begin{align*}
    &\frac{1}{\norm{h}_{\cI}} \Big\|D^kf_{a,\scI}(b+h) - D^kf_{a,\scI}(b) - F(a+b)[h,\cdot]\Big\|_{B_k(\cI^k;\cI)} \\
    & \leq k!\,C_{\cI}^k \sum_{i=1}^{k+1} \big\|f_{\sotimes}^{[k+1]}\big((a+b+h)_{(i)}, (a+b)_{(k+2-i)}\big) - f_{\sotimes}^{[k+1]}\big((a+b)_{(k+2)}\big)\big\|_{\cB^{\potimes(k+2)}} \xrightarrow{\norm{h}_{\cI} \to 0} 0.
\end{align*}
}This completes the proof.
\end{proof}

If $(\cI,\norm{\cdot}_{\cI}) = (\cB,\norm{\cdot})$ in Theorem \ref{thm.derivinidealHFC}, then $f_{a,\scI} = f_{a,\scB} = f_{\scB}(a+\cdot) - f(a)$, so using Theorem \ref{thm.derivinidealHFC} to differentiate at $b=0$ repeatedly yields the conclusion:
If $f \in \Hol\left(U\right)$, then $f_{\scB} \in \Hol\left(\cB_U;\cB\right)$, and
\[
\partial_{b_k}\cdots\partial_{b_1}f_{\scB}(a) = \sum_{\pi \in S_k} f_{\sotimes}^{[k]}(\underbrace{a,\ldots,a}_{\mathsmaller{k+1\,\mathrm{times}}})\sh\big[b_{\pi(1)},\ldots,b_{\pi(k)}\big] 
\]
for all $a \in \cB_U$ and $b_1,\ldots,b_k \in \cB$.
Thus, Theorems \ref{thm.mvarHFCdivdiff} and \ref{thm.derivinidealHFC} together generalize Theorems \ref{thm.HFCintro} and \ref{thm.HFCderformintro} from the introduction.

\phantomsection
\addcontentsline{toc}{section}{Acknowledgments}
\begin{acknowledgments}
I acknowledge the use of Gemini for help with proofreading.

I am indebted to Bruce Driver, Todd Kemp, and Michael Hartz for inspiring discussions and guidance.
Special thanks go to Bruce Driver for discussions of ``holomorphic functional calculus calculus'' that led to the development of the results in the appendix and to Michael Hartz for helping me understand the (``adult'') multivariate holomorphic functional calculus and bringing \cite{Curto1988} to my attention.
I am also grateful to Roland Speicher and the Universit\"at des Saarlandes for their financial support and hospitality during my 2023 visit that, among other things, facilitated the aforementioned discussions with Michael Hartz.
Finally, I greatly appreciate Ian Charlesworth's encouragement to lean into the ``functional calculus calculus'' play on words.
\end{acknowledgments}

\phantomsection
\small
\addcontentsline{toc}{section}{References}
\bibliographystyle{amsplain}
\bibliography{LFC.bib}
\end{document}